\documentclass[12pt]{amsart}

\usepackage{array,amssymb,amsxtra,latexsym,graphics,psfrag,epsfig,ifthen}
\usepackage[all]{xy}
\SelectTips {cm}{}
\usepackage{bbm,bm}
\usepackage{rotating}
\usepackage[bookmarks=false,colorlinks,linkcolor=blue,citecolor=green]{hyperref}
\usepackage{xcolor}

\usepackage[newitem,newenum,neverdecrease]{paralist}
\makeatletter
\if@plflushright
  
\else
  
\fi
\renewcommand{\@asparaenum@}{%
  \expandafter\list\csname label\@enumctr\endcsname{%
    \usecounter{\@enumctr}%
    \labelwidth\z@
    \labelsep.5em
    \leftmargin\z@
    \parsep\parskip
    \itemsep\z@
    \topsep\z@
    \partopsep\parskip
    \itemindent\parindent
    \advance\itemindent\labelsep
    \def\makelabel##1{\upshape ##1}}}
\makeatother
\usepackage{psfrag}
\usepackage{longtable}
\theoremstyle{plain}
\newtheorem{theorem}{Theorem}[section]
\newtheorem{proposition}[theorem]{Proposition}
\newtheorem{lemma}[theorem]{Lemma}
\newtheorem{corollary}[theorem]{Corollary}

\newtheorem*{theorem*}{Theorem}

\theoremstyle{definition}
\newtheorem{definition}[theorem]{Definition}

\theoremstyle{remark}
\newtheorem{example}[theorem]{Example}
\newtheorem{remark}[theorem]{Remark}
\newtheorem*{aproof}{Alternative proof of Proposition~\ref{p:equivalence-coalgebra}}

\numberwithin{equation}{section}
\numberwithin{figure}{section}
\numberwithin{table}{section}

\def\field{\Bbbk}

\newcommand{\SEarrow}{\begin{turn}{-45}$\Rightarrow$\end{turn}} 

\let\map=\xrightarrow
\newcommand{\bdot}{\bm\cdot}
\newcommand{\unit}{I}
\newcommand{\Eq}{\mathrm{Eq}}
\DeclareMathOperator{\End}{End}
\DeclareMathOperator{\Hom}{Hom}
\DeclareMathOperator{\Aut}{Aut}
\DeclareMathOperator{\GL}{GL}
\DeclareMathOperator{\Pic}{Pic}
\DeclareMathOperator{\apode}{\textsc{s}}

\newcommand{\resH}{\Hc_{\sigma}}
\newcommand{\indK}{\Kc^{\tau}}
\newcommand{\tgamma}{\overline{\gamma}}
\newcommand{\cgamma}{c}
\newcommand{\cdelta}{d}
\newcommand{\ctau}{b}
\newcommand{\trho}{\overline{\rho}}
\newcommand{\id}{\mathrm{id}}
\newcommand{\action}{\alpha}

\newcommand{\contra}[1]{#1^{\vee}} 

\newcommand{\Fc}{\mathcal{F}}

\newcommand{\Gc}{\mathcal{G}}

\newcommand{\Hc}{\mathcal{H}}
\newcommand{\Kc}{\mathcal{K}}
\newcommand{\Rc}{\mathcal{R}}
\newcommand{\Tc}{\mathcal{T}}

\newcommand{\Sc}{\mathcal{S}}

\newcommand{\Cc}{\mathcal{C}}
\newcommand{\cCc}{\contra{\Cc}}

\newcommand{\tC}{\widetilde{C}}
\newcommand{\hC}{\widehat{C}}
\newcommand{\hD}{\widehat{D}}
\newcommand{\hH}{\widehat{H}}
\newcommand{\hZ}{\widehat{Z}}
\newcommand{\Dc}{\mathcal{D}}
\newcommand{\Ic}{\mathcal{I}}
\newcommand{\Ec}{\mathcal{E}}

\newcommand{\Ac}{\mathcal{A}}
\newcommand{\Bc}{\mathcal{B}}
\newcommand{\Uc}{\mathcal{U}}

\newcommand{\Qc}{\mathcal{Q}}
\newcommand{\QcC}{\Qc_{\Cc}}

\newcommand{\Hb}{\mathbb{H}}
\newcommand{\Gb}{\mathbb{G}}
\newcommand{\Ab}{\mathbb{A}}

\newcommand{\HTS}{\mathsf{Hopf}(\Tc;\Sc,Z)}
\newcommand{\HHA}{\mathsf{Hopf}(H;A,Z)}
\newcommand{\HHAH}{\mathsf{Hopf}(H;A,H)}
\newcommand{\HAN}{\mathsf{Hopf}(A,H)\downarrow N}
\newcommand{\BBC}{\mathsf{Bimod}(B,C)}
\newcommand{\Com}{\mathsf{Com}}
\newcommand{\AH}{A^{\mathrm{co} H}}
\newcommand{\MH}{M^{\mathrm{co} H}}

\newcommand{\Cs}{\mathsf{C}}
\newcommand{\Ds}{\mathsf{D}}
\newcommand{\Es}{\mathsf{E}}

\newcommand{\Hs}{\mathsf{H}}
\newcommand{\Ms}{\mathsf{M}}
\newcommand{\Ns}{\mathsf{N}}
\newcommand{\Ps}{\mathsf{P}}
\newcommand{\Zs}{\mathsf{Z}}

\newcommand{\qand}{\quad\text{and}\quad}
\newcommand{\qqand}{\qquad\text{and}\qquad}

\begin{document}

\title[Generalized Hopf modules]{Generalized {H}opf Modules for bimonads}

\author[M.~Aguiar]{Marcelo Aguiar}
\address{Department of Mathematics\\
Texas A\&M University\\
College Station, TX 77843}
\email{maguiar@math.tamu.edu}
\urladdr{http://www.math.tamu.edu/~maguiar}
\thanks{Aguiar supported in part by NSF grant DMS-1001935.}

\author[S.~U.~Chase]{Stephen U. Chase}
\address{Department of Mathematics\\
Cornell University\\
Ithaca, NY 14850}
\email{chase@math.cornell.edu}



\date{\today}

\begin{abstract}
Brugui\`eres, Lack and Virelizier have recently obtained a vast generalization
of Sweedler's \emph{Fundamental Theorem of Hopf modules}, in which the role
of the Hopf algebra is played by a bimonad. We present an extension of this
result which involves, in addition to the bimonad, a comodule-monad and
a algebra-comonoid over it. As an application we obtain a generalization of another classical theorem from the Hopf algebra literature, due to Schneider,
which itself is an extension of Sweedler's result (to the setting of Hopf Galois extensions).
\end{abstract}

\maketitle

\section*{Introduction}

A Hopf module over a Hopf $\field$-algebra $H$ (with $\field$ a commutative ring) is a $\field$-module which is both an $H$-module and an $H$-comodule and satisfies a condition linking the two structures.Ê The \emph{Fundamental Theorem of Hopf Modules}, originating in the celebrated 1967 theorem of Sweedler~\cite[Theorem ~4.1.1]{Swe:69}, has over the years inspired many variations and generalizations.Ê 
A recent and especially intriguing contribution to this line of inquiry is the paper of Brugui{\`e}res, Lack and Virelizier~\cite{BLV:2011}, which (among other results)
presents a version of Sweedler's theorem in which the role of the Hopf algebra is played by a bimonad $\Tc$ on a monoidal category, a concept introduced by Moerdijk in~\cite{Moe:2002}.Ê In this paper we extend their theorem in two directions.

First, we obtain a bimonad version of the Fundamental Theorem for the so-called relative Hopf modules of the Hopf Galois theory.Ê This theorem, proved by Schneider~\cite[Theorem~I]{Sch:1990} and often called \emph{Schneider's Structure Theorem}~\cite[Section 8.5, especially Theorem 5.6]{Mon:1993}, has an even longer history than Sweedler's, since it is only a slight exaggeration to say that its origin lies in the 
\emph{non-commutative Hilbert's Theorem 90} of classical Galois theory.Ê In our version of the theorem, the role of the Hopf Galois extension is played by a monad $\Sc$ which is, in a sense we make precise, a comodule over the given bimonad $\Tc$.Ê 

The second direction in which we extend the Fundamental Theorem of Brugui{\`e}res, Lack and Virelizer~\cite[Theorem~6.11]{BLV:2011} involves the introduction of a comonoid $C$ into the picture.Ê The Hopf modules with which we deal are then not only algebras over the monad $\Sc$ but also comodules over a comonoid constructed from $C$ and $\Tc$, satisfying an axiom which means essentially that each structure preserves the other.Ê Our structure theorems then relate the category of such modules to the category of $C$-comodules.

Our primary motivation for pursuing these questions - and, in particular, for the use of the comonoid $C$ above - arose from our project to obtain a structure theorem for Hopf modules over a bimonoid in a $2$-monoidal category.Ê Such a category is a generalization of a braided monoidal category; the notion is discussed in~\cite{AguMah:2010}.
Unlike an ordinary monoidal category, a $2$-monoidal category has two distinct units (one for each of the two monoidal operations), and it turns out that, even to obtain a version of Sweedler's original theorem in this context, it is necessary to relate the relevant Hopf modules to comodules over one of the units.Ê 
We plan to present this application of our fundamental theorem in a sequel to this paper. 
It is this application that demands the comonoid $C$ in our version of the Fundamental Theorem.

In Sections~\ref{s:comparing} and~\ref{s:conjugating} we present a retrospective of monadicity theory for comonads.Ê In Section~\ref{s:conjugating} we pile another brick atop the edifice of that theory with our introduction of the conjugate of a comonad, which is a second comonad that is constructed from the original one and a pair of adjoint functors relating the underlying category to another.Ê We prove a comonadicity theorem for the conjugate comonad that is both a generalization and a corollary of the dual of Beck's \emph{Monadicity Theorem}~\cite[Theorem~VI.7.1, Exercise~VI.6]{Mac:1998}.
 We derive a universal property of the conjugate comonad, and at the end of the section use it to obtain necessary and sufficient conditions for a colax morphism of comonads with a right adjoint to yield an equivalence of their corresponding categories of coalgebras.

In Section~\ref{s:bimonad} we introduce comodule-monads over a bimonad, as well as the more general notion of a comodule over a comonoidal adjunction.Ê Sections~\ref{s:bimonad} and~\ref{s:comod} consist mostly of basic properties of and technical results about these concepts.Ê In Section~\ref{s:bimonad} we also define the Hopf modules with which we deal in the rest of the paper.Ê At the end of Section~\ref{s:comod} we discuss analogues, in our setting, of the \emph{Hopf} and \emph{Galois operators} of~\cite[Sections~2.6 and~2.8]{BLV:2011}.Ê The latter of these, when invertible, generalizes one of the familiar criteria for a Hopf Galois extension~\cite[Definition~8.1.1]{Mon:1993}.

In Section~\ref{s:fundamental} we prove our main theorems for comodule-monads over bimonads, which we derive from more general theorems for comodules over comonoidal adjunctions.Ê In Section~\ref{s:schneider} we apply these theorems to obtain a generalization of the above-mentioned Structure Theorem of Schneider for Hopf modules over a Hopf Galois extension; our version, as with the main theorems of Section~\ref{s:fundamental}, incorporates a coalgebra $C$ into the context.Ê At the end of the section we provide some historical remarks, and a discussion of the connection of these results with Hilbert's Theorem 90.

Any treatment of the structure theorem for relative Hopf modules involves, explicitly or implicitly, the juggling of three distinct sets of phenomena: The category equivalence of the theorem, the invertibility of the Galois operator (the Galois condition), and, finally, existence of certain equalizers and their preservation and reflection by certain functors (faithful flatness conditions).Ê In our categorical setting, managing the relationships among these various conditions leads to some complexity in the statements of our theorems.Ê However, the reader who wishes to gain an overview of the substance of our work, without digesting the proofs, can steer the following path through the paper: Definition~\ref{d:conjugate-comonad}, Theorem~\ref{t:descent}, and Corollary~\ref{c:descent}, followed by Proposition~\ref{p:universal} and Theorem~\ref{t:equivalence-coalgebra-conv}; then Definitions~\ref{d:comodule}, \ref{d:comod-monad}, \ref{d:hopf-mod}, and~\ref{d:galois-map}; and, finally, Theorems~\ref{t:hopf-conv} and~\ref{t:schneider}.

\section{Comonads and their coalgebras}\label{s:comparing}

The material in this section is for the most part standard.
We recall basic notions pertaining to coalgebras over a comonad
and set up the relevant notation in Sections~\ref{ss:coalgebra}--\ref{ss:beck}.
We then consider two comonads linked by a morphism (lax or colax) and
discuss the corresponding functors on the categories of coalgebras
(Sections~\ref{ss:cores} and~\ref{ss:coind}). Mates and
the Adjoint Lifting Theorem are reviewed in Sections~\ref{ss:mates} and~\ref{ss:cores-coind}, in a manner convenient to our purposes. A relationship between conservative
and faithful functors is discussed in Section~\ref{ss:con-fai}.

\subsection{Coalgebras over a comonad}\label{ss:coalgebra}

Let $\Ns$ be a category and
\[
\Dc:\Ns\to\Ns
\]
a \emph{comonad}. We use $(\delta,\epsilon)$ to denote its structure. They are natural transformations
\[
\delta: \Dc\to\Dc^2
\qand
\epsilon:\Dc\to\Ic
\]
which are coassociative and counital, called the \emph{comultiplication} and \emph{counit} of $\Dc$, respectively. Throughout, $\Ic$ denotes the identity functor of the appropriate category.

We let $\Ns_{\Dc}$ denote the category of \emph{$\Dc$-coalgebras} in $\Ns$~\cite[Section~VI.2]{Mac:1998}. The objects are pairs $(N, d)$ where $N$ is an object of $\Dc$ and $d : N\to\Dc(N)$ is a morphism that is coassociative and counital.

For any $\Dc$-coalgebra $(N,d)$, the parallel pair of arrows
\begin{equation}\label{e:pair}
\xymatrix@C+5pt{
\Dc(N) \ar@<0.5ex>[r]^-{\delta_N} \ar@<-0.5ex>[r]_-{\Dc(d)} & \Dc^2(N)
}
\end{equation}
is \emph{coreflexive}. Indeed, the map 
\[
 \xymatrix@C+5pt{
\Dc(N)  & \Dc^2(N) \ar[l]_-{\Dc(\epsilon_N)}
}
\]
is a common retraction: $\Dc(\epsilon_N)\delta_N=\id_{\Dc(N)}=\Dc(\epsilon_N)\Dc(d)$.
Moreover, the diagram
\begin{equation}\label{e:splitfork}
\xymatrix@C+5pt{
N \ar[r]^-{d} & \Dc(N) \ar@<0.5ex>[r]^-{\delta_N} \ar@<-0.5ex>[r]_-{\Dc(d)} & \Dc^2(N)
}
\end{equation}
in $\Ns_{\Dc}$ is a \emph{split cofork} in $\Ns$~\cite[Section~VI.7, pp.~148--149]{Mac:1998}.
The splitting data is
\[
 \xymatrix@C+5pt{
N & \Dc(N) \ar[l]_-{\epsilon_N} & \Dc^2(N) \ar[l]_-{\epsilon_{\Dc(N)}}.
}
\]
Explicitly, we have
\[
\delta_N d=\Dc(d)d, \quad 
\epsilon_N d=\id_N, \quad \epsilon_{\Dc(N)}\delta_N=\id_{\Dc(N)}, \qand
d\epsilon_N = \epsilon_{\Dc(N)}\Dc(d). 
\]

\subsection{The adjunction of a comonad}\label{ss:adj-comonad}

The adjunction associated to the comonad $\Dc$~\cite[Theorem~VI.2.1]{Mac:1998} 
\[
\xymatrix{
\Ns_{\Dc} \ar@/^/[r]^{\Uc_\Dc} & \Ns  \ar@/^/[l]^{\Fc_\Dc}
}
\]
consists of  the \emph{forgetful} functor $\Uc_\Dc$ (left adjoint)
\[
\Uc_\Dc(N,d) := N
\]
and the \emph{cofree coalgebra} functor $\Fc_\Dc$ (right adjoint)
\[
\Fc_\Dc(N) := \bigl(\Dc(N),\delta_N)\bigr).
\]
The unit and counit of the adjunction are
\[
\eta_{(N,d)}:  (N,d)  \map{d}  \bigl(\Dc(N),\delta_N\bigr) =  \Fc_\Dc\Uc_\Dc(N,d)
\qand
\xi_{N}: \Uc_\Dc\Fc_\Dc(N) = \Dc(N)  \map{\epsilon_N} N
\]
respectively.
Adjunctions of this form are called \emph{comonadic}.

A functor is called \emph{conservative} if it reflects isomorphisms.
The forgetful functor $\Uc_\Dc$ is conservative: if $f$
is a morphism of $\Dc$-coalgebras and $f$ is invertible in $\Ns$,
then $f^{-1}$ is a morphism of $\Dc$-coalgebras.

Another fundamental property of the forgetful functor is the following.

\begin{lemma}\label{l:creation}
The forgetful functor $\Uc_{\Dc}:\Ns_{\Dc}\to \Ns$ creates all 
limits that exist in its codomain and that are preserved by both $\Dc$ and $\Dc^2$,
and creates all colimits that exist in its codomain.
\end{lemma}
\begin{proof}
This is contained in~\cite[Propositions~4.3.1 and~4.3.2]{Bor:1994ii}.
For the first statement, see also~\cite[Theorem~V.1.5]{MS:2004}.
\end{proof}

Above, limit creation is understood in the strict sense of~\cite[Section~V.1]{Mac:1998}.
We will mainly employ creation of equalizers. Let us make this notion explicit.

Suppose 
\[
\xymatrix@C+5pt{
 (N,d) \ar@<0.5ex>[r]^-{f} \ar@<-0.5ex>[r]_-{g} & (N',d')
}
\]
is a a parallel pair of morphisms of $\Dc$-coalgebras for which an equalizer
\[
\xymatrix@C+5pt{
 E \ar[r]^{e} & N \ar@<0.5ex>[r]^-{f} \ar@<-0.5ex>[r]_-{g} & N'
}
\]
exists in $\Ns$. Suppose also that
\[
\xymatrix@C+5pt{
 \Dc(E) \ar[r]^{\Dc(e)} & \Dc(N) \ar@<0.5ex>[r]^-{\Dc(f)} \ar@<-0.5ex>[r]_-{\Dc(g)} & \Dc(N')
}
\qand
\xymatrix@C+5pt{
 \Dc^2(E) \ar[r]^{\Dc^2(e)} & \Dc^2(N) \ar@<0.5ex>[r]^-{\Dc^2(f)} \ar@<-0.5ex>[r]_-{\Dc^2(g)} & \Dc^2(N')
}
\]
are equalizers in $\Ns$.
Then $E$ has a unique $\Dc$-coalgebra structure $d''$ for which
\[
\xymatrix@C+5pt{
 (E,d'') \ar[r]^{e} & (N,d) \ar@<0.5ex>[r]^-{f} \ar@<-0.5ex>[r]_-{g} & (N',d')
}
\]
is an equalizer in $\Ns_{\Dc}$. In particular, the original pair has an equalizer in $\Ns_{\Dc}$
and this is preserved by $\Uc_{\Dc}$.

\subsection{The comparison functor}\label{ss:comparison}

Suppose that an adjunction $(\Hc,\Kc)$ is given
\[
\xymatrix{
\Ms \ar@/^/[r]^{\Hc} & \Ns \ar@/^/[l]^{\Kc} 
}
\]
with structure maps (unit and counit)
\[
\eta:\Ic\to\Kc\Hc
\qand
\xi:\Hc\Kc\to\Ic.
\]
The functor
\[
\Hc\Kc:\Ns\to\Ns
\]
is then a comonad on $\Ns$~\cite[Section~VI.1]{Mac:1998}. 
The comultiplication is
\begin{equation}\label{e:com-adj}
 \Hc\Kc = \Hc\Ic\Kc \map{\Hc\eta\Kc} \Hc\Kc\Hc\Kc
\end{equation}
  and the counit is
\begin{equation}\label{e:cou-adj}
\Hc\Kc\map{\xi}\Ic.
\end{equation}

The \emph{comparison} functor
\begin{equation}\label{e:comp-fun}
\Qc: \Ms \to \Ns_{\Hc\Kc}
\end{equation}
is the unique one such that
\begin{equation*}
\begin{gathered}
\xymatrix@R-10pt@C+5pt{
\Ms \ar[rd]^-{\Hc}  \ar[dd]_{\Qc}  &  \\
& \Ns \\
 \Ns_{\Hc\Kc} \ar[ru]_{\Uc_{\Hc\Kc}}
}
\end{gathered}
\qand
\begin{gathered}
\xymatrix@R-10pt@C+5pt{
 \Ms \ar[dd]_{\Qc}  & \\
 & \Ns \ar[lu]_-{\Kc} \ar[ld]^{\Fc_{\Hc\Kc}} \\
 \Ns_{\Hc\Kc}   &
}
\end{gathered}
\end{equation*}
commute~\cite[Theorem~VI.3.1]{Mac:1998}. Explicitly, 
\begin{equation}\label{e:comp-def}
\Qc(M)=\bigl(\Hc(M),\Hc(\eta_M)\bigr).
\end{equation}

\subsection{Beck's Comonadicity Theorem}\label{ss:beck}

We recall this fundamental result of Beck which describes precisely when
the comparison functor  $\Qc$ of Section~\ref{ss:comparison} is an equivalence.

A pair 
of parallel arrows in $\Ms$ is said to be \emph{$\Hc$-split} if 
its image under $\Hc$ is part of a split cofork.

Consider the following hypotheses.
\begin{align}
\label{e:coreflexive}
& \text{Any pair which is  coreflexive and $\Hc$-split has an equalizer in the category $\Ms$.}\\
\label{e:Hpreserve}
& \text{The functor $\Hc$ preserves the equalizer of any such pair.}\\
\label{e:Hconserve}
& \text{The functor $\Hc$ is conservative.}
\end{align}

\begin{theorem}[Beck's Comonadicity Theorem]\label{t:beck}
The functor $\Qc:\Ms \to \Ns_{\Hc\Kc}$ 
is an equivalence if and only if hypotheses~\eqref{e:coreflexive} through~\eqref{e:Hconserve} hold.
\end{theorem}
\begin{proof} This is the dual of Beck's Monadicity 
Theorem, in the form given in~\cite[Theorem~3.3.10]{BW:2005}
and~\cite[Theorem~2.1]{BLV:2011}.
A closely related version is given in~\cite[Theorem~4.4.4]{Bor:1994ii} and in~\cite[Exercise~VI.6]{Mac:1998}.
\end{proof}

We expand on this result in Section~\ref{ss:descent}. We mention in passing that
in the special case when the adjunction $(\Hc,\Kc)$ is monadic,
additional information on the functor $\Qc$ is given in~\cite[Theorem~10.1]{FreMac:1971}.

\subsection{Transferring coalgebras along a colax morphism of comonads}\label{ss:cores}

Let $\Ms$ and $\Ns$ be categories, $\Cc$ a comonad on $\Ms$, and $\Dc$
a comonad on $\Ns$. 
\[
\xymatrix{
\Ms \ar@(ul,dl)_{\Cc}  & \Ns  \ar@(ur,dr)^{\Dc}
}
\]
We use similar notation for the comonad $\Cs$ and for the associated adjunction as for that for $\Ds$, as in Section~\ref{ss:comparison}.

Let
\[
\Hc:\Ms\to\Ns
\qand
\sigma:\Hc\Cc \to \Dc\Hc
\] 
be a functor and a natural transformation, respectively, such that the following diagrams commute.
\begin{align}\label{e:mor-comonad}
&  
\begin{gathered}
\xymatrix@+5pt{    
\Hc\Cc\Cc \ar[r]^-{\sigma\Cc} & \Dc\Hc\Cc \ar[r]^-{\Dc\sigma} & \Dc\Dc\Hc\\
\Hc\Cc \ar[rr]_-{\sigma} \ar[u]^{\Hc\delta} & & \Dc\Hc \ar[u]_{\delta\Hc}
}
\end{gathered}
& & 
\begin{gathered}
\xymatrix@=1.6pc{ 
\Hc\Cc \ar[rr]^-{\sigma}\ar[rd]_{\Hc\epsilon} & & \Dc\Hc \ar[ld]^{\epsilon\Hc}\\
& \Hc
}
\end{gathered}
\end{align}  
One then says that 
\[
(\Hc,\sigma): (\Ms,\Cc) \to (\Ns,\Dc)
\]
is a \emph{colax morphism} of comonads. 

In this situation, there is an induced functor 
\[
\resH : \Ms_{\Cc} \to \Ns_{\Dc}
\]
on the categories of coalgebras, defined as follows. Given a $\Cc$-coalgebra
$(M,c)$, we let
\[
\resH(M,c) := \bigl(\Hc(M),d\bigr)
\]
where
\begin{equation}\label{e:Dcoalg}
d: \Hc(M) \map{\Hc(c)} \Hc\Cc(M) \map{\sigma_M} \Dc\Hc(M).
\end{equation}
Diagrams~\eqref{e:mor-comonad} guarantee that $\resH(M,c)$ is a $\Dc$-coalgebra.
If $f:M\to M'$ is a morphism of $\Cc$-coalgebras, then $\Hc(f):\Hc(M)\to\Hc(M')$ is a
morphism of $\Dc$-coalgebras, by naturality of $\sigma$. This defines
$\resH$ on morphisms. 

It is easy to see that the diagram
\begin{equation}\label{e:cores-forget}
\begin{gathered}
\xymatrix{
\Ms \ar[r]^-{\Hc} & \Ns  \\
\Ms_\Cc \ar[r]_-{\resH} \ar[u]^{\Uc_\Cc}  & \Ns_\Dc \ar[u]_{\Uc_\Dc}
}
\end{gathered}
\end{equation}
is commutative. On the other hand, the diagram
\begin{equation}
\label{e:cores-cofree}
\begin{gathered}
\xymatrix{
\Ms \ar[r]^-{\Hc} \ar[d]_{\Fc_\Cc} & \Ns \ar[d]^{\Fc_\Dc} \\
\Ms_\Cc \ar[r]_-{\resH}   & \Ns_\Dc 
}
\end{gathered}
\end{equation}
commutes (up to isomorphism) if and only if the transformation $\sigma$ is invertible~\cite[Lemma~3]{Joh:1975}.

Consider now the special case in which $\Ms=\Ns$ and $\Hc=\Ic$ is the
identity functor. Thus the comonads $\Cs$ and $\Ds$ act on the same category $\Ns$.
Given a transformation $\sigma:\Cc\to\Dc$ such that $(\Ic,\sigma):(\Ns,\Cc)\to(\Ns,\Dc)$ is a colax morphism of comonads, there is the induced functor 
$\Ic_{\sigma}: \Ns_{\Cc} \to \Ns_{\Dc}$ between coalgebra categories, and as a 
special case of~\eqref{e:cores-forget} we have the commutative diagram
\begin{equation}\label{e:cores-forget-id}
\begin{gathered}
\xymatrix{
 \Ns_\Cc\ar[rd]_{\Uc_\Cc} \ar[rr]^-{\Ic_{\sigma}} & & \Ns_\Dc \ar[ld]^{\Uc_\Dc} \\
& \Ns.
}
\end{gathered}
\end{equation}

\begin{lemma}\label{l:functor-transformation}
Let $\Fc:\Ns_\Cc\to\Ns_\Dc$ be a functor such that
\begin{equation}\label{e:cores-forget-F}
\begin{gathered}
 \xymatrix{
 \Ns_\Cc\ar[rd]_{\Uc_\Cc} \ar[rr]^-{\Fc} & & \Ns_\Dc \ar[ld]^{\Uc_\Dc} \\
& \Ns}
\end{gathered}
\end{equation}
commutes. Then there exists a unique  transformation $\sigma:\Cc\to\Dc$ such that 
$(\Ic,\sigma):(\Ns,\Cc)\to(\Ns,\Dc)$ is a colax morphism of comonads and
$\Fc=\Ic_{\sigma}$.
\end{lemma}
\begin{proof}
This is~\cite[Theorem~6.3]{BW:2005}. A more general result is given 
in~\cite[Theorem~6] {Str:1972}.
\end{proof}

The hypothesis that diagram~\eqref{e:cores-forget-F} commute can be described explicitly as follows. If $(N,a)$ is an object of $\Ns_{\Cc}$, then $\Fc(N,a) = (N,b)$ for a suitable morphism $b: N\to \Dc(N)$; moreover, if $f: (N,a)\to(N',a')$ is a morphism in $\Ns_ {\Cc}$ (i.e., a morphism $f: N\to N'$ in $\Ns$ such that $\Cc(f)a = a'f$), and $\Fc(N,a) = (N,b)$ and 
$\Fc(N',a') = (N,b')$, then also $\Dc(f)b = b'f$.

\subsection{Transferring coalgebras along a lax morphism of comonads}\label{ss:coind}

We return to the setting
\[
\xymatrix{
\Ms \ar@(ul,dl)_{\Cc}  & \Ns  \ar@(ur,dr)^{\Dc}
}
\]
where $\Cc$ and $\Dc$ are comonads. Let
\[
\Kc:\Ns\to\Ms
\qand
\tau:\Cc\Kc \to \Kc\Dc
\] 
be a functor and a natural transformation, respectively, such that the following diagrams commute.
\begin{align}\label{e:lax-mor-comonad}
&  
\begin{gathered}
\xymatrix@+5pt{    
\Cc\Cc\Kc \ar[r]^-{\Cc\tau} & \Cc\Kc\Dc \ar[r]^-{\tau\Dc} & \Kc\Dc\Dc\\
\Cc\Kc \ar[rr]_-{\tau} \ar[u]^{\delta\Kc} & & \Kc\Dc \ar[u]_{\Kc\delta}
}
\end{gathered}
& & 
\begin{gathered}
\xymatrix@=1.6pc{ 
\Cc\Kc \ar[rr]^-{\tau}\ar[rd]_{\epsilon\Kc} & & \Kc\Dc \ar[ld]^{\Kc\epsilon}\\
& \Kc
}
\end{gathered}
\end{align}  
One then says that 
\[
(\Kc,\tau): (\Ns,\Dc) \to (\Ms,\Cc)
\]
is a \emph{lax morphism} of comonads.

Given a $\Dc$-coalgebra $(N,d)$, consider the maps
\begin{equation}\label{e:eq-coalg}
\Cc\Kc(N)\map{\delta_{\Kc(N)}} \Cc\Cc\Kc(N) \map{\Cc(\tau_N)} \Cc\Kc\Dc(N)
\qand
\Cc\Kc(N)\map{\Cc\Kc(d)} \Cc\Kc\Dc(N).
\end{equation}

\begin{lemma}\label{l:eq-coalg}
The maps in~\eqref{e:eq-coalg} form a coreflexive pair in the category $\Ms_{\Cc}$.
\end{lemma}
\begin{proof}
Each arrow in~\eqref{e:eq-coalg} is a morphism of $\Cc$-coalgebras (the arrow
$\delta_{\Kc(N)}$ by coassociativity of $\delta$, and the arrows in the image of $\Cc$ by naturality of $\delta$). 
The map
$\Cc\Kc(\epsilon_N)$, where $\epsilon:\Dc\to\Ic$ is the counit of the comonad $\Dc$,
is a common retraction of the two maps (this uses the second diagram in~\eqref{e:lax-mor-comonad}), and is a morphism of $\Cc$-coalgebras.
\end{proof}

We fix the transformation $\tau$
and make the following assumption.
\begin{equation}\label{e:existence}
\text{The pair~\eqref{e:eq-coalg} possesses an equalizer in
$\Ms_{\Cc}$ for all $\Dc$-coalgebras $(N,d)$.}
\end{equation}

We let $\indK(N,d)$ denote the equalizer in $\Ms_{\Cc}$ of the pair~\eqref{e:eq-coalg}.
This defines a functor
\[
\indK: \Ns_{\Dc} \to \Ms_{\Cc}
\]
on objects. The maps in~\eqref{e:eq-coalg} are natural with respect to morphisms of $\Dc$-coalgebras; this
allows us to define $\indK$ on morphisms. By construction, there is a canonical 
morphism of coalgebras
\begin{equation}\label{e:indK}
\indK(N,d) \to \Cc\Kc(N).
\end{equation}

\begin{lemma}\label{l:coind-cofree}
The diagram
\begin{equation}
\label{e:coind-cofree}
\begin{gathered}
\xymatrix{
\Ns \ar[r]^-{\Kc} \ar[d]_{\Fc_\Dc} & \Ms \ar[d]^{\Fc_\Cc} \\
\Ns_\Dc \ar[r]_-{\indK} & \Ms_\Cc
}
\end{gathered}
\end{equation}
is commutative (up to isomorphism). 
\end{lemma}
\begin{proof}
Let $N$ be an object of $\Ns$. Consider the maps $f$, $g$ and $h$
defined as follows.
\begin{gather*}
f: \Cc\Kc(N) \map{\delta_{\Kc(N)}} \Cc\Cc\Kc(N) \map{\Cc(\tau_N)} 
\Cc\Kc\Dc(N)\\
g: \Cc\Kc\Dc(N) \map{\Cc\Kc(\delta_N)} \Cc\Kc\Dc\Dc(N)\\
h:\Cc\Kc\Dc(N) \map{\delta_{\Kc\Dc(N)}} \Cc\Cc\Kc\Dc(N)
\map{\Cc(\tau_{\Dc(N)})} \Cc\Kc\Dc\Dc(N).
\end{gather*}
Note the pair $(g,h)$ is~\eqref{e:eq-coalg} for the
$\Dc$-coalgebra $\Fc_\Dc(N)=(\Dc(N),\delta_N)$, whose
equalizer is $\indK\Fc_\Dc(N)$.
We show below that the diagram 
\[
\xymatrix{
\Cc\Kc(N) \ar[r]^-{f} & \Cc\Kc\Dc(N) \ar@<0.5ex>[r]^-{g} \ar@<-0.5ex>[r]_-{h} & \Cc\Kc\Dc\Dc(N)
}
\]
is a split cofork in the category $\Ms_{\Cc}$.
It then follows that $f$ is the equalizer of $(g,h)$, from which
$\Fc_\Cc\Kc(N)\cong \indK\Fc_\Dc(N)$.

First of all, $f$ is a morphism of $\Cc$-coalgebras, because so are
$\delta_{\Kc(N)}$ (by coassociativity) and $\Cc(\tau_N)$ (by naturality).

That the diagram is a cofork is shown by the commutativity of the
following diagram.
\[
\xymatrix@R-5pt{
\Cc\Kc \ar[r]^-{\delta\Kc} \ar[d]_{\delta\Kc} & \Cc\Cc\Kc \ar[r]^-{\Cc\tau} \ar[d]^{\delta\Cc\Kc} & \Cc\Kc\Dc \ar[d]^{\delta\Kc\Dc} \\
\Cc\Cc\Kc \ar[r]^-{\Cc\delta\Kc} \ar[d]_{\Cc\tau} & \Cc\Cc\Cc\Kc \ar[r]^-{\Cc\Cc\tau} & \Cc\Cc\Kc\Dc \ar[d]^{\Cc\tau\Dc} \\
\Cc\Kc\Dc \ar[rr]_-{\Cc\Kc\delta} & & \Cc\Kc\Dc\Dc
}
\]
The top left corner commutes
by coassociativity and the top right by naturality. The commutativity
of the bottom rectangle follows from that of the first diagram~\eqref{e:lax-mor-comonad}.

Consider now the maps $p$ and $q$ defined by
\begin{gather*}
p:\Cc\Kc\Dc\Dc \map{\Cc\Kc\Dc\epsilon} \Cc\Kc\Dc\\
q:\Cc\Kc\Dc \map{\Cc\Kc\epsilon} \Cc\Kc.
\end{gather*}
We show they split the above cofork. 

We have $pg=\id$ by the counit axiom for $\Dc$. Also, $qf=\id$
in view of the commutative diagram
\[
\xymatrix@R-5pt{
\Cc\Kc \ar[r]^-{\delta\Kc} \ar@{=}[rd] & \Cc\Cc\Kc \ar[r]^-{\Cc\tau} \ar[d]^{\Cc\epsilon\Kc} & \Cc\Kc\Dc \ar[d]^{\Cc\Kc\epsilon} \\
& \Cc\Kc \ar@{=}[r] & \Cc\Kc
}
\]
where again we use counitality and the second diagram in~\eqref{e:lax-mor-comonad}. It only remains to check that $ph=fq$. This holds in view of the diagram
\[
\xymatrix@R-5pt{
\Cc\Kc\Dc \ar[r]^-{\delta\Kc\Dc} \ar[d]_{\Cc\Kc\epsilon} & \Cc\Cc\Kc\Dc \ar[r]^-{\Cc\tau\Dc} \ar[d]_{\Cc\Cc\Kc\epsilon} & \Cc\Kc\Dc\Dc \ar[d]^{\Cc\Kc\Dc\epsilon} \\
\Cc\Kc \ar[r]_-{\delta\Kc}  & \Cc\Cc\Kc \ar[r]_-{\Cc\tau} & \Cc\Kc\Dc
}
\]
which commutes by naturality.
\end{proof}

\begin{remark}
We will apply Lemma~\ref{l:coind-cofree} in the same context 
as that of 
the Adjoint Lifting Theorem (Theorem~\ref{t:adjoint} below).
In this situation, diagram~\eqref{e:coind-cofree} is right adjoint 
to~\eqref{e:cores-forget}, and hence the commutativity of any one
of these diagrams is equivalent to that of the other.
\end{remark}

\subsection{Mates}\label{ss:mates}

Suppose that, in addition to the comonads $\Cc$ and $\Dc$ as in Sections~\ref{ss:cores}--\ref{ss:coind}, an adjunction $(\Hc,\Kc)$ is given, as below.
\[
\xymatrix{
\Ms \ar@/^/[r]^{\Hc} & \Ns \ar@/^/[l]^{\Kc} 
}
\]
Let
\[
\eta:\Ic\to\Kc\Hc
\qand
\xi:\Hc\Kc\to\Ic
\]
be the unit and counit of the adjunction. 

In this situation, there is a bijective correspondence between transformations
\[
\sigma:\Hc\Cc\to\Dc\Hc
\]
and transformations
\[
\tau: \Cc\Kc\to\Kc\Dc.
\]
Given $\sigma$, the transformation $\tau$ is defined by
\[
\xymatrix{
\tau: \Cc\Kc=\Ic\Cc\Kc \map{\eta\Cc\Kc} \Kc\Hc\Cc\Kc
\map{\Kc\sigma\Kc} \Kc\Dc\Hc\Kc \map{\Kc\Dc\xi} \Kc\Dc\Ic = \Kc\Dc.
}
\]
Conversely, given $\tau$, the transformation $\sigma$ is defined by
\[
\xymatrix{
\sigma: \Hc\Cc=\Hc\Cc\Ic \map{\Hc\Cc\eta} \Hc\Cc\Kc\Hc
\map{\Hc\tau\Hc} \Hc\Kc\Dc\Hc \map{\xi\Dc\Hc} \Ic\Dc\Hc=\Dc\Hc.
}
\]
The transformations $\sigma$ and $\tau$ are said to be \emph{mates}.
They determine each other. Moreover,~\cite[Theorem~9]{Str:1972} 
\begin{multline*}
(\Hc,\sigma): (\Ms,\Cc) \to (\Ns,\Dc)
\text{ is a colax morphism of comonads} \\
 \iff  (\Kc,\tau):  (\Ns,\Dc)\to (\Ms,\Cc)
\text{ is a lax morphism of comonads.}
\end{multline*}

For more information on mates, see~\cite[Section~2.2]{KelStr:1974} or~\cite[Section~6.1]{Lei:2004}.

\subsection{Transferring coalgebras along an adjunction}\label{ss:cores-coind}

We now make use of all the preceding data 
discussed in Sections~\ref{ss:cores}--\ref{ss:mates}, as summarized by the
following diagrams.
\begin{equation*}
\begin{gathered}
\xymatrix{
\Ms \ar@(ul,dl)_{\Cc} \ar@/^/[r]^{\Hc} & \Ns \ar@/^/[l]^{\Kc} \ar@(ur,dr)^{\Dc}
}
\end{gathered}
\qquad
\begin{gathered}
\xymatrix{
\Ms \ar[r]^-{\Hc} \ar@{}[rd]|{\SEarrow \!\!\sigma} & \Ns\\
\Ms \ar[u]^-{\Cc} \ar[r]_{\Hc} & \Ns \ar[u]_{\Dc}
}
\end{gathered}
\qquad
\begin{gathered}
\xymatrix{
\Ms  \ar[d]_{\Cc} \ar@{}[rd]|{\SEarrow \!\!\tau} & \Ns \ar[l]_-{\Kc} \ar[d]^{\Dc} \\
\Ms   & \Ns \ar[l]^{\Kc} 
}
\end{gathered}
\end{equation*}
The transformations $\sigma$ and $\tau$ are mates for which $(\Hc,\sigma)$ and $(\Kc,\tau)$
are morphisms of comonads (colax and lax, respectively) and
we assume hypothesis~\eqref{e:existence}.

According to the
constructions of Sections~\ref{ss:cores}--\ref{ss:coind}, there are functors between
categories of coalgebras as follows.
\[
\xymatrix{
\Ms_{\Cc} \ar@/^/[r]^{\resH} & \Ns_{\Dc} \ar@/^/[l]^{\indK} 
}
\]
We proceed to turn them into an adjoint pair.

We first define a transformation
\[
\Hat{\eta}:\Ic\to \indK\resH.
\]
Let $(M,c)$ be a $\Cc$-coalgebra, write $\resH(M,c)=\bigl(\Hc(M),d\bigr)$. 
One verifies that the diagram 
\begin{equation}\label{e:fork2}
\xymatrix@C+5pt{
M \ar[r]^-{e} & \Cc\Kc\Hc(M) \ar@<0.5ex>[r]^-{f} \ar@<-0.5ex>[r]_-{g} & \Cc\Kc\Dc\Hc(M)
}
\end{equation}
is a cofork, where the maps $e,f$ and $g$ are defined as follows.
\begin{gather*}
e:M \map{c} \Cc(M) \map{\Cc(\eta_M)} \Cc\Kc\Hc(M)\\
f: \Cc\Kc\Hc(M) \map{\delta_{\Kc\Hc(M)}} \Cc\Cc\Kc\Hc(M)
\map{\Cc(\tau_{\Hc(M)})} \Cc\Kc\Dc\Hc(M)\\
g:\Cc\Kc\Hc(M) \map{\Cc\Kc\Hc(c)} \Cc\Kc\Hc\Cc(M)
\map{\Cc\Kc(\sigma_M)} \Cc\Kc\Dc\Hc(M).
\end{gather*}
The verification uses the definition of mate.

The map $e$ is a morphism of $\Cc$-coalgebras (by coassociativity of $c$ and
naturality of $\delta$).
In view of the definition of $d$~\eqref{e:Dcoalg}, the pair $(f,g)$ coincides with the pair of arrows in~\eqref{e:eq-coalg}, for $N=\Hc(M)$. It follows from the definition of 
$\indK\bigl(\Hc(M),d\bigr)$ that there is a unique morphism of $\Cc$-coalgebras
\[
\Hat{\eta}_{(M,c)}: (M,c) \to \indK\resH(M,c)
\]
such that
\[
\xymatrix{
M \ar[r]^-{e} \ar[d]_{\Hat{\eta}}
& \Cc\Kc\Hc(M) \\
\indK\resH(M,c) \ar[ru]_-{\Eq(f,g)}
}
\]
commutes.

 We turn to the transformation
\[
\Hat{\xi}:\resH\indK \to\Ic.
\]
On a $\Dc$-coalgebra $(N,d)$, $\Hat{\xi}$ is defined as the composite
\[
\resH\indK(N,d)=\Hc\indK(N,d) \map{} \Hc\Cc\Kc(N) \map{\Hc(\epsilon_{\Kc(N)})}
\Hc\Ic\Kc(N) = \Hc\Kc(N) \map{\xi_N} N,
\]
where the first map is obtained by applying $\Hc$ to~\eqref{e:indK}. One 
verifies that this is a morphism of $\Dc$-coalgebras.

\begin{theorem}[The Adjoint Lifting Theorem]\label{t:adjoint}
Under hypothesis~\eqref{e:existence}, the functors
\[
\xymatrix{
\Ms_{\Cc} \ar@/^/[r]^{\resH} & \Ns_{\Dc} \ar@/^/[l]^{\indK} 
}
\]
of Sections~\ref{ss:cores}--\ref{ss:coind}, with the transformations
$\Hat{\eta}$ and $\Hat{\xi}$ above, form an adjunction.
\end{theorem}
\begin{proof}
This statement is dual to~\cite[Theorem~4.5.6]{Bor:1994ii}; see also~\cite[Exercise~4.8.6]{Bor:1994ii}. 
The theorem is given under the hypothesis that $\Ms_{\Cc}$ has equalizers
of all coreflexive pairs.
By Lemma~\ref{l:eq-coalg}, the pairs of the form~\eqref{e:eq-coalg}
are coreflexive, so this hypothesis is stronger than~\eqref{e:existence}.
The result continues to hold under the weaker hypothesis, with the same proof.
\end{proof}

Additional references for the Adjoint Lifting Theorem are~\cite[Theorem 7.4.b]{BW:2005},
~\cite[Theorem~2]{Joh:1975},
and~\cite[Proposition~A.1.1.3]{Joh:2002}.

\begin{remark}\label{r:linton}
Consider the case in which $\Ms=\Ns$ and $\Hc=\Kc=\Ic$. In this case the mates $\sigma$ and $\tau$ coincide; they consist of a natural transformation $\Cc\to\Dc$
that commutes with the comonad structures, and $\resH(M,c)=(M,\sigma_M c)$.
In this special case, the result of Theorem~\ref{t:adjoint} appears (in dual form) in work of Linton~\cite[Corollary~1]{Lin:1969}. This and other early instances of the Adjoint Lifting Theorem in the literature are listed by Johnstone in~\cite[page~295]{Joh:1975}.
Note that, even though $(\Hc,\Kc)$ is in this case
an adjoint equivalence, the categories of coalgebras $\Ms_\Cc$ and $\Ns_\Dc$ need
not be equivalent.

On the other hand,
suppose that $\Cc=\Ic$ is the identity comonad on $\Ms$, $\Dc=\Hc\Kc$ is
the comonad on $\Ns$ associated to the adjunction $(\Hc,\Kc)$, and $\sigma$ is
the transformation
\[
\Hc\Cc=\Hc\Ic \map{\Hc\eta} \Hc\Kc\Hc = \Dc\Hc.
\]
It is easy to see that then $(\Hc,\sigma)$ is a colax morphism of comonads
and that 
\[
\resH:\Ms\to \Ns_{\Dc}
\]
is the comparison functor of Section~\ref{ss:comparison}. As already mentioned, Beck's Comonadicity Theorem (Theorem~\ref{t:beck}) gives conditions under which the comparison functor is an equivalence. We are interested in a similar situation in which neither comonad $\Cc$ nor $\Dc$ is trivial but we can still conclude that $(\resH,\indK)$
is an adjoint equivalence. This is the topic of Section~\ref{ss:descent}.
\end{remark}

\subsection{Conservative and faithful functors}\label{ss:con-fai}

The following miscellaneous results will be useful in Section~\ref{s:schneider}.
A category $\Ms$ is \emph{balanced} if an arrow $f$ in $\Ms$ that is both
an epimorphism and a monomorphism is in fact an isomorphism.

\begin{proposition}\label{p:con-fai}
Let $\Hc:\Ms\to\Ns$ be a functor.
\begin{enumerate}[(i)]
\item Assume that $\Ms$ has all equalizers (or coequalizers) and $\Hc$ preserves them. 
If $\Hc$ is conservative, then it is faithful.
\item Assume that $\Ms$ is balanced. If $\Hc$ is faithful,
then it is conservative.
\end{enumerate}
\end{proposition}
\begin{proof}
Assume that $\Ms$ has equalizers and these are preserved by $\Hc$.
Let $f,g:M\to M'$ be morphisms of $\Ms$ such that
$\Hc(f)=\Hc(g)$. Then, there are equalizers
\[
\xymatrix{
E \ar[r]^-{i} & M \ar@<0.5ex>[r]^-{f} \ar@<-0.5ex>[r]_-{g} & M'
}
\qqand
\xymatrix{
\Hc(E) \ar[r]^-{\Hc(i)} & \Hc(M) \ar@<0.5ex>[r]^-{\Hc(f)} \ar@<-0.5ex>[r]_-{\Hc(g)} & \Hc(M')
}
\]
in $\Ms$ and $\Ns$ respectively. Since $\Hc(f)=\Hc(g)$, it follows that $\Hc(i)$
is an isomorphism. Since $\Hc$ is conservative, $i$ is likewise an isomorphism,
and so $f=g$. Thus $\Hc$ is faithful. This proves one alternative of statement (i);
the remaining one follows since the notions of ``faithful'' and ``conservative'' are self-dual.

Statement (ii) follows from the fact that a faithful functor reflects both epimorphisms and
monomorphisms~\cite[Propositions~1.7.6 and~1.8.4]{Bor:1994i}.
\end{proof}

The following is an immediate consequence.

\begin{corollary}\label{c:con-fai}
Let $\Hc:\Ms\to\Ns$ be a left (or right) exact functor between abelian categories. Then
$\Hc$ is conservative if and only if it is faithful.
\end{corollary}

\section{Conjugating a comonad by an adjunction}\label{s:conjugating}

Consider the situation
\[
\xymatrix{
\Ms \ar@(ul,dl)_{\Cc} \ar@/^/[r]^{\Hc} & \Ns \ar@/^/[l]^{\Kc} 
}
\]
in which $(\Hc,\Kc)$ is an adjunction and $\Cc$ is a comonad. In Section~\ref{ss:conjugate-comonad} we
turn the composite functor
$\Hc\Cc\Kc$ 
into a comonad on $\Ns$. We say this comonad is obtained from $\Cc$ by conjugation
under the adjunction $(\Hc,\Kc)$ and denote it by $\cCc$. 
The comonads $\Cc$ and $\cCc$ are related by a canonical colax morphism
(Section~\ref{ss:colax}). This provides the ingredients for 
an application of the Adjoint Lifting Theorem in Section~\ref{ss:lifting-conjugate},
which results in an adjunction between the categories of coalgebras over $\Cc$
and $\cCc$.
Our main goal here, achieved in Section~\ref{ss:descent}, is to obtain precise conditions when this is in fact
an adjoint equivalence. We call this result the \emph{Conjugate Comonadicity
Theorem}. It is an extension of Beck's Comonadicity Theorem (and follows from it).
We expand on the study of the conjugate comonad in Sections~\ref{ss:universal} and~\ref{ss:equivalence-coalgebra}, where we establish
a universal property of $\cCc$ and derive a criterion for a colax morphism of
comonads to induce an equivalence between coalgebra categories.

Let 
\[
\delta:\Cc\to\Cc^2 \qand \epsilon:\Cc\to\Ic
\]  
be the comultiplication 
and the counit of the comonad $\Cc$. Let
\[
\eta:\Ic\to\Kc\Hc \qand \xi:\Hc\Kc\to\Ic
\] 
be the unit and the counit of the adjunction $(\Hc,\Kc)$.
We let $\cCc$ denote the composite functor
\[
\cCc:\Ns\map{\Kc}\Ms\map{\Cc}\Ms\map{\Hc}\Ns.
\]
(The notation does not show that the functor $\cCc$, and the comonad structure
to be defined below, depend not only on the comonad $\Cc$ but also on the adjunction $(\Hc,\Kc)$.)

\subsection{The conjugate comonad}\label{ss:conjugate-comonad}

We turn the functor $\cCc$ into a comonad on $\Ns$. 

Consider the adjunction $(\Uc_\Cc,\Fc_\Cc)$ associated to
the comonad $\Cc$, as in Section~\ref{ss:comparison}. 
Composing the adjunctions
\begin{equation}\label{e:comp-adj}
\xymatrix{
\Ms_{\Cc}  \ar@/^/[r]^{\Uc_\Cc} & \Ms \ar@/^/[r]^{\Hc} \ar@/^/[l]^{\Fc_\Cc} & \Ns \ar@/^/[l]^{\Kc} 
}
\end{equation}
we obtain an adjunction $(\Hc\Uc_\Cc, \Fc_\Cc\Kc)$. The comonad on $\Ns$ defined by this adjunction (as in Section~\ref{ss:comparison}) is 
\[
\Hc\Uc_\Cc\Fc_\Cc\Kc=\Hc\Cc\Kc=\cCc,
\] 
and the structure turns out to be as follows.

The comultiplication is
\begin{equation}\label{e:conjugate-com}
\cCc=\Hc\Cc\Kc \map{\Hc \delta \Kc} \Hc\Cc\Cc\Kc = \Hc\Cc\Ic\Cc\Kc \map{\Hc\Cc\eta\Cc\Kc} \Hc\Cc\Kc\Hc\Cc\Kc=(\cCc)^2.
\end{equation}
 The counit is
\begin{equation}\label{e:conjugate-cou}
\cCc=\Hc\Cc\Kc\map{\Hc\epsilon\Kc}\Hc\Ic\Kc=\Hc\Kc\map{\xi}\Ic.
\end{equation}

\begin{definition}\label{d:conjugate-comonad}
We refer to the functor $\cCc=\Hc\Cc\Kc:\Ns\to\Ns$ with the above structure
as the \emph{conjugate comonad} of $\Cc$ under the adjunction $(\Hc,\Kc)$.
\end{definition}

\begin{remark} When $\Cc=\Ic$ is the identity comonad, the conjugate comonad
is simply the comonad $\Hc\Kc$ defined by the adjunction, as in Section~\ref{ss:comparison}.
On the other hand, as explained above,
the general case of the construction can be reduced to this
special one.
\end{remark}

Consider again the composite adjunction~\eqref{e:comp-adj}.
We use $\QcC$ to denote the corresponding comparison functor,
as in~\eqref{e:comp-fun}. It follows from the discussion in Section~\ref{ss:comparison}
that $\QcC$ is the unique functor 
\begin{equation}\label{e:comp-funC}
\QcC: \Ms_{\Cc} \to \Ns_{\cCc}
\end{equation}
between categories of coalgebras such that
\begin{equation*}
\begin{gathered}
\xymatrix{
\Ms \ar[r]^-{\Hc} & \Ns \\
\Ms_{\Cc} \ar[u]^{\Uc_{\Cc}} \ar[r]_{\QcC} & \Ns_{\cCc} \ar[u]_{\Uc_{\cCc}}
}
\end{gathered}
\qand
\begin{gathered}
\xymatrix{
 \Ms \ar[d]_{\Fc_{\Cc}} & \Ns \ar[l]_-{\Kc} \ar[d]^{\Fc_{\cCc}} \\
 \Ms_{\Cc} \ar[r]_{\QcC} & \Ns_{\cCc}   &
}
\end{gathered}
\end{equation*}
commute. Explicitly, if $(M,c)$ is a $\Cc$-coalgebra,
then $\QcC(M,c)=\bigl(\Hc(M),d\bigr)$ where $d$ is the composite
\begin{equation}\label{e:comp-coalg}
\Hc(M) \map{\Hc(c)} \Hc\Cc(M) \map{\Hc\Cc(\eta_M)} \Hc\Cc\Kc\Hc(M)=\cCc\Hc(M).
\end{equation}

In Section~\ref{ss:descent} we discuss conditions under which the functor $\QcC$
is an equivalence.


\subsection{A colax morphism}\label{ss:colax}

 Define a transformation $\gamma$ by
\begin{equation}\label{e:def-sigma}
\gamma: \Hc\Cc =\Hc\Cc\Ic\map{\Hc\Cc\eta} \Hc\Cc\Kc\Hc = \cCc\Hc.
\end{equation}

\begin{lemma}\label{l:conjugate-comonad}
The transformation $\gamma$ turns $\Hc$ into a colax morphism of comonads
\[
(\Hc,\gamma): (\Ms,\Cc)\to (\Ns,\cCc).
\]
In addition, its mate is the transformation $\tgamma$ given by
\begin{equation}\label{e:def-tau}
\tgamma: \Cc\Kc = \Ic\Cc\Kc\map{\eta\Cc\Kc} \Kc\Hc\Cc\Kc = \Kc\cCc.
\end{equation}
\end{lemma}
\begin{proof} In order to check that $\gamma$ is a colax morphism, consider the following diagram.
\[
\xymatrix@C+20pt{
 \Hc\Cc\ar@/^2.2pc/[rr]^-{\gamma} \ar[d]_{\Hc\delta} \ar[r]^-{\Hc\Cc\eta} & \Hc\Cc\Kc\Hc\ar@{=}[r] \ar[d]^{\Hc\delta\Kc\Hc} &     \cCc\Hc \ar[ddd]^{\delta\Hc} \\
 \Hc\Cc\Cc \ar[r]^{\Hc\Cc\Cc\eta} \ar[d]_{\Hc\Cc\eta\Cc} \ar@/_3.2pc/[dd]_{\gamma\Cc} &  
\Hc\Cc\Cc\Kc\Hc \ar[d]^{\Hc\Cc\eta\Cc\Kc\Hc} \\
 \Hc\Cc\Kc\Hc\Cc \ar@{=}[d] \ar[r]_-{\Hc\Cc\Kc\Hc\Cc\eta} & \Hc\Cc\Kc\Hc\Cc\Kc\Hc \ar@{=}[rd]   \\
\cCc\Hc\Cc \ar[rr]_-{\cCc\gamma} & & \cCc\cCc\Hc
}
\]
The outer diagrams commute by definition of $\gamma$ and
of $\delta:\cCc\to\cCc\cCc$. The inner rectangles commute by naturality.
It follows that the diagram commutes, and this yields the first condition in~\eqref{e:mor-comonad}. 

The second condition in~\eqref{e:mor-comonad} is the commutativity of the
following diagram. 
\[
\xymatrix@C+20pt{
\Hc\Cc \ar[r]^-{\Hc\epsilon} \ar[d]_{\Hc\Cc\eta} \ar@/_2.5pc/[dd]_-{\gamma}  & \Hc\ar[d]^-{\Hc\eta} \ar@/^2.5pc/@{=}[dd] \\
\Hc\Cc\Kc\Hc \ar[r]^{\Hc\epsilon\Kc\Hc} \ar@{=}[d] & \Hc\Kc\Hc \ar[d]^{\xi\Hc} \\
\cCc\Hc \ar[r]_-{\epsilon\Hc} & \Hc
}
\]
This employs one of the axioms involving the unit and counit of the adjunction $(\Hc,\Kc)$.

Finally, the fact that the mate of $\gamma$ is given by~\eqref{e:def-tau} follows
from the commutativity of the following diagram.
\[
\xymatrix@C+5pt{
\Cc\Kc \ar[r]^-{\eta\Cc\Kc} \ar[dd]_{\tgamma} & \Kc\Hc\Cc\Kc \ar[d]^{\Kc\Hc\Cc\eta\Kc} 
\ar@/^4.5pc/[dd]^{\Kc\gamma\Kc}  \ar@{=}[ldd] \\
& \Kc\Hc\Cc\Kc\Hc\Kc \ar@{=}[d] \\
\Kc\cCc & \Kc\cCc\Hc\Kc \ar[l]^{\Kc\cCc\xi}
}
\]
This employs the remaining axiom involving the unit and counit of the adjunction.
\end{proof}

We now have two functors $\Ms_{\Cc} \to \Ns_{\cCc}$ between categories of coalgebras. The comparison functor $\QcC$
defined in~\eqref{e:comp-funC} and the functor $\Hc_{\gamma}$
induced by the colax morphism $(\Hc,\gamma)$ as in Section~\ref{ss:cores}.

\begin{lemma}\label{l:comp-colax}
The functors $\QcC$ and $\Hc_{\gamma}$ coincide.
\end{lemma}
\begin{proof}
Let $(M,c)$ be a $\Cc$-coalgebra. We calculate $\Hc_{\gamma}(M,c)$ using~\eqref{e:Dcoalg} and~\eqref{e:def-sigma}, and $\QcC(M,c)$ using~\eqref{e:comp-coalg}. In both cases the result is $\bigl(\Hc(M),d\bigr)$ where $d$ is as in~\eqref{e:comp-coalg}.
On morphisms both functors agree with $\Hc$.
\end{proof}

\subsection{Lifting for the conjugate comonad}\label{ss:lifting-conjugate}

We now consider the comonad $\Cc$ together with its conjugate $\cCc$
\[
\xymatrix{
\Ms \ar@(ul,dl)_{\Cc} \ar@/^/[r]^{\Hc} & \Ns \ar@/^/[l]^{\Kc}  \ar@(ur,dr)^{\cCc}
}
\]
in the context of the Adjoint Lifting Theorem of Section~\ref{ss:cores-coind}.
We work with the transformation $\gamma$ and its mate $\tgamma$ of Section~\ref{ss:colax}.

The first thing to note is that the pair of arrows~\eqref{e:eq-coalg}
is in this situation of a special kind. 

\begin{lemma}\label{l:Hsplit}
For any $\cCc$-coalgebra $(N,d)$, the pair of arrows
\[
\Cc\Kc(N)\map{\delta_{\Kc(N)}} \Cc\Cc\Kc(N) \map{\Cc(\tgamma_N)} \Cc\Kc\cCc(N)
\qand
\Cc\Kc(N)\map{\Cc\Kc(d)} \Cc\Kc\cCc(N)
\]
in the category $\Ms_{\Cc}$ is coreflexive and $\Hc\Uc_{\Cc}$-split.
\end{lemma}
\begin{proof} 
The pair is coreflexive by Lemma~\ref{l:eq-coalg} applied to the comonad $\cCc$. Applying $\Hc\Uc_{\Cc}$ to the pair we obtain
\[
\cCc(N)\map{\Hc(\delta_{\Kc(N)})} \Hc\Cc\Cc\Kc(N) \map{\Hc\Cc(\tgamma_N)} (\cCc)^2(N)
\qand
\cCc(N)\map{\cCc(d)} (\cCc)^2(N).
\]
According to~\eqref{e:def-tau}, $\tgamma=\eta\Cc\Kc$.
Recalling~\eqref{e:conjugate-com} and~\eqref{e:conjugate-cou}, we see that
the previous pair is precisely the pair~\eqref{e:pair} associated to the $\cCc$-coalgebra $(N,d)$.
We know from Section~\ref{ss:coalgebra} that this pair is part of the split cofork~\eqref{e:splitfork}.
\end{proof}

Consider now the following hypothesis.
\begin{equation}
\label{e:preserve}
\begin{split} 
& \text{The functors $\Cc$ and $\Cc^2$ preserve the equalizer of any parallel pair in $\Ms$}\\
& \text{which is both coreflexive and $\Hc$-split.}
\end{split}
\end{equation}

\begin{lemma}\label{l:lifting-conjugate}
Suppose that:
\begin{itemize}
\item The category $\Ms$ satisfies hypothesis~\eqref{e:coreflexive}.
\item The comonad $\Cc$ satisfies hypothesis~\eqref{e:preserve}.
\end{itemize}
Then the category $\Ms_{\Cc}$ possesses equalizers of all parallel pairs which
are both coreflexive and $\Hc\Uc_{\Cc}$-split.
Moreover, these equalizers are preserved by $\Uc_{\Cc}$.
\end{lemma}
\begin{proof}
Consider such a pair 
\[
\xymatrix@C+5pt{
 (M,c) \ar@<0.5ex>[r]^-{f} \ar@<-0.5ex>[r]_-{g} & (M',c')
}
\]
of arrows in $\Ms_{\Cc}$.  Forgetting the $\Cc$-coalgebra structure we obtain
a pair of parallel arrows in $\Ms$ which is coreflexive and $\Hc$-split. By hypothesis~\eqref{e:coreflexive}, this pair has an equalizer 
\[
\xymatrix@C+5pt{
 E \ar[r] & M \ar@<0.5ex>[r]^-{f} \ar@<-0.5ex>[r]_-{g} & M'
}
\]
in $\Ms$, and by~\eqref{e:preserve}, this equalizer is preserved by $\Cc$
and $\Cc^2$.
By Lemma~\ref{l:creation}, $E$ has
a $\Cc$-coalgebra structure $c''$ for which 
\[
\xymatrix@C+5pt{
 (E,c'') \ar[r] & (M,c) \ar@<0.5ex>[r]^-{f} \ar@<-0.5ex>[r]_-{g} & (M',c')
}
\]
is an equalizer in $\Ms_{\Cc}$. This shows that $\Ms_{\Cc}$ possesses 
the required equalizers and that they are preserved by $\Uc_{\Cc}$.
\end{proof}

\begin{proposition}\label{p:lifting-conjugate}
Suppose that:
\begin{itemize}
\item The category $\Ms$ satisfies hypothesis~\eqref{e:coreflexive}.
\item The comonad $\Cc$ satisfies hypothesis~\eqref{e:preserve}.
\end{itemize}
Then there is an adjunction
\[
\xymatrix{
\Ms_{\Cc} \ar@/^/[r]^{\Hc_{\gamma}} & \Ns_{\cCc} \ar@/^/[l]^{\Kc^{\tgamma}}.
}
\]
The transformations $\gamma$ and $\tgamma$ are as in Section~\ref{ss:colax}
and the adjunction as in Section~\ref{ss:cores-coind}.
\end{proposition}
\begin{proof}
It follows from Lemmas~\ref{l:Hsplit} 
and~\ref{l:lifting-conjugate} that the category $\Ms_{\Cc}$ possesses
equalizers of all pairs of the form~\eqref{e:eq-coalg} for $\Dc=\cCc$. Thus,
 hypothesis~\eqref{e:existence} holds.
The adjunction then follows from Theorem~\ref{t:adjoint}.
\end{proof}

We proceed to show that under certain additional hypotheses the
preceding is in fact
an adjoint equivalence.

\subsection{The Conjugate Comonadicity Theorem}\label{ss:descent}

We return to the comparison functor $\QcC$ of~\eqref{e:comp-funC}.

\begin{theorem}[The Conjugate Comonadicity Theorem]\label{t:descent}
The following are equivalent statements.
\begin{enumerate}[(i)]
\item For all comonads $\Cc$ on $\Ms$ which satisfy hypothesis~\eqref{e:preserve},
the comparison functor  $\QcC:\Ms_{\Cc} \to \Ns_{\cCc}$
is an equivalence.
\item Hypotheses~\eqref{e:coreflexive} through~\eqref{e:Hconserve} hold.
\end{enumerate}
\end{theorem}
\begin{proof} 
The  identity comonad $\Ic$ satisfies~\eqref{e:preserve} trivially,
and $\Qc_{\Ic}$ is the comparison functor of~\eqref{e:comp-fun}.
Thus, (i) implies (ii) by Theorem~\ref{t:beck} (Beck's).

For the converse, assume~\eqref{e:coreflexive}--\eqref{e:Hconserve}
and the comonad $\Cc$ satisfies~\eqref{e:preserve}.
It suffices to verify that the composite adjunction $(\Hc\Uc_\Cc, \Fc_\Cc\Kc)$ of~\eqref{e:comp-adj} satisfies the hypotheses of Beck's theorem.

By Lemma~\ref{l:lifting-conjugate}, the category $\Ms_{\Cc}$ possesses equalizers of all parallel pairs which
are both coreflexive and
$\Hc\Uc_{\Cc}$-split, and they are preserved by $\Uc_{\Cc}$.
Hence these equalizers are preserved by $\Hc\Uc_\Cc$, in view of~\eqref{e:Hpreserve}.
In addition, since both $\Hc$ and $\Uc_{\Cc}$ are conservative (the former by~\eqref{e:Hconserve}), so is their composite $\Hc\Uc_{\Cc}$.
Thus the hypotheses  of Beck's theorem are satisfied by the adjunction $(\Hc\Uc_\Cc, \Fc_\Cc\Kc)$ and the proof is complete.
\end{proof}

\begin{remark}
Hypotheses~\eqref{e:coreflexive}--\eqref{e:Hconserve} enter both in Theorem~\ref{t:beck} and in Theorem~\ref{t:descent}. Thus, one may view the Conjugate Comonadicity Theorem as adding one equivalent condition (statement (i) above)
to those in Beck's theorem.
\end{remark}

We connect with the result of Proposition~\ref{p:lifting-conjugate}.

\begin{corollary}\label{c:descent}
Assume hypotheses~\eqref{e:coreflexive} through \eqref{e:Hconserve} 
and~\eqref{e:preserve} hold. Then the adjunction
\[
\xymatrix{
\Ms_{\Cc} \ar@/^/[r]^{\Hc_{\gamma}} & \Ns_{\cCc} \ar@/^/[l]^{\Kc^{\tgamma}}.
}
\]
of Proposition~\ref{p:lifting-conjugate} is an adjoint equivalence.
\end{corollary}
\begin{proof}
In this situation, the hypotheses of Proposition~\ref{p:lifting-conjugate} are satisfied,
so the adjunction exists. From Lemma~\ref{l:comp-colax}, we know that
$\Hc_{\gamma}=\QcC$.  Since the latter functor is an equivalence by Theorem~\ref{t:descent},
the adjunction is an adjoint equivalence.
\end{proof}


\subsection{Universal property of the conjugate comonad}\label{ss:universal}

We consider the following situation.
\[
\xymatrix{
\Ms \ar@(ul,dl)_{\Cc} \ar@/^/[r]^{\Hc} & \Ns \ar@/^/[l]^{\Kc} \ar@(u,r)^{\cCc} \ar@(d,r)_{\Dc}
}
\]
We discuss a universal property of the conjugate comonad $\cCc$ of Definition~\ref{d:conjugate-comonad} and the colax morphism $(\Hc,\gamma)$ of Lemma~\ref{l:conjugate-comonad}. Briefly, it states that any colax morphism from $\Cc$ to $\Dc$
factors through $\gamma$.

\begin{proposition}\label{p:universal}
Let $\Dc:\Ns\to\Ns$ be a functor and $\sigma:\Hc\Cc\to\Dc\Hc$ a transformation.
\begin{enumerate}[(i)]
\item There is a unique transformation
$
\sigma':\cCc\to\Dc
$
such that the following diagram commutes.
\begin{equation}\label{e:universal}
\begin{gathered}
\xymatrix@-5pt{
\Hc\Cc \ar[rr]^-{\sigma} \ar[rd]_-{\gamma} & &  \Dc\Hc\\
 & \cCc\Hc \ar@{-->}[ru]_-{\sigma'\Hc} & \\
}
\end{gathered}
\end{equation}
\item Moreover, if $\Ds$ is a comonad and
$
(\Hc,\sigma): (\Ms,\Cc)\to (\Ns,\Dc)
$
is a colax morphism of comonads, then $(\Ic,\sigma'): (\Ns,\cCc)\to (\Ns,\Dc)$ is another such morphism.
\end{enumerate}
\end{proposition}
\begin{proof} We use $(\delta,\epsilon)$ to denote the structure maps of all comonads.

We let $\sigma':\Hc\Cc\Kc\to\Dc$ be the composite
\begin{equation}\label{e:universal2}
\cCc=\Hc\Cc\Kc \map{\sigma\Kc} \Dc\Hc\Kc \map{\Dc\xi} \Dc\Ic=\Dc.
\end{equation}

We prove (i). Consider the following diagram.
\[
\xymatrix@C+15pt@R-5pt{
\Hc\Cc \ar[r]^-{\gamma} \ar@{=}[d] & \cCc\Hc \ar[rr]^-{\sigma'\Hc} \ar@{=}[d] & &\Dc\Hc \ar@{=}[d] \\
 \Hc\Cc \ar[r]^-{\Hc\Cc\eta} \ar[dr]_{\sigma} & \Hc\Cc\Kc\Hc \ar[r]^-{\sigma\Kc\Hc} & 
 \Dc\Hc\Kc\Hc \ar[r]^-{\Dc\xi\Hc} & \Dc\Hc\\
 & \Dc\Hc \ar[ur]_{\Dc\Hc\eta} \ar@/_1.3pc/@{=}[rru] & & 
}
\]
The top left rectangles commute by definition of $\gamma$~\eqref{e:def-sigma}
and $\sigma'$. The bottom diagrams commute by naturality and one of the adjunction axioms. The commutativity of the whole diagram results, and this is~\eqref{e:universal}. 
 
To complete the proof of (i), we verify the uniqueness of $\sigma'$. 
Assume another transformation $\sigma''$ makes~\eqref{e:universal}
commutative.
Consider the following diagram.

\[
\xymatrix@C+15pt@R-5pt{
\Hc\Cc\Kc\ar@{=}@/^2.2pc/[rr] \ar[r]^-{\Hc\Cc\eta\Kc}
\ar[dr]|{\gamma\Kc}  \ar[ddr]|{\sigma\Kc}
\ar@/_3.8pc/[ddrr]_-{\sigma'}
& \Hc\Cc\Kc\Hc\Kc \ar[r]^-{\Hc\Cc\Kc\xi} \ar@{=}[d] & \Hc\Cc\Kc  \ar@{=}[d] \\
& \cCc\Hc\Kc \ar[r]^-{\cCc\xi} \ar[d]^(.45){\sigma''\Hc\Kc} & \cCc \ar[d]^{\sigma''} \\
& \Dc\Hc\Kc \ar[r]^{\Dc\xi} & \Dc
}
\]
The top left triangle commutes by definition of $\gamma$~\eqref{e:def-sigma}.
The triangle below it commutes by assumption and the bottom piece by
definition of $\sigma'$. The top piece commutes by
one of the adjunction axioms and the remaining rectangles by naturality.
We deduce that $\sigma'=\sigma''$.

We turn to (ii). Assume that $(\Hc,\sigma): (\Ms,\Cc)\to (\Ns,\Dc)$
is a colax morphism of comonads. The fact that $\sigma'$ preserves counits (second diagram in~\eqref{e:lax-mor-comonad}) is the commutativity of the following diagram.
\[
\xymatrix@C+15pt@R-5pt{
\cCc= \Hc\Cc\Kc\ar@/^2.2pc/[rr]^-{\sigma'} \ar[r]^-{\sigma\Kc}
\ar[dr]^-{\Hc\epsilon\Kc} \ar[ddr]_{\epsilon} & \Dc\Hc\Kc \ar[r]^-{\Dc\xi} \ar[d]^(.4){\epsilon\Hc\Kc} &     \Dc \ar[ddl]^{\epsilon} \\
& \Hc\Kc\ar[d]^(.4){\xi} & \\
& \Ic & 
}
\]
The central triangle commutes by the second diagram in~\eqref{e:lax-mor-comonad} for $\sigma$. The left triangle is~\eqref{e:conjugate-cou} and the right one commutes
by naturality. The fact that $\sigma'$ preserves comultiplications (first diagram in~\eqref{e:lax-mor-comonad}) is the commutativity of the following diagram.
\[
\xymatrix@C+15pt@R-5pt{
\Hc\Cc\Kc\ar@/^2.2pc/[rrr]^-{\sigma'} \ar[rr]^-{\sigma\Kc}
\ar[d]^-{\Hc\delta\Kc}  \ar@/_2.2pc/[dd]_-{\delta}
& & \Dc\Hc\Kc \ar[r]^(.45){\Dc\xi} \ar[d]^{\delta\Hc\Kc} &     \Dc \ar[d]^{\delta} \\
\Hc\Cc\Cc\Kc \ar[r]^-{\sigma\Cc\Kc}
\ar[d]^-{\Hc\Cc\eta\Cc\Kc}  
& \Dc\Hc\Cc\Kc \ar[r]^-{\Dc\sigma\Kc} \ar[d]^{\Dc\Hc\eta\Cc\Kc} 
\ar@/_1.3pc/[rr]_-{\Dc\sigma'}
&    \Dc\Dc\Hc\Kc \ar[r]^-{\Dc\Dc\xi} & \Dc\Dc  \\
\Hc\Cc\Kc\Hc\Cc\Kc \ar[r]_{\sigma\Kc\cCc} \ar[dr]_{\sigma'\cCc} & 
\Dc\Hc\Kc\Hc\Cc\Kc \ar[d]^{\Dc\xi\cCc} & & \\
& \Dc\cCc \ar[rruu]_-{\Dc\sigma'}
}
\]
Above, the top left rectangle commutes by the first diagram in~\eqref{e:lax-mor-comonad} for $\sigma$.  The diagram in the bottom right commutes since the
vertical composition is the identity (by one of the adjunction axioms). The remaining pieces
commute either by naturality or by definition of $\sigma'$ or $\delta$. This
 completes
the proof of (ii) and of the proposition.
\end{proof}

\begin{remark}
The commutativity of~\eqref{e:universal} can be formulated in terms
of compositions of colax morphisms of comonads. It states that diagram
\begin{equation}\label{e:universal3}
\begin{gathered}
\xymatrix@-5pt{
 (\Ms, \Cc) \ar[rr]^-{(\Hc,\sigma)} \ar[rd]_{(\Hc,\gamma)} & & (\Ns,\Dc)\\
 & (\Ns,\cCc) \ar[ru]_{(\Ic,\sigma')} & 
 }
 \end{gathered}
\end{equation}
commutes
\end{remark}


\subsection{Equivalence between categories of coalgebras}\label{ss:equivalence-coalgebra}

We discuss necessary and sufficient conditions under which a colax
morphism of comonads gives rise to an equivalence of coalgebra categories.
In Proposition~\ref{p:equivalence-coalgebra}
we treat the special case of comonads on the same category; 
the general case is dealt with in the next section.

Lemmas~\ref{l:creation-fibration}--\ref{l:fibration-equivalence} contain preparatory material. We are indebted to Ignacio Lopez-Franco for help with these results.
They appear to be known, but we have included short proofs for convenience.  

A functor $\Fc:\Cs\to\Ds$ is a \emph{discrete isofibration}
if for any isomorphism $f:A\to\Fc(B)$ in $\Ds$, there exists (a unique object $\Tilde{A}$ and)
a unique isomorphism $\Tilde{f}:\Tilde{A}\to B$ in $\Cs$ such that $\Fc(\Tilde{f})=f$
(and $\Fc(\Tilde{A})=A$).
In this situation, we say that $\Tilde{f}$ is a \emph{lifting} of $f$.

\begin{lemma}\label{l:creation-fibration}
If a functor $\Fc$ creates isomorphisms, it is a discrete isofibration.
\end{lemma}
\begin{proof}
We clarify the hypothesis. An isomorphism $f:A\to B$ in $\Cs$
is a limiting cone of the functor that sends the object of the unit category to 
the object $B$ in $\Cs$.
Then $\Fc$ creates this type of limit if and only if given an isomorphism
$f:A\to\Fc(B)$ in $\Ds$, there exists a unique morphism $\Tilde{f}:\Tilde{A}\to B$ in $\Cs$ such that $\Fc(\Tilde{f})=f$, and moreover, $\Tilde{f}$ is an isomorphism.
This is stronger than saying that $\Fc$ is a discrete isofibration.
\end{proof}

\begin{lemma}\label{l:fibration}
 Let $\Fc:\Cs\to\Ds$ and $\Uc:\Ds\to\Es$ be functors such that
 both $\Uc$ and $\Uc\Fc$ are discrete isofibrations. Then $\Fc$ is another
 discrete isofibration.
\end{lemma}
\begin{proof}
Given $f:A\to\Fc(B)$ in $\Ds$, let $\Tilde{f}:\Tilde{A}\to B$ in $\Cs$ be the unique lift
of $\Uc(f)$ under $\Uc\Fc$. Then both $f$ and $\Fc(\Tilde{f})$ are lifts of $\Uc(f)$ under $\Uc$,
so $f=\Fc(\Tilde{f})$ by uniqueness. This proves that liftings under $\Fc$ exist. Uniqueness is similar.
\end{proof}

\begin{lemma}\label{l:fibration-equivalence}
Let $\Fc$ be a discrete isofibration. If $\Fc$ is part of an equivalence, then
it is part of an isomorphism.
\end{lemma}
\begin{proof}
Since $\Fc$ is part of an equivalence, it is full, faithful, and essentially
surjective on objects~\cite[Theorem IV.4.1]{Mac:1998}. 

Let $A$ and $B$ be objects of $\Cs$
such that $\Fc(A)=\Fc(B)$. Since $\Fc$ is full and faithful, there is an isomorphism
$f:A\to B$ such that $\Fc(f)=\id_{\Fc(B)}$. Then $f$ is a lifting of $\id_{\Fc(B)}$. 
By uniqueness, $f$ must be $\id_{B}$, and so $A=B$. 
Thus $\Fc$ is injective on objects.

Let $A$ be an object of $\Ds$. Since $\Fc$ is essentially surjective on objects,
there is an object $B$ of $\Cs$ and an isomorphism $f:A\to\Fc(B)$. Let
$\Tilde{f}:\Tilde{A}\to B$ be a lifting of $f$. Then $\Fc(\Tilde{A})=A$.
Thus $\Fc$ is surjective on objects.

Since $\Fc$ is full, faithful, and bijective on objects, it is part of an isomorphism.
\end{proof}

We return to coalgebra categories.  In the proof below, we apply Lemma~\ref{l:fibration-equivalence} to the case in which the functor $\Fc$ is as in Lemma~\ref{l:functor-transformation}.  Note that, in that case, the condition that $\Fc$ be bijective on objects means simply the following. Given any object $(N,b)$ of $\Ns_{\Dc}$, there is a unique morphism $a: N\to\Cc(N)$ such that $(N,a)$ is an object of $\Ns_{\Cc}$ and $\Fc(N,a) = (N,b)$.  

\begin{proposition}\label{p:equivalence-coalgebra}
Let $\Dc$ and $\Ec$ be two comonads on the same
category $\Ns$.
Let $\rho:\Ec\to\Dc$ be a transformation such that $(\Ic,\rho):(\Ns,\Ec)\to(\Ns,\Dc)$ is a colax morphism of comonads. Consider the induced functor 
$\Ic_{\rho}: \Ns_{\Ec} \to \Ns_{\Dc}$ of Section~\ref{ss:cores}.
Then:
\[
\text{$\Ic_{\rho}$ is an equivalence $\iff$ $\rho$ is invertible.}
\]
\end{proposition}
\begin{proof}
If $\rho$ is invertible, then $\Ic_{\rho}$ is an equivalence with inverse
$\Ic_{\rho^{-1}}$, by functoriality. 

We prove the converse.
By Lemma~\ref{l:creation},
the forgetful functors $\Uc_{\Dc}$ and $\Uc_{\Ec}$
create isomorphisms; hence,
by Lemma~\ref{l:creation-fibration}, they are discrete isofibrations.
According to~\eqref{e:cores-forget-id}, we have $\Uc_{\Dc}\Ic_{\rho}=\Uc_{\Ec}$.
It follows from Lemma~\ref{l:fibration} that $\Ic_{\rho}$ is a discrete isofibration.
By Lemma~\ref{l:fibration-equivalence}, $\Ic_{\rho}$ is in fact an isomorphism,
being an equivalence.
Its inverse satisfies $\Uc_{\Dc}=\Uc_{\Ec}(\Ic_{\rho})^{-1}$.
It follows from Lemma~\ref{l:functor-transformation}
that $(\Ic_{\rho})^{-1}=\Ic_{\sigma}$ for some $\sigma:\Dc\to\Ec$. 
By uniqueness in Lemma~\ref{l:functor-transformation}, $\sigma$ and $\rho$ are inverses.
\end{proof}

In the presence of additional hypotheses, an alternative proof of
Proposition~\ref{p:equivalence-coalgebra} is available. Assume that:
\begin{itemize}
\item The category $\Ns$ has equalizers of all coreflexive pairs.
\item The comonad $\Ec$ preserves them.
\end{itemize}

\begin{aproof}
The hypotheses on $\Ns$ and $\Ec$ imply that $\Ns_{\Ec}$ has all equalizers
of coreflexive pairs, in view of Lemma~\ref{l:creation}. This
allows us to apply the Adjoint Lifting Theorem (Theorem~\ref{t:adjoint})
to the comonads $\Ec$ and $\Dc$, the trivial adjunction on $\Ns$ ($\Hc=\Kc=\Ic_{\Ns}$), and the colax morphism of comonads $(\Ic,\rho)$. We obtain an adjunction
\[
\xymatrix{
\Ns_{\Ec} \ar@/^/[r]^{\Ic_{\rho}} & \Ns_{\Dc} \ar@/^/[l]^{\Ic^{\trho}} 
}
\]
where $\trho$ is the mate of $\rho$. 
(We have $\trho=\rho$, but this is not needed for the rest of the argument.)
We know from~\eqref{e:coind-cofree} that the diagram
\[
\xymatrix@-5pt{
& \Ns \ar[dl]_{\Fc_\Ec} \ar[rd]^{\Fc_\Dc} & \\
\Ns_\Ec  & & \Ns_\Dc \ar[ll]^-{\Ic^{\trho}}
}
\]
commutes up to isomorphism. 
Now, if $\Ic_{\rho}$ is an equivalence, then $(\Ic_{\rho})^{-1}$ must be isomorphic to
$\Ic^{\trho}$, by uniqueness of adjoints. Hence the diagram
\[
\xymatrix@-5pt{
& \Ns \ar[dl]_{\Fc_\Ec} \ar[rd]^{\Fc_\Dc} & \\
\Ns_\Ec \ar[rr]_-{\Ic_{\rho}} & & \Ns_\Dc 
}
\]
commutes up to isomorphism. Note this is~\eqref{e:cores-cofree}.
But then, as recalled in Section~\ref{ss:cores},
the transformation $\rho$ is invertible. \qed
\end{aproof}

\subsection{Equivalence between categories of coalgebras: the general case}
\label{ss:equivalence-coalgebra-gen}

We now determine when a colax morphism of comonads from $(\Ms,\Cc)$ to $(\Ns,\Dc)$
gives rise to an equivalence from $\Cc$-coalgebras to $\Dc$-coalgebras.
Roughly, this happens when $\Dc$ is isomorphic to the conjugate comonad $\cCc$.
The two theorems that follow provide the precise statements. For both of them,
we consider the following situation.
\[
\xymatrix{
\Ms \ar@(ul,dl)_{\Cc} \ar@/^/[r]^{\Hc} & \Ns \ar@/^/[l]^{\Kc}  \ar@(ur,dr)^{\Dc}
}
\]
Thus, $\Cc$ is a comonad on $\Ms$, $\Dc$ is a comonad on $\Ns$, and $(\Hc,\Kc)$
is an adjunction.
Also, let $\sigma:\Hc\Cc\to\Dc\Hc$ be a transformation
such that $(\Hc,\sigma): (\Ms,\Cc)\to (\Ns,\Dc)$
is a colax morphism of comonads. Let $\sigma':\cCc\to\Dc$ be the transformation
afforded by the universal property of the conjugate comonad
$\cCc$, as discussed in Proposition~\ref{p:universal}.
Finally, let $\tau$ be the mate of $\sigma$, as in Section~\ref{ss:mates}.

\begin{theorem}\label{t:equivalence-coalgebra}
In the above situation, assume that:
\begin{itemize}
\item The category $\Ms$ and the functor $\Hc$ satisfy hypotheses~\eqref{e:coreflexive}--\eqref{e:Hconserve}.
\item The functor $\Cc$ satisfies hypothesis~\eqref{e:preserve}.
\end{itemize}
The functors
\[
\xymatrix{
\Ms_{\Cc} \ar@/^/[r]^{\resH} & \Ns_{\Dc} \ar@/^/[l]^{\indK} 
}
\]
of Sections~\ref{ss:cores}--\ref{ss:coind} form an adjoint equivalence if and only if
the transformation $\sigma':\cCc\to\Dc$ is invertible.
\end{theorem}
\begin{proof} 
Let $\gamma$ be the transformation in~\eqref{e:def-sigma}.
The hypotheses on $\Ms$, $\Hc$ and $\Cc$ ensure that Corollary~\ref{c:descent}
applies, and hence the induced functor $\Hc_{\gamma}$ is an equivalence. 
On the other hand, from~\eqref{e:universal3}
we have the following commutative diagram.
 \[
 \xymatrix@-5pt{
 \Ms_{\Cc} \ar[rr]^-{\resH} \ar[rd]_{\Hc_{\gamma}} & & \Ns_{\Dc}\\
 & \Ns_{\cCc} \ar[ru]_{\Ic_{\sigma'}} & 
 }
 \]
Hence,
\[
\text{$\resH$ is an equivalence $\iff$
$\Ic_{\sigma'}$ is an equivalence $\iff$
$\sigma'$ is invertible,} 
\]
the latter by Proposition~\ref{p:equivalence-coalgebra}.

It remains to verify that $\indK$ is a right adjoint of $\resH$. This follows from Theorem~\ref{t:adjoint}, provided we show that $\Dc$ satisfies hypothesis~\eqref{e:existence}. Now, as in the proof of Proposition~\ref{p:lifting-conjugate}, 
the comonad $\cCc$ satisfies hypothesis~\eqref{e:existence}. It is easy to
see that then the same is true for the isomorphic comonad $\Dc$.
\end{proof}

It is useful to state this result under simpler (though stronger) hypotheses.

\begin{theorem}\label{t:equivalence-coalgebra-conv}
In the above situation, assume that:
\begin{itemize}
\item The category $\Ms$ has equalizers of all coreflexive pairs.
\item The functors $\Cc$ and $\Hc$ preserve these equalizers.
\item The functor $\Hc$ is conservative.
\end{itemize}
Then the conclusion of Theorem~\ref{t:equivalence-coalgebra} holds.
\end{theorem}
\begin{proof}
Since coreflexive pairs are preserved by arbitrary functors,
their equalizers are preserved not only by $\Cc$ but also by $\Cc^2$.
Thus the given hypotheses are stronger than those of Theorem~\ref{t:equivalence-coalgebra}.
\end{proof}

%

We close the section by examining the extent to which the hypotheses
on the functor $\Hc$ are necessary for $\resH$ to be an equivalence.

\begin{proposition}\label{p:equivalence-coalgebra-conv}
Still in the above situation, consider the following hypotheses.
\begin{enumerate}[(i)]
\item The functor $\resH: \Ms_{\Cc}\to\Ns_{\Dc}$ is an equivalence.
\item The functors $\Cc$ and $\Dc$ preserve all existing equalizers in $\Ms$ and $\Ns$, respectively.
\item The transformation $\sigma':\cCc\to\Dc$ is invertible.
\end{enumerate}
We then have the following statements.
\begin{itemize}
\item If \textup{(i)} holds, then the functor $\Hc\Uc_\Cc$ is conservative.
\item Assume \textup{(i)} and \textup{(ii)} hold.
 Let $(f,g)$ be a parallel pair in $\Ms_{\Cc}$ such that the pair
$\bigl(\Uc_\Cc(f),\Uc_\Cc(g)\bigr)$ has an equalizer in $\Ms$ and
the pair $\bigl(\Hc\Uc_\Cc(f),\Hc\Uc_\Cc(g)\bigr)$ has an equalizer in $\Ns$.
Then the functor $\Hc$ preserves the equalizer of $\bigl(\Uc_\Cc(f),\Uc_\Cc(g)\bigr)$.
\item Assume \textup{(i)} and \textup{(iii)} hold. Then any parallel pair in $\Ms_{\Cc}$ which is coreflexive and $\Hc\Uc_{\Cc}$-split
has an equalizer in $\Ms_{\Cc}$.
\end{itemize}
\end{proposition}
\begin{proof}
From~\eqref{e:cores-forget}, $\Hc\Uc_\Cc=\Uc_\Dc\resH$. 
The first statement
then follows since $\Uc_\Dc$ is conservative (Section~\ref{ss:adj-comonad})

We turn to the second statement. 
By hypothesis, 
\[
\bigl(\Uc_\Dc\resH(f),\Uc_\Dc\resH(g)\bigr)=\bigl(\Hc\Uc_\Cc(f),\Hc\Uc_\Cc(g)\bigr)
\] 
has an equalizer in $\Ns$.
Since $\Dc$ preserves all existing equalizers, the same is true of $\Dc^2$.
Hence, by Lemma~\ref{l:creation}, $\bigl(\resH(f),\resH(g)\bigr)$ has an equalizer in $\Ns_\Dc$ and this is preserved by $\Uc_\Dc$. Then, $(f,g)$ has an equalizer in $\Ns$
and this is preserved by $\resH$, since $\resH$ is an equivalence.

As for $\Dc$ above, the hypothesis on $\Cc$ ensures that Lemma~\ref{l:creation}
applies to it. Hence
the equalizer of 
$\bigl(\Uc_\Cc(f),\Uc_\Cc(g)\bigr)$ is the image under $\Uc_\Cc$ of the
equalizer of $(f,g)$. Therefore, applying $\Hc$ to the former equalizer
is the same as applying $\Hc\Uc_\Cc=\Uc_\Dc\resH$ to the latter, which
as explained above is the equalizer of $\bigl(\Uc_\Dc\resH(f),\Uc_\Dc\resH(g)\bigr)$.
This proves the second statement.

Finally, as in the proofs of the previous theorems, we have
$\resH = \Ic_{\sigma'} \Hc_{\gamma}= \Ic_{\sigma'} \QcC$. Hypotheses (i) and (iii) imply that $\QcC$
is an equivalence. Applying Theorem~\ref{t:beck} (Beck's) to the adjunction~\eqref{e:comp-adj}, we deduce that~\eqref{e:coreflexive} holds for this adjunction, and
this is the third statement.
\end{proof}

\section{Bimonads and comodule-monads}\label{s:bimonad}

Comonoidal functors and their comodules are reviewed in Sections~\ref{ss:comonoidal} and~\ref{ss:comod-functor}.
The former are functors $\Fc:\Cs\to\Ds$ between monoidal categories which
are compatible with the monoidal structures in such a way that comonoids are
preserved. 
The latter are
functors $\Hc:\Ms\to\Ns$ between module-categories
(categories on which $\Cs$ and $\Ds$ act, respectively) which are such
that if $M$ is a comodule over a comonoid $C$, then $\Hc(M)$ is a comodule
over $\Fc(C)$.

A bimonad is a monad with a compatible comonoidal structure.
Sections~\ref{ss:bimonad} and~\ref{ss:comod-monad} review this notion (introduced by Moerdijk~\cite{Moe:2002})
and that of comodule-monad over a bimonad. Another important notion reviewed 
in Section~\ref{ss:bimonad} is that
of an algebra-comonoid over a bimonad $\Tc$. These objects may be seen 
either as $\Tc$-algebras in the category of comonoids or as comonoids in the
category of $\Tc$-algebras.

Section~\ref{ss:hopf-mod} introduces generalized Hopf modules.
This is a central notion in the paper. It depends on three ingredients: a bimonad $\Tc$,
a comodule-monad $\Sc$ over $\Tc$, and an algebra-comonoid $Z$ over $\Tc$.
For a special choice of these data, generalized Hopf modules become precisely
the Hopf modules of Brugui\`eres, Lack and Virelizier~\cite[Section~4.2]{BruVir:2007}
and~\cite[Section~6.5]{BLV:2011} (see Remark~\ref{r:hopf-mod}).
Generalized Hopf modules may be seen either as $\Sc$-algebras in the category
of $Z$-comodules, or as $Z$-comodules in the category of $\Sc$-algebras
(Proposition~\ref{p:hopf-mod}).

We use the symbol $\bdot$ to denote the
monoidal operation of all monoidal categories, and $\unit$ for 
the unit object. We also use $\bdot$ to
denote the action of a monoidal category on a (left) module-category.
For the axioms defining module-categories, see for instance~\cite[Section~1]{JK:2001}.

\subsection{Comonoidal functors}\label{ss:comonoidal}

Let $\Cs$ and $\Ds$ be monoidal categories.

\begin{definition}\label{d:comonoidal}
A \emph{comonoidal functor}
$(\Fc,\psi,\psi_0):(\Cs,\bdot,\unit) \to(\Ds,\bdot,\unit)$ consists of a functor
\[
\Fc:\Cs\to\Ds,
\]
a natural transformation
\[
\psi_{X,Y}: \Fc(X\bdot Y) \to \Fc(X)\bdot \Fc(Y)
\]
of functors $\Cs\times\Cs\to\Cs$, and a map
\[
\psi_0: \Fc(\unit)\to\unit,
\]
subject to certain axioms; see for instance~\cite[Definition~3.2]{AguMah:2010}.
\end{definition}

Comonoidal functors are also called \emph{colax monoidal functors}.

Let $\Com(\Cs)$ denote the category of comonoids in $\Cs$.
Comonoidal functors preserve comonoids. 
Let $C$ be a comonoid in $\Cs$. We use
\[
\Delta:C\to C\bdot C
\qand
\epsilon:C\to\unit
\]
to denote its structure. Then $\Fc(C)$ is a comonoid in $\Ds$ with
structure
\begin{equation}\label{e:coprod-colax}
\begin{gathered}
D=\Fc(C)\map{\Fc(\Delta)} \Fc(C\bdot C) \map{\psi_{C,C}} \Fc(C)\bdot\Fc(C)=D\bdot D
\\
D=\Fc(C)\map{\Fc(\epsilon)} \Fc(\unit)\map{\psi_0} \unit.
\end{gathered}
\end{equation}
In this manner, $\Fc$ induces a functor
\begin{equation}\label{e:comon-transfer}
\Com(\Cs)\to\Com(\Ds)
\end{equation}
which we also denote by $\Fc$.

\subsection{Bimonads}\label{ss:bimonad}

\begin{definition}\label{d:bimonad}
A \emph{bimonad on $\Cs$}  is a monad $\Tc:\Cs\to\Cs$ that is in addition a comonoidal functor, in such a way that the monad structure maps are morphisms of comonoidal functors. 
\end{definition}

For more details, see~\cite[Definition~1.1]{Moe:2002},~\cite[Section~2.3]{BruVir:2007}, or~\cite[Section~2.4]{BLV:2011}.

\begin{remark}
Bimonads are also called  \emph{opmonoidal monads}~\cite{McC:2002,Szl:2003}
or \emph{comonoidal monads}.
They were originally called \emph{Hopf monads}~\cite{Moe:2002}.
In more recent work~\cite{BLV:2011,BruVir:2007}, the latter term is reserved for certain special bimonads.
\end{remark}

We use
\[
\mu:\Tc^2\to\Tc \qand \iota:\Ic\to\Tc
\]
to denote the monad structure maps and
\[
\psi_{X,Y}:\Tc(X\bdot Y) \to \Tc(X)\bdot\Tc(Y)
\]
for the comonoidal structure of the bimonad $\Tc$.

Let $\Cs^\Tc$ the category of algebras over the monad $\Tc$~\cite[Section~VI.2]{Mac:1998} (or Section~\ref{ss:coalgebra}).
The objects are pairs $(A,a)$ where $A$ is an object of $\Cs$
and $a:\Tc(A)\to A$ is a morphism that is associative and unital.

The monoidal structure of $\Cs$ is inherited by $\Cs^\Tc$~\cite[Proposition~1.4]{Moe:2002}. This makes use
of the comonoidal structure of the bimonad $\Tc$.
Explicitly, if $(A,a)$ and $(B,b)$ are two $\Tc$-algebras, then so is $(A\bdot B,c)$ with
\[
c:\Tc(A\bdot B) \map{\psi_{A,B}} \Tc(A)\bdot\Tc(B) \map{a\bdot b} A\bdot B.
\]
Also, the unit object of $\Cs$ is a $\Tc$-algebra with
\[
\Tc(\unit)\map{\psi_0} \unit.
\]

As in~\eqref{e:coprod-colax}, the functor $\Tc$ restricts to $\Com(\Cs)$, and 
in fact induces a monad
on this category~\cite[Proposition~2.1]{Moe:2002}, which we also denote by $\Tc$.

Algebras for the monad $\Tc$ on the category of comonoids coincide with
comonoids in the monoidal category of $\Tc$-algebras~\cite[Proposition~2.2]{Moe:2002}. More precisely, there is an isomorphism of categories
\[
\Com(\Cs)^{\Tc} = \Com(\Cs^{\Tc}).
\]

\begin{definition}\label{d:comTalg}
An object of the previous category is called a \emph{$\Tc$-algebra-comonoid}.
\end{definition}
(This differs from the terminology in~\cite[Definition~2.3]{Moe:2002}.)

Explicitly, a $\Tc$-algebra-comonoid $(Z,\action,\Delta,\epsilon)$ consists of an object $Z$ of $\Cs$ together with arrows
\[
\action:\Tc(Z)\to Z,\quad \Delta:Z\to Z\bdot Z \qand \epsilon:Z\to\unit
\]
such that $(Z,\action)$ is a $\Tc$-algebra, $(Z,\Delta,\epsilon)$ is a comonoid, and diagrams
\begin{equation}\label{e:comTalg}
\begin{gathered}
\xymatrix{
\Tc(Z) \ar[rr]^-{\action} \ar[d]_{\Tc(\Delta)} & & Z \ar[d]^{\Delta} \\
\Tc(Z\bdot Z) \ar[r]_-{\psi_{Z,Z}} & \Tc(Z)\bdot\Tc(Z) \ar[r]_-{\action\bdot \action} & Z\bdot Z
}
\end{gathered}
\qand
\begin{gathered}
\xymatrix{
\Tc(Z) \ar[r]^-{\action} \ar[d]_{\Tc(\epsilon)} &  Z \ar[d]^{\epsilon} \\
\Tc(\unit) \ar[r]_-{\psi_0} & \unit
}
\end{gathered}
\end{equation}
commute. These diagrams express the fact that $\Delta$ and $\epsilon$ are
morphisms of $\Tc$-algebras, or equivalently that $\action$ is  morphism of comonoids.

Let $C$ be an arbitrary comonoid in $\Cs$. Since $\Tc$ is a monad on $\Com(\Cs)$,
the comonoid $Z=\Tc(C)$ carries a canonical structure of $\Tc$-algebra which turns it
into a $\Tc$-algebra-comonoid. Explicitly, the structure is
\begin{equation}\label{e:freecomTalg}
\begin{gathered}
\Tc(Z)=\Tc^2(C)\map{\mu_C}\Tc(C)=Z\\
Z = \Tc(C)\map{\Tc(\Delta)} \Tc(C\bdot C) \map{\psi_{C,C}} \Tc(C)\bdot\Tc(C)=Z\bdot Z\\
Z = \Tc(C)\map{\Tc(\epsilon)} \Tc(\unit) \map{\psi_0} \unit.
\end{gathered}
\end{equation}
We refer to $\Tc(C)$ as the \emph{free} $\Tc$-algebra-comonoid on $C$.

\subsection{Comodules over comonoidal functors}\label{ss:comod-functor}

Let $(\Fc,\psi):(\Cs,\bdot) \to(\Ds,\bdot)$ be a comonoidal functor (Definition~\ref{d:comonoidal}). Let $\Ms$ and $\Ns$ be module-categories over $\Cs$ and $\Ds$,
respectively. The action of the monoidal categories is from the left.

\begin{definition}\label{d:comodule}
A \emph{comodule} over $(\Fc,\psi)$ is a pair $(\Hc,\chi)$ where
\[
\Hc:\Ms\to\Ns
\]
is a functor and
\[
\chi_{X,M}: \Hc(X\bdot M) \to \Fc(X)\bdot \Hc(M)
\]
is a natural transformation of functors $\Cs\times\Ms\to\Ns$
such that diagrams
\begin{gather}\label{e:comodule1}
\begin{gathered}
\xymatrix@C+25pt{
\Hc(X\bdot Y\bdot M) \ar[r]^-{\chi_{X\bdot Y,M}} \ar[d]_{\chi_{X,Y\bdot M}} & 
\Fc(X\bdot Y)\bdot\Hc(M)  \ar[d]^{\psi_{X,Y}\bdot\id_{\Hc(M)}}\\
\Fc(X)\bdot\Hc(Y\bdot M) \ar[r]_-{\id_{\Fc(X)}\bdot \chi_{Y,M}} & \Fc(X)\bdot\Fc(Y)\bdot\Hc(M)
}
\end{gathered}
\\
\label{e:comodule2}
\begin{gathered}
\xymatrix@C+5pt{
\Hc(\unit\bdot M) \ar[r]^-{\chi_{\unit,M}} \ar@{=}[d] & 
\Fc(\unit)\bdot \Hc(M)  \ar[d]^{\psi_{0}\bdot\id_{\Hc(M)}}\\
\Hc(M) \ar@{=}[r] & \unit\bdot\Hc(M)
}
\end{gathered}
\end{gather}
commute for all objects $X,Y$ of $\Cs$ and $M$ of $\Ms$.
\end{definition}

Let $(\Hc,\chi)$ and $(\Hc',\chi')$ be two comodules over $(\Fc,\psi)$.
A \emph{morphism of comodules} is a natural transformation $f:\Hc\to\Hc'$ such that
\[
\xymatrix@C+10pt{
\Hc(X\bdot M) \ar[r]^-{\chi_{X,M}} \ar[d]_{f_{X\bdot M}} & 
\Fc(X)\bdot \Hc(M)  \ar[d]^{\id_{\Fc(X)}\bdot f_M}\\
\Hc'(X\bdot M) \ar[r]_-{\chi'_{X,M}} & \Fc(X)\bdot\Hc'(M)
}
\]
commutes for all $X$ of $\Cs$ and $M$ of $\Ms$.

Comodules can be composed. First recall that if 
\[
\Fc:\Cs\to\Ds \qand \Fc':\Ds\to\Es
\]
are comonoidal functors (with structure $\psi$ and $\psi'$), then their composite
\[
\Fc'\Fc:\Cs\to\Es
\] 
is comonoidal with structure transformation
\[
\Fc'\Fc(X\bdot Y) \map{\Fc'(\psi_{X,Y})} \Fc'\bigl(\Fc(X)\bdot\Fc(Y)\bigr) 
\map{\psi'_{\Fc(X),\Fc(Y)}} \Fc'\Fc(X) \bdot \Fc'\Fc(Y)
\]
and counit structure map
\[
\Fc'\Fc(\unit) \map{\Fc'(\psi_0)} \Fc'(\unit) \map{\psi'_0} \unit.
\]
For details, see for instance~\cite[Theorem~3.21]{AguMah:2010}. Now suppose
that 
\[
\Hc:\Ms\to\Ns \qand \Hc':\Ns\to\Ps
\]
are comodules over $\Fc$ and $\Fc'$ respectively (with transformations $\chi$ and $\chi'$). Then $\Hc'\Hc$ is a comodule over $\Fc'\Fc$ with transformation
\begin{equation}\label{e:comod-compose}
\Hc'\Hc(X\bdot M) \map{\Hc'(\chi_{X,M})} \Hc'\bigl(\Fc(X)\bdot\Hc(M)\bigr) 
\map{\chi'_{\Fc(X),\Hc(M)}} \Fc'\Fc(X) \bdot \Hc'\Hc(M).
\end{equation}

\subsection{Comodules over comonoids}\label{ss:comod-comonoid}

Let $C$ be a comonoid in $\Cs$. A $C$-comodule $(M,d)$ consists of
an object $M$ in $\Ms$ and an arrow 
\[
d:M\to C\bdot M
\]
subject to the usual associativity and counit axioms. 

Let $\hC:\Ms\to\Ms$ denote the left action by $C$ on the category $\Ms$:
on an object $M$ in $\Ms$,
\begin{equation}\label{e:hat-comonad}
\hC(M) = C\bdot M.
\end{equation}
The functor $\hC$ is a comonad on $\Ms$ and
the category $\Ms_{\hC}$ of coalgebras for this comonad is precisely the
category of left $C$-comodules in $\Ms$.

Let $M$ be a $C$-comodule. Given a comonoidal functor $\Fc$ and an $\Fc$-comodule $\Hc$
(Definition~\ref{d:comodule}), the object $\Hc(M)$ is an $\Fc(C)$-comodule with
\[
\Hc(M) \map{\Hc(d)} \Hc(C\bdot M) \map{\chi_{C,M}} \Hc(C)\bdot\Hc(M).
\]
Compare with~\eqref{e:coprod-colax}. In this manner, $\Hc$ induces a functor
\begin{equation}\label{e:comod-transfer}
\Ms_{\hC} \to \Ns_{\widehat{\Fc(C)}}
\end{equation}
which we also denote by $\Hc$.

By employing functors from the one-arrow category, one may view
the transformations of comonoids under comonoidal functors
(and of comodules over comonoids
under comodules over comonoidal functors)
as an instance of the composition of comonoidal functors
(and of comodules over them) discussed in Section~\ref{ss:comod-functor}.

\subsection{Comodules over bimonads}\label{ss:comod-monad}

Let $\Tc$ be a bimonad on a monoidal category $\Cs$.
Let $\Ms$ be a $\Cs$-module-category and let $\Sc:\Ms\to\Ms$ be a monad. We use $\mu$ and $\iota$ to denote the monad structure of either $\Tc$ or $\Sc$. 
Suppose that $\Sc$ is a comodule over $(\Tc,\psi)$ via 
\[
\chi_{X,M}:\Sc(X\bdot M) \to \Tc(X)\bdot \Sc(M).
\]

\begin{definition}\label{d:comod-monad}
We say that $(\Sc,\chi)$ is a \emph{comodule-monad} over the bimonad $(\Tc,\psi)$
if the following diagrams commute for all $X$ in $\Cs$ and $M$ in $\Ms$.
\begin{gather}
\label{e:comod-monad1}
\begin{gathered}
\xymatrix@C+30pt{
\Sc^2(X\bdot M) \ar[d]_{\mu_{X\bdot M}} \ar[r]^-{\Sc(\chi_{X,M})}  & 
\Sc\bigl(\Tc(X)\bdot\Sc(M)\bigr) \ar[r]^-{\chi_{\Tc(X),\Sc(M)}}  &
\Tc^2(X)\bdot \Sc^2(M) \ar[d]^{\mu_X\bdot\mu_M}\\
\Sc(X\bdot M) \ar[rr]_-{\chi_{X,M}} & & \Tc(X)\bdot \Sc(M)
}
\end{gathered}\\
\label{e:comod-monad2}
\begin{gathered}
\xymatrix@C-25pt{
& X\bdot M \ar[dl]_{\iota_{X\bdot M}}  \ar[rd]^{\iota_{X}\bdot \iota_{M}} & \\
\Sc(X\bdot M) \ar[rr]_-{\chi_{X,M}}  & & \Tc(X)\bdot\Sc(M) 
}
\end{gathered}
\end{gather}
\end{definition}

As discussed in Section~\ref{ss:comod-functor}, $\Sc^2:\Ms\to\Ms$ (the composite of $\Sc$ with itself) is a comodule over $\Tc^2:\Cs\to\Cs$. In addition, since $\mu:\Tc^2\to\Tc$
is a morphism of comonoidal functors, we may turn $\Sc^2$ into a comodule over $\Tc$. Similarly, the morphism $\iota:\Ic\to\Tc$ allows us to turn 
the identity functor $\Ic:\Ms\to\Ms$ into a comodule over $\Tc$. 
The conditions in Definition~\ref{d:comod-monad} may then be reformulated as follows:  $\mu:\Sc^2\to\Sc$ and $\iota:\Ic\to\Sc$ are morphisms of comodules over $(\Tc,\psi)$.

\begin{remark}
Let 
\[
\Hc,\Hc': \Ms\to\Ms
\]
be comodules over $\Tc$. As discussed in Section~\ref{ss:comod-functor}, the composite $\Hc'\Hc:\Ms\to\Ms$ is then a comodule over $\Tc^2:\Cs\to\Cs$. In addition, since $\mu:\Tc^2\to\Tc$
is a morphism of comonoidal functors, we may turn $\Hc'\Hc$ into a comodule over $\Tc$. 
As mentioned above, $\Ic:\Ms\to\Ms$ is a comodule over $\Tc$.
These statements may be summarized as follows.
Let $\End(\Ms)$ be the category of endofunctors of $\Ms$. It is strict monoidal under
composition. Let $\End_{\Tc}(\Ms)$ denote the category of comodules over $(\Tc,\psi)$. Then $\End_{\Tc}(\Ms)$ is strict monoidal in such a way that the forgetful functor
$\End_{\Tc}(\Ms)\to \End(\Ms)$ is strict monoidal.

It follows from Definition~\ref{d:comod-monad} that $(\Sc,\chi)$ is a
comodule-monad over $(\Tc,\psi)$ if and only if it is a  monoid in the monoidal category
$\End_{\Tc}(\Ms)$.
\end{remark}

Recall (Section~\ref{ss:bimonad}) that the category of $\Tc$-algebras is monoidal.
There is an action of $\Cs^\Tc$ on $\Ms^\Sc$:
given a $\Tc$-algebra $(A,a)$ and a $\Sc$-algebra $(B,b)$, define
\begin{equation}\label{e:alg-cat}
c:\Sc(A\bdot B) \map{\chi_{A,B}} \Tc(A)\bdot\Sc(B) \map{a\bdot b} A\bdot B.
\end{equation}

\begin{proposition}\label{p:alg-cat}
$(A\bdot B,c)$ is an $\Sc$-algebra.
\end{proposition}
\begin{proof}
Associativity for $c$ holds by the commutativity of the following diagram.
\[
\xymatrix@C+15pt{
\Sc^2(A\bdot B) \ar[rr]^-{\mu_{A\bdot B}} \ar[d]_{\Sc(\chi_{A,B})} 
\ar@/_4.5pc/[dd]_-{\Sc(c)} & & \Sc(A\bdot B) \ar[d]^{\chi_{A,B}} \ar@/^3.5pc/[dd]^-{c} \\
\Sc\bigl(\Tc(A)\bdot\Sc(B)\bigr) \ar[r]^-{\chi_{\Tc(A),\Sc(B)}} \ar[d]_{\Sc(a\bdot b)} &
\Tc^2(A)\bdot \Sc^2(B) \ar[r]^-{\mu_A\bdot\mu_B} \ar[d]|{\Tc(a)\bdot\Sc(b)} & \Tc(A)\bdot\Sc(B) \ar[d]^{a\bdot b}\\
\Sc(A\bdot B) \ar[r]_{\chi_{A,B}} & \Tc(A)\bdot\Sc(B) \ar[r]_-{a\bdot b} & A\bdot B
}
\]
The top rectangle commutes by~\eqref{e:comod-monad1}, the bottom left square
by naturality of $\chi$ and the bottom right square by associativity for $a$ and $b$.
Unitality for $c$ follows similarly from~\eqref{e:comod-monad2}. 
\end{proof}

The action of $\Cs^\Tc$ on $\Ms^\Sc$ defined in Proposition~\ref{p:alg-cat} turns 
$\Ms^\Sc$ into a $\Cs^\Tc$-module-category;
associativity and unitality follow from those of the action of $\Cs$ on $\Ms$
together with the comodule axioms~\eqref{e:comodule1}--\eqref{e:comodule2}.

\subsection{Generalized Hopf modules}\label{ss:hopf-mod}

Assume now that we are given:
\begin{itemize}
\item A monoidal category $\Cs$ and a $\Cs$-module-category $\Ms$.
\item A bimonad $\Tc$ on $\Cs$ (Definition~\ref{d:bimonad}).
\item A $\Tc$-comodule-monad $\Sc$ on $\Ms$ (Definition~\ref{d:comod-monad}).
\item A $\Tc$-algebra-comonoid $Z$ in $\Cs$ (Definition~\ref{d:comTalg}).
\end{itemize}

We employ the notation of Sections~\ref{ss:comonoidal}--\ref{ss:comod-monad}.
In particular, the structure of $Z$ is denoted by $(\action,\Delta,\epsilon)$.

\begin{definition}\label{d:hopf-mod}
A \emph{Hopf $(\Tc;\Sc,Z)$-module} $(M,\sigma,\zeta)$ consists of an object $M$ in $\Ms$
and arrows
\[
\sigma:\Sc(M)\to M
\qand
\zeta: M\to Z\bdot M
\]
such that $(M,\sigma)$ is an $\Sc$-algebra, $(M,\zeta)$ is a (left) $Z$-comodule,
and the following diagram commutes.
\begin{equation}\label{e:hopf-mod}
\begin{gathered}
\xymatrix@C+10pt@R-5pt{
\Sc(M) \ar[r]^-{\sigma} \ar[d]_{\Sc(\zeta)} & M \ar[dd]^{\zeta} \\
\Sc(Z\bdot M) \ar[d]_{\chi_{Z,M}} & \\
\Tc(Z)\bdot \Sc(M)  \ar[r]_-{\action \bdot \sigma} & Z\bdot M 
}
\end{gathered}
\end{equation}
\end{definition}

A morphism of Hopf $(\Tc;\Sc,Z)$-modules is an arrow in $\Ms$ that is both
a morphism of $\Sc$-algebras and of $Z$-comodules. We let $\HTS$
denote the category of Hopf $(\Tc;\Sc,Z)$-modules.

\begin{remark}\label{r:hopf-mod}
Suppose that $Z=\unit$ is the trivial $\Tc$-algebra-comonoid (the unit object of $\Cs$).
In this case~\eqref{e:hopf-mod} is automatic (follows from~\eqref{e:comodule2})
and a Hopf $(\Tc;\Sc,Z)$-module is just an $\Sc$-algebra. 

On the other hand, when $\Ms=\Cs$, $\Sc=\Tc$, and $Z=\Tc(\unit)$ is the free
$\Tc$-algebra-comonoid on the trivial comonoid in $\Cs$,
a Hopf $(\Tc;\Sc,Z)$-module is precisely a Hopf module in the
sense of  Brugui\`eres, Lack and Virelizier~\cite[Section~4.2]{BruVir:2007}
and~\cite[Section~6.5]{BLV:2011}.
For this reason, we refer to Hopf $(\Tc;\Sc,Z)$-modules loosely as generalized
Hopf modules. 
\end{remark}

\begin{example}\label{eg:hopf-mod}
Consider the case in which still $\Ms=\Cs$ and $\Sc=\Tc$, but the $\Tc$-algebra-comonoid $Z$ is arbitrary. Then $(Z,\action,\Delta)$ is a Hopf $(\Tc;\Sc,Z)$-module. Indeed,~\eqref{e:hopf-mod} coincides in this case with the first diagram
in~\eqref{e:comTalg}.
\end{example}

We return to generalized Hopf modules. Since $Z$ is a comonoid in $\Cs^{\Tc}$,
and this category acts on $\Ms^{\Sc}$, the comonad $\hZ$ on $\Ms^{\Sc}$
is defined, as in~\eqref{e:hat-comonad}.

On the other hand, $Z$ is also a comonoid in $\Cs$, and
since the functor $\Sc$ is a $\Tc$-comodule, it sends $Z$-comodules
to $\Tc(Z)$-comodules:
\[
\Ms_{\hZ} \to \Ms_{\widehat{\Tc(Z)}}.
\]
This is an instance of~\eqref{e:comod-transfer}. Now, as mentioned after diagram~\eqref{e:comTalg}, the structure map
$\action:\Tc(Z)\to Z$ is a morphism of comonoids, and hence induces a functor
\[
\Ms_{\widehat{\Tc(Z)}} \to \Ms_{\hZ}.
\]
Composing these two functors we obtain a new one
\[
\Ms_{\hZ} \to \Ms_{\hZ}
\]
which we still denote by $\Sc$ and which is still a monad.
Explicitly, if $(M, \zeta)$ is a $Z$-comodule in $\Ms$, then $\Sc(M,\zeta) = (\Sc(M), \gamma)$, where $\gamma$ is the composition 
\[
\Sc(M) \map{\Sc(\zeta)} \Sc(Z\bdot M) \map{\chi_{Z,M}} \Tc(Z)\bdot \Sc(M) 
\map{\alpha\bdot\id} Z\bdot \Sc(M).
\]

The category of generalized Hopf modules admits the following descriptions.

\begin{proposition}\label{p:hopf-mod}
There are isomorphisms of categories
\[
(\Ms^{\Sc})_{\hZ} = \HTS = (\Ms_{\hZ})^{\Sc}.
\]
\end{proposition}
\begin{proof}
An object of $(\Ms^{\Sc})_{\hZ}$ is an $\Sc$-algebra in $\Ms$ with a $Z$-comodule
structure $\zeta:M\to Z\bdot M$ which is a morphism of $\Sc$-algebras,
with the $\Sc$-algebra structure of $Z\bdot M$ as in~\eqref{e:alg-cat}.
An object of $(\Ms_{\hZ})^{\Sc}$ is a $Z$-comodule in $\Ms$ with an $\Sc$-algebra
structure $\sigma:\Sc(M)\to M$ which is a morphism of $Z$-comodules. In both cases,
the conditions are expressed by the commutativity of~\eqref{e:hopf-mod}.
\end{proof}

\begin{remark}\label{r:mixed}
It is possible to show the existence of a \emph{mixed distributive law}
between the monad $\Sc$ on $\Ms$ and the comonad $\hZ$ on $\Ms$.
Proposition~\ref{p:hopf-mod} is then a consequence of a general result
for such laws; see~\cite[Propositions~2.1 and~2.2]{Bur:1973} or~\cite[Theorem~2.4]{Wol:1973}. The law is the transformation $\Sc\hZ\to\hZ\Sc$ defined by
\[
\Sc(Z\bdot M) \map{\chi_{Z,M}} \Tc(Z)\bdot \Sc(M) \map{\action\bdot\id} Z\bdot \Sc(M).
\]
\end{remark}

\section{Comodules over comonoidal adjunctions}\label{s:comod}

Adjunctions $(\Fc,\Gc)$ between monoidal categories
and adjunctions $(\Hc,\Kc)$ between module-categories
(both compatible with the monoidal structures) 
are discussed in Section~\ref{ss:comod-adj}. It is well-known that in this
situation the comonoidal structure of $\Gc$ is necessarily invertible; we show the
same is true for the comonoidal structure of $\Kc$ (while this follows from a result
of Kelly, we provide a direct proof).

The monad of a comonoidal adjunction
is a bimonad; moreover, the monad of a comodule over such an adjunction is a comodule-monad. These statements are
discussed in Section~\ref{ss:adj-bimonad} and (in the converse direction) in Section~\ref{ss:bimonad-adj}. They elaborate
on results from~\cite[Section~2.5]{BLV:2011}.

In Section~\ref{ss:Hopf-Galois} we discuss two canonical transformations,
following ideas of Brugui\`eres, Lack and Virelizier~\cite[Sections~2.6 and~2.8]{BLV:2011}. The \emph{Hopf operator} is associated to a comodule over a comonoidal adjunction
and the \emph{Galois map} is associated to a comodule-monad over a bimonad.
Our goal is to show that when the adjunctions are monadic,
invertibility of the Hopf operator is equivalent to invertibility of the Galois map
(Corollary~\ref{c:Galois-Hopf}). This is a result of Brugui\`eres, Lack and 
Virelizier~\cite[Lemma~2.18]{BLV:2011}
which is extended here to the setting of comodules.

\subsection{Comodules over comonoidal adjunctions}\label{ss:comod-adj}
Suppose now that
\[
\xymatrix{
\Cs  \ar@/^/[r]^{\Fc} & \Ds \ar@/^/[l]^{\Gc} 
}
\]
is a \emph{comonoidal adjunction}~\cite[Definition~3.88]{AguMah:2010}. (These adjunctions are called \emph{colax-colax} in ~\cite{AguMah:2010}.)
We let 
\[
\eta:\Ic\to\Gc\Fc \qand \xi:\Fc\Gc\to\Ic
\] 
denote the unit and counit of the adjunction.
Thus, $\Fc$ and $\Gc$ are comonoidal functors and $\eta$ and $\xi$ are morphisms of comonoidal functors.
We use $\psi$ to denote the comonoidal structure of either $\Fc$ or $\Gc$.
Recall that in this situation $\Gc$ is necessarily strong, that is, the transformation
$\psi_{V,W}:\Gc(V\bdot W) \to \Gc(V)\bdot\Gc(W)$ is invertible.
We generalize this fact in Proposition~\ref{p:rightstrong} below.
Both the former result and its generalization are special cases of a result of Kelly;
see Remark~\ref{r:kelly}.

Suppose also that
\[
\xymatrix{
\Ms \ar@/^/[r]^{\Hc} & \Ns \ar@/^/[l]^{\Kc} 
}
\]
is an adjunction,  $\Hc$ is an $\Fc$-comodule and $\Kc$ is a $\Gc$-comodule.
We use $\chi$ to denote the comodule structure of either $\Hc$ or $\Kc$
and $\eta$ and $\xi$ for the unit and counit of the adjunction $(\Hc,\Kc)$.
Since comonoidal functors and their comodules compose (Section~\ref{ss:comod-functor}), $\Gc\Fc$ is comonoidal
and $\Kc\Hc$ is a comodule over it.

\begin{definition}\label{d:comod-adj}
In the above situation, we say that the adjunction
$(\Hc,\Kc)$ is a \emph{comodule} over the comonoidal adjunction $(\Fc,\Gc)$ if the
following diagrams commute
\begin{equation}\label{e:comod-adj}
\begin{gathered}
\xymatrix@C-30pt{
& X\bdot M \ar[dl]_{\eta_{X\bdot M}}  \ar[rd]^{\eta_{X}\bdot \eta_{M}} & \\
\Kc\Hc(X\bdot M) \ar[rr]  & &\Gc\Fc(X)\bdot\Kc\Hc(M) 
}
\qquad
\xymatrix@C-30pt{
\Hc\Kc(V\bdot N) \ar[dr]_{\xi_{V\bdot N}} \ar[rr]  & & \Fc\Gc(V)\bdot\Hc\Kc(N) \ar[ld]^{\xi_{V}\bdot \xi_{N}}\\
 & V\bdot N &
}
\end{gathered}
\end{equation}
for all objects $X$ of $\Cs$, $M$ of $\Ms$, $V$ of $\Ds$, and $N$ of $\Ns$.
\end{definition}

These conditions may be reformulated as follows. First note that
$\eta:\Ic\to\Gc\Fc$ allows us to regard $\Ic:\Ms\to\Ms$ as a comodule over $\Gc\Fc$;
also, $\xi:\Fc\Gc\to\Ic$ allows us to regard
$\Hc\Kc$ as a comodule over  $\Ic:\Ds\to\Ds$. 
Then diagrams~\eqref{e:comod-adj}
state that $\eta:\Ic\to\Kc\Hc$ is a morphism of comodules over $\Gc\Fc$ and 
$\xi:\Hc\Kc\to\Ic$ is a morphism of comodules over $\Ic$.

\begin{proposition}\label{p:rightstrong}
If $(\Hc,\Kc)$ is a comodule over the comonoidal adjunction $(\Fc,\Gc)$, then
the comodule structure of $\Kc$ 
\[
\chi_{V,N}:  \Kc(V\bdot N) \to \Gc(V)\bdot\Kc(N)
\]
is invertible.
\end{proposition}
\begin{proof}
Let $V$ be an object of $\Ds$ and $N$ one of $\Ns$. Consider the composite
\[
\Hc\bigl(\Gc(V)\bdot\Kc(N)\bigr) \map{\chi_{\Gc(V),\Kc(N)}}
\Fc\Gc(V)\bdot \Hc\Kc(N) \map{\xi_V\bdot\xi_N}
V\bdot N.
\]
By adjointness, it corresponds to a map
\[
\gamma_{V,N}: \Gc(V)\bdot\Kc(N) \to \Kc(V\bdot N).
\]

Now consider the following diagram.
\[
\def\labelstyle{\scriptstyle}
\xymatrix@R+1pc@C-1.3pc{
\Gc(V)\bdot\Kc(N)\ar@{<-}[rr]^-{\chi_{V,N}}\ar@{<-}[rd]|-{\Gc(\xi_V)\bdot\Kc(\xi_N)}
\ar@{<-}[dd]|{\id_{\Gc(V)\bdot\Kc(N)}} & & \Kc(V\bdot N)\ar@{<-}[dr]^-{\ \Kc(\xi_V\bdot\xi_N)}
\ar@{<--}@/^2.9pc/[lldd]^-{\gamma_{V,N}}\\
& \Gc\Fc\Gc(V)\bdot\Kc\Hc\Kc(N)\ar@{<-}[rr]^-{\chi_{\Fc\Gc(V),\Hc\Kc(N)}}
\ar@{<-}[dl]|-{\eta_{\Gc(V)}\bdot\eta_{\Kc(N)}} & & \Kc\left(\Fc\Gc(V) \bdot \Hc\Kc(N)\right)
\ar@{<-}[dl]^-{\ \ \ \Kc(\chi_{\Gc(V),\Kc(N)})} \\
\Gc(V)\bdot\Kc(N) & & \Kc\Hc\left(\Gc(V)\bdot\Kc(N)\right)\ar@{<-}[ll]^-{\eta_{\Gc(V) \bdot \Kc(N)}}
}
\]
The top front square commutes by naturality of $\chi$.
The bottom front square is a special case of the first diagram in~\eqref{e:comod-adj}
(the case $X=\Gc(V)$, $M=\Kc(N)$);
it commutes since $(\Hc,\Kc)$ is a comodule over $(\Fc,\Gc)$.
The front triangle commutes by the adjunction axioms.
There are two faces on the back, a triangle on the left and
a square on the right.
The back square commutes by definition of
$\gamma$. It follows that the back triangle commutes.
This says that $ \chi_{V,N}\gamma_{V,N}=\id_{\Gc(V) \bdot \Kc(N)}$.

Similarly, the diagram
\[
\def\labelstyle{\scriptstyle}
\xymatrix@R+1pc@C-.4pc{
\Kc(V \bdot N) & & \Kc(V\bdot N)\ar@{<-}[dr]^-{\ \Kc(\xi_V\bdot\xi_N)}
\ar@{<--}@/^2pc/[lldd]^(0.6){\gamma_{V,N}}\ar@{<-}[ll]_-{\id_{\Kc(V \bdot N)}}\ar@{<-}[ld]|-{\Kc(\xi_{V \bdot N})}
\\
& \Kc\Hc\Kc(V \bdot N)
\ar@{<-}[ul]|-{\eta_{\Kc(V \bdot N)}} & & \hspace*{-10pt}\Kc\left(\Fc\Gc(V)\bdot\Hc\Kc(N)\right)
\ar@{<-}[dl]^-{\ \ \ \Kc(\chi_{\Gc(V),\Kc(N)})} \\
\Gc(V)\bdot\Kc(N) \ar@{<-}[uu]^-{\chi_{V,N}} & & \Kc\Hc\left(\Gc(V)\bdot\Kc(N)\right) \ar@{<-}[ll]^-{\eta_{\Gc(V) \bdot \Kc(N)}}\ar@{<-}[lu]_(0.4){\Kc\Hc(\chi_{V,N})}
}
\]
shows that $\gamma_{V,N} \chi_{V,N} = \id_{\Kc(V \bullet N)}$.
(The quadrilateral on the right results from applying $\Kc$ to the second diagram
in~\eqref{e:comod-adj}.)
Thus, $\chi$ and $\gamma$ are inverses.
\end{proof}

\begin{remark}\label{r:kelly}
When $\Fc=\Hc$ and $\Gc=\Kc$, Proposition~\ref{p:rightstrong} recovers~\cite[Proposition~3.96]{AguMah:2010}, which is a special case of
a result of Kelly~\cite[Theorem~1.4]{Kel:1974}. Proposition~\ref{p:rightstrong}
may also be viewed as a special case of Kelly's result: to this end, one constructs
a doctrine whose categories are pairs $(\Cs,\Ms)$ with $\Cs$ a monoidal category
and $\Ms$ a $\Cs$-module-category.
\end{remark}

\subsection{From bimonads (and their comodule-monads) to comonoidal adjunctions
(and their comodules)}\label{ss:bimonad-adj}

Let $\Tc$ be a monad on a category $\Cs$.
The adjunction associated to $\Tc$~\cite[Theorem~VI.2.1]{Mac:1998} 
(or Section~\ref{ss:adj-comonad}) is denoted
\[
\xymatrix@C+10pt{
\Cs \ar@/^/[r]^{\Fc^\Tc} & \Cs^\Tc  \ar@/^/[l]^{\Uc^\Tc}
}
\]
where $\Cs^\Tc$ is the category of $\Tc$-algebras, $\Fc^\Tc$ is the free algebra functor (left adjoint)
\[
\Fc^\Tc(X) := \bigl(\Tc(X),\mu_X\bigr)
\]
and $\Uc^\Tc$ is the forgetful functor (right adjoint)
\[
\Uc^\Tc(A,a) := A.
\]
The unit and counit of the adjunction are
\[
\eta_{X}:  X \map{\iota_X} \Tc(X)  = \Uc^\Tc\Fc^\Tc(X) 
\qand
\xi_{(A,a)}:    \Fc^\Tc\Uc^\Tc(A,c) = \bigl(\Tc(A),\mu_A\bigr)   \map{a} (A,a).
\]
respectively. Adjunctions of this form are called \emph{monadic}.

Suppose now that $\Cs$ is monoidal and $\Tc$ is a bimonad.
As recalled in Section~\ref{ss:bimonad}, the category
 $\Cs^\Tc$ is then monoidal. In addition, the functors
$\Fc^\Tc$ and $\Uc^\Tc$ are comonoidal, and
$(\Fc^\Tc,\Uc^\Tc)$
is a comonoidal adjunction~\cite[Example~2.4]{BLV:2011}.

Let $\Ms$ be a $\Cs$-module-category and $\Sc$ a comodule-monad on $\Ms$ over 
$\Tc$. Recall that $\Ms^\Sc$ is a $\Cs^\Tc$-module-category with the action
of Proposition~\ref{p:alg-cat}. Consider now the adjunction
\[
\xymatrix@C+10pt{
\Ms \ar@/^/[r]^{\Fc^\Sc} & \Ms^\Sc  \ar@/^/[l]^{\Uc^\Sc}
}
\]
associated to the monad $\Sc$.

\begin{proposition}\label{p:bimonad-adj}
The functor $\Fc^\Sc$ is a comodule over $\Fc^\Tc$, 
$\Uc^\Sc$ is a comodule over $\Uc^\Tc$, and the adjunction $(\Fc^\Sc,\Uc^\Sc)$ is a comodule over the adjunction $(\Fc^\Tc,\Uc^\Tc)$. 
\end{proposition}
\begin{proof}
The comodule structure of $\Sc$ defines a morphism of $\Sc$-algebras
\[
\Fc^\Sc(X\bdot M)=\bigl(\Sc(X\bdot M),\mu_{X\bdot M}\bigr)
\map{\chi_{X,M}} \bigl(\Tc(X)\bdot\Sc(M),c\bigr) = \Fc^\Tc(X)\bdot \Fc^\Sc(M),
\]
where $c$ is the $\Sc$-algebra structure~\eqref{e:alg-cat} for the free algebras
$(A,a)=\bigl(\Tc(X),\mu_X\bigr)$
and $(B,b)=\bigl(\Sc(M),\mu_M\bigr)$.
Indeed, this is equivalent to~\eqref{e:comod-monad1}.
This turns $\Fc^\Sc$ into a $\Fc^\Tc$-comodule.

For the forgetful functor we have
\[
\Uc^\Sc\bigl((A,a)\bdot(B,b)\bigr)=\Uc^\Sc(A\bdot B,c) = A\bdot B
=\Uc^\Tc(A,a)\bdot\Uc^\Sc(B,b),
\]
so the identity turns $\Uc^\Sc$ into a $\Uc^\Tc$-comodule.

Finally, consider axioms~\eqref{e:comod-adj} for the adjunction
$(\Fc^\Sc,\Uc^\Sc)$ to be a comodule over $(\Fc^\Tc,\Uc^\Tc)$.
The first diagram in~\eqref{e:comod-adj} is the same as~\eqref{e:comod-monad2}
and hence commutes.
The counits of the adjunction are given by the algebra structures, so
the second diagram commutes by~\eqref{e:alg-cat}.
\end{proof}

\subsection{From comonoidal adjunctions (and their comodules) to bimonads (and their comodule-monads)}\label{ss:adj-bimonad}

We go back to the situation of Section~\ref{ss:comod-adj}. 
Given a comonoidal adjunction
\[
\xymatrix{
\Cs  \ar@/^/[r]^{\Fc} & \Ds \ar@/^/[l]^{\Gc},
}
\]
consider the associated monad $\Tc=\Gc\Fc$ on $\Cs$~\cite[Section~VI.1]{Mac:1998} (or Section~\ref{ss:comparison}). Its structure is as follows.
\begin{equation}\label{e:adj-monad}
\begin{gathered}
\mu:\Tc^2 = \Gc\Fc\Gc\Fc\map{\Gc\xi\Fc}\Gc\Ic\Fc=\Tc\\
\iota:\Ic\map{\eta}\Gc\Fc=\Tc\\
\end{gathered}
\end{equation}

Since comonoidal functors compose (Section~\ref{ss:comod-functor}), the functor
$\Tc$ is comonoidal. Moreover, $\Tc$ is a bimonad~\cite[Theorem~2.6]{BruVir:2007}. 

Suppose now that
\[
\xymatrix{
\Ms \ar@/^/[r]^{\Hc} & \Ns \ar@/^/[l]^{\Kc} 
}
\]
is an adjunction,  $\Hc$ is an $\Fc$-comodule and $\Kc$ is a $\Gc$-comodule.
Since comodules compose, the monad $\Sc=\Kc\Hc$ on $\Ms$ is
 a comodule over $\Tc$.

\begin{proposition}\label{p:adj-bimonad}
Suppose $(\Hc,\Kc)$ is a comodule over the comonoidal adjunction $(\Fc,\Gc)$.
Then $\Sc$ is a comodule-monad over the bimonad $\Tc$.
\end{proposition}
\begin{proof}
Consider the following diagram.
\[
\def\labelstyle{\scriptstyle}
\def\objectstyle{\scriptstyle}
\xymatrix@C+14pt{
\Sc^2(X\bdot M) \ar[rr]^-{\Sc(\chi_{X,M})} \ar@{=}[d] \ar@/_3.9pc/[ddddd]|-{\mu_{X,M}}
& & \Sc\bigl(\Tc(X)\bdot\Sc(M)\bigr) \ar@{=}[d] 
\ar@/^2pc/[dddr]^-{\chi_{\Tc(X),\Sc(M)}}\\
\Kc\Hc\Kc\Hc(X\bdot M) \ar[r]^-{\Kc\Hc\Kc(\chi_{X,M})} \ar[dd]|{\Kc(\xi_{\Hc(X\bdot M)})} & \Kc\Hc\Kc\bigl(\Fc(X)\bdot\Hc(M)\bigr) \ar[r]^-{\Kc\Hc(\chi_{\Fc(X),\Hc(M)})}
\ar[dd]|{\Kc(\xi_{\Fc(X)\bdot\Hc(M)})} & 
\Kc\Hc\bigl(\Gc\Fc(X)\bdot\Kc\Hc(M)\bigr) \ar[d]^{\Kc(\chi_{\Gc\Fc(X),\Kc\Hc(M)})} \\ 
& & \Kc\bigl(\Fc\Gc\Fc(X)\bdot\Hc\Kc\Hc(M)\bigr) \ar[dl]|-{\Kc(\xi_{\Fc(X)}\bdot\xi_{\Hc(M)})} \ar[d]^{\chi_{\Fc\Gc\Fc(X),\Hc\Kc\Hc(M)}}\\
\Kc\Hc(X\bdot M) \ar[r]^-{\Kc(\chi_{X,M})} \ar@{=}[dd] & \Kc\bigl(\Fc(X)\bdot\Hc(M)\bigr)
\ar[rd]|-{\chi_{\Fc(X),\Hc(M)}} & \Gc\Fc\Gc\Fc(X)\bdot \Kc\Hc\Kc\Hc(M) \ar@{=}[r] \ar[d]^{\Gc(\xi_{\Fc(X)})\bdot\Kc(\xi_{\Hc(M)})} & \Tc^2(X)\bdot\Sc^2(M) \ar@/^2pc/[ddl]^-{\mu_X\bdot\mu_M}\\
& & \Gc\Fc(X)\bdot \Kc\Hc(M) \ar@{=}[d] \\
\Sc(X\bdot M) \ar[rr]_-{\chi_{X,M}} & & \Tc(X)\bdot \Sc(M)
}
\]
The diagrams around the boundary commute by definition. 
The commutativity of the quadrilateral in the center follows from that of the second
diagram in~\eqref{e:comod-adj}.
The two remaining quadrilaterals commute by naturality of $\xi$ and $\chi$.
It follows that diagram~\eqref{e:comod-monad1} commutes.
The first diagram in~\eqref{e:comod-adj} is the same as~\eqref{e:comod-monad2},
and so $\Sc$ is a comodule-monad over $\Tc$. 
\end{proof}

\subsection{Hopf operators and Galois maps}\label{ss:Hopf-Galois}

Let $(\Fc,\Gc)$ be a comonoidal adjunction and $(\Hc,\Kc)$ a comodule over it,
as in Definition~\ref{d:comod-adj}.

\begin{definition}\label{d:hopf-oper}
The \emph{Hopf operator} is the map $\Hb_{X,N}$ defined by
\begin{equation}\label{e:Hopf}
\Hc\bigl(X\bdot\Kc(N)\bigr)
\map{\chi_{X,\Kc(N)}} \Fc(X)\bdot \Hc\Kc(N) 
\map{\id_{\Fc(X)}\bdot\xi_N} \Fc(X)\bdot N
\end{equation}
for all objects $X$ of $\Cs$ and $N$ of $\Ns$.
\end{definition}

Let $\Tc$ be a bimonad and $\Sc$ a comodule-monad over it, as in Definition~\ref{d:comod-monad}.

\begin{definition}\label{d:galois-map}
The \emph{fusion} or \emph{Galois map} is the map $\Gb_{X,M}$ defined by
\begin{equation}\label{e:Galois}
\Sc\bigl(X\bdot\Sc(M)\bigr)
\map{\chi_{X,\Sc(M)}}
\Tc(X)\bdot\Sc^2(M) \map{\id_{\Tc(X)}\bdot\mu_{M}}
\Tc(X)\bdot\Sc(M)
\end{equation}
for all objects $X$ of $\Cs$ and $M$ of $\Ms$.
\end{definition}

\begin{remark}
The Hopf operator and the fusion map for comonoidal adjunctions were introduced by
Brugui\`eres, Lack and Virelizier~\cite[Sections~2.6 and~2.8]{BLV:2011}.
Above we have extended these notions to the setting of comodules over comonoidal
adjunctions. In the same manner, Proposition~\ref{p:Hopf-Galois}
and Corollary~\ref{c:Galois-Hopf} below correspond to~\cite[Proposition~2.14 and Lemma~2.18]{BLV:2011}.
\end{remark}

Let $(\Fc,\Gc)$ be a comonoidal adjunction, $(\Hc,\Kc)$ a comodule over it,
and $\Hb$ the corresponding Hopf operator.
Let $\Tc=\Gc\Fc$ and $\Sc=\Kc\Hc$. By Proposition~\ref{p:adj-bimonad},
$\Tc$ is a bimonad and $\Sc$ is a $\Tc$-comodule-monad.
Let $\Gb$ be the corresponding Galois map.

In order to bridge between the Hopf operator and the Galois map,
we introduce an auxiliary transformation $\Ab_{X,N}$ defined by
\begin{equation}\label{e:aux}
\Sc\bigl(X\bdot \Kc(N)\bigr)
\map{\chi_{X,\Kc(N)}}
\Tc(X)\bdot\Sc\Kc(N) = \Tc(X)\bdot \Kc\Hc\Kc(N) 
 \map{\id_{\Tc(X)}\bdot \Kc(\xi_N)}
\Tc(X)\bdot \Kc(N)
\end{equation}
for all objects $X$ of $\Cs$ and $N$ of $\Ns$.

\begin{lemma}\label{l:aux}
In the preceding situation, we have
\[
\Ab_{X,\Hc(M)} = \Gb_{X,M} \qand \chi_{\Fc(X),N} \Kc(\Hb_{X,N}) = \Ab_{X,N}
\]
for all $X$ in $\Cs$, $M$ in $\Ms$, and $N$ in $\Ns$.
\end{lemma}
\begin{proof}
Consider the following diagram.
\[
\xymatrix@C+30pt@R-5pt{
\Sc\bigl(X\bdot\Kc\Hc(M)\bigr) \ar[r]^-{\chi_{X,\Kc\Hc(M)}} \ar@{=}[d]  &
\Tc(X)\bdot \Kc\Hc\Kc\Hc(M)\bigr) \ar[r]^-{\id\bdot\Kc(\xi_{\Hc(M)})} \ar@{=}[d] 
& \Tc(X)\bdot \Kc\Hc(M) \ar@{=}[d] \\
\Sc\bigl(X\bdot\Sc(M)\bigr) \ar[r]_-{\chi_{X,\Sc(N)}} 
& \Tc(X)\bdot \Sc\Sc(M)\bigr) \ar[r]_-{\id\bdot\mu_M}
& \Tc(X)\bdot\Sc(M)
}
\]
The maps along the top equal the corresponding maps along the bottom;
in the case of the right maps this is~\eqref{e:adj-monad}.
By~\eqref{e:aux} and~\eqref{e:Galois},
the composite along the top is $\Ab_{X,\Hc(M)}$ and that along the bottom is
$\Gb_{X,M}$. This proves the first equality.

Consider now the following diagram.
\[
\xymatrix@C+30pt@R-5pt{
\Kc\Hc\bigl(X\bdot\Kc(N)\bigr) \ar[r]^-{\Kc(\chi_{X,\Kc(N)})} \ar@{=}[dd] &
\Kc\bigl(\Fc(X)\bdot\Hc\Kc(N)\bigr) \ar[r]^-{\Kc(\id\bdot\xi_N)} 
\ar[d]^{\chi_{\Fc(X),\Hc\Kc(N)}} &
\Kc\bigl(\Fc(X)\bdot N\bigr) \ar[d]^{\chi_{\Fc(X),N}}\\
& \Gc\Fc(X)\bdot\Kc\Hc\Kc(N) \ar[r]^-{\id\bdot\Kc(\xi_N)} \ar@{=}[d]
& \Gc\Fc(X)\bdot\Kc(N) \ar@{=}[d]\\
\Sc\bigl(X\bdot\Kc(N)\bigr) \ar[r]_-{\chi_{X,\Kc(N)}} 
& \Tc(X)\bdot\Sc\Kc(N) \ar[r]_-{\id\bdot\Kc(\xi_N)}
& \Tc(X)\bdot\Kc(N)
}
\]
The left half commutes by~\eqref{e:comod-compose} and the right corner
by naturality. The top row is $\Kc(\Hb_{X,N})$ and the bottom $\Ab_{X,N}$.
This proves the second equality.
\end{proof}

We can now relate the invertibility of the Hopf operator to that of the Galois map.

\begin{proposition}\label{p:Hopf-Galois}
Let $(\Fc,\Gc)$ be a comonoidal adjunction, $(\Hc,\Kc)$ a comodule over it,
and $\Hb$ the corresponding Hopf operator.
Let $\Tc=\Gc\Fc$, $\Sc=\Kc\Hc$, and $\Gb$ the corresponding Galois map.
Fix an object $X$ in $\Cs$ and consider the following statements.
\begin{enumerate}[(i)]
\item The transformation $\Hb_{X,-}$ is invertible.
\item The transformation $\Ab_{X,-}$ is invertible.
\item The transformation $\Gb_{X,-}$ is invertible.
\end{enumerate}
We have that
\[
\textup{(i) $\Rightarrow$ (ii) $\Rightarrow$ (iii).}
\]
Moreover, if the adjunction $(\Hc,\Kc)$ is monadic, then all three statements are equivalent.
\end{proposition}
\begin{proof}
The implications (i) $\Rightarrow$ (ii) $\Rightarrow$ (iii) are immediate from
Lemma~\ref{l:aux} and the fact that  $\chi_{\Fc(X),N}$ (an instance of the comonoidal structure of $\Kc$) is invertible by Proposition~\ref{p:rightstrong}.

To show (iii) $\Rightarrow$ (ii) we appeal to~\cite[Lemma~2.19]{BLV:2011}:
if $\Ac,\Bc:\Ms\to \Ms'$ are two functors and $\alpha:\Ac\Kc\to\Bc\Kc$ is a natural transformation, then
monadicity of $(\Hc,\Kc)$ together with invertibility of $\alpha\Hc$ imply
invertibility of $\alpha$. 
Applying this result to $\alpha=\Ab_{X,-}$ yields the desired implication,
since $\Ab_{X,\Hc(-)} = \Gb_{X,-}$ by Lemma~\ref{l:aux}.

Finally, by monadicity, $\Kc$ reflects isomorphisms. 
Since $\Kc(\Hb_{X,-}) = \chi_{\Fc(X),-}^{-1} \Ab_{X,-}$ by Lemma~\ref{l:aux}, 
we have that (ii) $\Rightarrow$ (i).
\end{proof}

\begin{corollary}\label{c:Galois-Hopf}
Let $\Tc$ be a bimonad on $\Cs$, $\Sc$ a comodule-monad on $\Ms$ over $\Tc$,
and $\Gb$ the corresponding Galois map. Let $\Hb$ be the Hopf operator
corresponding to the adjunctions $(\Fc^\Tc,\Uc^\Tc)$ and $(\Fc^\Sc,\Uc^\Sc)$.
Let $X$ be an object $\Cs$.
Then invertibility of the Galois map $\Gb_{X,M}$ for every $M$ of $\Ms$
is equivalent
to invertibility of the Hopf operator $\Hb_{X,(B,b)}$ for every $\Sc$-algebra $(B,b)$.
\end{corollary}
\begin{proof}
This is the special case of Proposition~\ref{p:Hopf-Galois} in which both
adjunctions are monadic.
\end{proof}

\section{The Fundamental Theorem of generalized Hopf modules}\label{s:fundamental}

In this section we prove the main result of the paper.
The progression of results leading to it is as follows.
Consider a comonoidal functor $\Fc:\Cs\to\Ds$, a comonoid $C$ in $\Cs$,
and the comonoid $D=\Fc(C)$ in $\Ds$. 
An $\Fc$-comodule $\Hc$ gives rise to
a canonical colax morphism relating the comonads $\hC$ and $\hD$. 
This is discussed in Section~\ref{ss:chi-colax}. 
If $\Kc$ is a right adjoint for $\Hc$, we may consider the conjugate comonad $\tC$.
Its universal property (Section~\ref{ss:universal}) affords a transformation
relating $\tC$ to $\hD$. We identify this transformation in Section~\ref{ss:hopf-colax}:
if $\Fc$ is part of a comonoidal adjunction, and $(\Hc,\Kc)$ is a
comodule adjunction over it, then the transformation coincides with
the Hopf operator of Section~\ref{ss:Hopf-Galois}. 
Ultimately we are interested in
comparing the categories of $C$-comodules and $D$-comodules.
The Conjugate Comonadicity Theorem and the companion results of Section~\ref{s:conjugating} allow us to reduce this question to the invertibility of the Hopf operator
$\Hb_{C}$.
Finally, in Section~\ref{ss:fundamental} we consider the case when the adjunctions
are monadic, arising from a bimonad $\Tc$ and a $\Tc$-comodule-monad $\Sc$.
The comparison now takes place between $C$-comodules
and Hopf $(\Tc;\Sc,\Tc(C))$-modules,
and the invertibility that controls it is that of the Galois map $\Gb_{C}$. This is the Fundamental
Theorem.

The Fundamental Theorem reduces to~\cite[Theorem~6.11]{BLV:2011}
in the special case when $C$ is the trivial comonoid and $\Tc=\Sc$.
We have followed the approach to its proof taken by Brugui\`eres, Lack and Virelizier
in~\cite{BLV:2011}. The additional complications brought in by $C$ are handled
with the aid of the Conjugate Comonadicity Theorem.

\subsection{From comodules to colax morphisms of comonads}\label{ss:chi-colax}

We place ourselves in the situation of Section~\ref{ss:comod-functor}.
Thus, $\Cs$ and $\Ds$ are monoidal categories which act on categories
$\Ms$ and $\Ns$ (respectively), and
\[
\Fc:\Cs\to\Ds \qand \Hc:\Ms\to\Ns
\]
are functors such that $\Fc$ is comonoidal with structure $\psi$ and $\Hc$
is an $\Fc$-comodule with structure $\chi$.

Let $C$ be a comonoid in $\Cs$ and $D=\Fc(C)$ the corresponding comonoid in 
$\Ds$, as in~\eqref{e:coprod-colax}.
Consider the comonads $\hC$ on $\Ms$ and $\hD$ on $\Ns$, as in Section~\ref{ss:comod-comonoid}.
The $\Fc$-comodule structure of $\Hc$ 
\[
\chi_{C,M}: \Hc(C\bdot M) \map{\chi_{C,M}} \Fc(C)\bdot \Hc(M)=D\bdot \Hc(M)
\]
yields a natural transformation 
\[
\chi_C:\Hc\hC\to \hD\Hc.
\]

\begin{lemma}\label{l:comodule-colax}
The pair
\[
(\Hc,\chi_C):(\Ms,\hC) \to (\Ns,\hD)
\]
is a colax morphism of comonads.
\end{lemma}
\begin{proof}
Consider the following diagram.
\[
\xymatrix@C+4.5pt@R-5pt{    
\Hc(C\bdot C\bdot M) \ar[r]^-{\chi_{C,C\bdot M}} \ar[rrd]_{\chi_{C\bdot C,M}}
& \Fc(C)\bdot\Hc(C\bdot M) \ar[r]^-{\id\bdot\chi_{C,M}} & 
\Fc(C)\bdot \Fc(C)\bdot\Hc(M) \ar@{=}[r] & D\bdot D\bdot\Hc(M) \\
& & \Fc(C\bdot C)\bdot\Hc(M) \ar[u]_{\psi_{C,C}\bdot\id} \\
\Hc(C\bdot M) \ar[rr]_-{\chi_{C,M}} \ar[uu]^{\Hc(\Delta\cdot\id)} & & 
\Fc(C)\bdot\Hc(M) \ar[u]_{\Fc(\Delta)\bdot\id} 
\ar@{=}[r] & D\bdot\Hc(M) \ar[uu]_-{\Delta\bdot\id}
}
\]
The diagram in the top left commutes by~\eqref{e:comodule1}, the one in the bottom by naturality, and the one on the right by~\eqref{e:coprod-colax}.
The commutativity of the entire diagram yields the first condition in~\eqref{e:mor-comonad}. Similarly, the commutativity of
\[
\xymatrix@C+5pt@R-5pt{ 
\Hc(C\bdot M) \ar[r]^-{\chi_{C,M}} \ar[d]_{\Hc(\epsilon\bdot\id)} & \Fc(C)\bdot\Hc(M)
\ar@{=}[r] \ar[d]^{\Fc(\epsilon)\bdot\id} & D\bdot\Hc(M) \ar@/^1pc/[ldd]^-{\epsilon\bdot\id} \\
\Hc(\unit\bdot M) \ar[r]_-{\chi_{\unit,M}} \ar@{=}[d] & \Fc(\unit)\bdot\Hc(M) \ar[d]^{\psi_0\bdot\id} \\
\Hc(M) \ar@{=}[r] & \unit\bdot\Hc(M)
}
\]
yields the second condition in~\eqref{e:mor-comonad}.
\end{proof}

We now progress to the situation of Section~\ref{ss:comod-adj}. Thus, we are given
adjunctions
\[
\xymatrix{
\Cs  \ar@/^/[r]^{\Fc} & \Ds \ar@/^/[l]^{\Gc} 
}
\qquad
\xymatrix{
\Ms \ar@/^/[r]^{\Hc} & \Ns \ar@/^/[l]^{\Kc} 
}
\]
the first being comonoidal and the second being a comodule over the first.

Let
\[
\rho_C : \hC\Kc\to \Kc\hD
\]
be the mate of the transformation $\chi_C:\Hc\hC\to \hD\Hc$. It follows from Lemma~\ref{l:comodule-colax} (and the discussion in Section~\ref{ss:mates}) that $(\Kc,\rho_C):(\Ns,\hD)\to(\Ms,\hC)$ is a lax morphism of comonads.

\begin{lemma}\label{l:comodule-colax2}
On an object $N$ of $\Ns$, the transformation $\rho_C$ is the composite
\[
\hC\Kc(N) = C\bdot \Kc(N) \map{\eta_C\bdot\id} \Gc\Fc(C)\bdot\Kc(N) 
= \Gc(D)\bdot\Kc(N) \map{\chi_{D,N}^{-1}} \Kc(D\bdot N)=\Kc\hD(N).
\]
\end{lemma}
The last arrow is the inverse of the $\Gc$-comodule structure map of $\Kc$, afforded
by Proposition~\ref{p:rightstrong}.
\begin{proof}
Consider the following diagram.
\[
\xymatrix@C+20pt{
\Kc\Hc\bigl(C\bdot\Kc(N)\bigr) \ar[r]^-{\Kc(\chi_{C,\Kc(N)})} & \Kc\bigl(\Fc(C)\bdot\Hc\Kc(N)\bigr) \ar[d]^{\chi_{\Fc(C),\Hc\Kc(N)}} \ar[r]^-{\Kc(\id\bdot\xi_N)} & \Kc\bigl(\Fc(C)\bdot N\bigr) \ar[d]^{\chi_{\Fc(C),N}} \\
C\bdot \Kc(N)\ar[u]^-{\eta_{C\bdot\Kc(N)}} \ar[r]_-{\eta_C\bdot\eta_{\Kc(N)}} 
\ar[rd]_-{\eta_c\bdot\id} &
\Gc\Fc(C)\bdot \Kc\Hc\Kc(N) \ar[r]_-{\id\bdot\Kc(\xi_N)} & \Gc\Fc(C)\bdot\Kc(N) \\
& \Gc\Fc(C)\bdot\Kc(N) \ar@{=}[ru] \ar[u]_{\id\bdot\eta_{\Kc(N)}}
}
\]
According to the definition of mate (Section~\ref{ss:mates}), $\rho_C$ is the map from $C\bdot \Kc(N)$ to
$\Kc\bigl(\Fc(C)\bdot N\bigr)$ obtained by going around the top left corner of the diagram.
The diagram commutes: the rectangle on the left by the first diagram in~\eqref{e:comod-adj},
and the other smaller diagrams by naturality, functoriality, and one of the adjunction axioms.
Therefore $\rho_C$ is also obtained by going around the bottom right corner, and this gives the result. 
\end{proof}

\subsection{The Hopf operator as a colax morphism}\label{ss:hopf-colax}

We continue in the situation at the end of Section~\ref{ss:chi-colax}, and bring
the conjugate comonad into the discussion.

Let $C$ and $D$ be as before. In addition to the comonads $\hC$ and $\hD$,
we may consider the conjugate comonad of $\hC$, as in Section~\ref{ss:conjugate-comonad}. Let us denote the latter by $\tC$ (rather than $\contra{(\hC)}$). Thus,
\[
\tC=\Hc\hC\Kc.
\]

Setting $X=C$ in~\eqref{e:Hopf}, we see that the Hopf operator
yields a natural transformation
\[
\tC(N) = \Hc\bigl(C\bdot\Kc(N)\bigr)
\map{\Hb_{C,N}} \Fc(C)\bdot N = \hD(N).
\]
We denote this transformation by
\[
\Hb_{C}: \tC \to \hD.
\]

\begin{lemma}\label{l:hopf-colax}
The pair 
\[
(\Ic,\Hb_{C}): (\Ns,\tC)\to (\Ns,\hD)
\] 
is a colax morphism of comonads.
\end{lemma}
\begin{proof}
We apply the universal property of Proposition~\ref{p:universal} to the
colax morphism of comonads
 $(\Hc,\chi_C):(\Ms,\hC) \to (\Ns,\hD)$ of Lemma~\ref{l:comodule-colax}.
This yields a
transformation $\chi_C':\tC \to \hD$ such that $(\Ic,\chi_C')$ is a colax morphism of comonads.
According to~\eqref{e:universal2}, $\chi_C'$ is the composite
\[
\tC(N) \map{\chi_{C,\Kc(N)}} \Fc(C)\bdot\Hc\Kc(N)=D\bdot \Hc\Kc(N) \map{\id_D\bdot \xi_N} D\bdot N = \hD(N).
\]
Comparing to~\eqref{e:Hopf}, we see that $\chi_C'= \Hb_{C}$, which gives the result.
\end{proof}

\begin{remark} When $C=\unit$ is the trivial comonoid in $\Cs$,
Lemma~\ref{l:hopf-colax} becomes~\cite[Lemma~6.5]{BLV:2011}. 
\end{remark}

We now wish to apply the Adjoint Lifting Theorem (Theorem~\ref{t:adjoint}) to
the adjunction $(\Hc,\Kc)$ and the mates $\chi_C$ and $\rho_C$
of Section~\ref{ss:chi-colax}. When the theorem applies,
we obtain an adjunction
\begin{equation}\label{e:hopf-colax}
\xymatrix@C+15pt{
\Ms_{\hC} \ar@/^/[r]^{\Hc_{\chi_C}} & \Ns_{\hD} \ar@/^/[l]^{\Kc^{\rho_C}} 
}
\end{equation}
between comodule categories. According to next results,
the Hopf operator determines when this is an adjoint equivalence.

\begin{theorem}\label{t:hopf-colax}
Let $C$ be a comonoid in $\Cs$ and $D=\Fc(C)$. Assume that:
\begin{itemize}
\item The category $\Ms$ and the functor $\Hc$ satisfy hypotheses~\eqref{e:coreflexive}--\eqref{e:Hconserve}.
\item The functor $\hC$ satisfies hypothesis~\eqref{e:preserve}.
 \end{itemize}
In this situation the adjunction~\eqref{e:hopf-colax} is defined. Moreover, it
is an adjoint equivalence if and only
if the transformation $\Hb_{C}$ is invertible. 
\end{theorem}
\begin{proof}
This follows from Theorem~\ref{t:equivalence-coalgebra}, in view of the fact
(shown in the proof of Lemma~\ref{l:hopf-colax}) that $\chi_C'= \Hb_{C}$.
\end{proof}

We state this result under simpler (though stronger) hypotheses.

\begin{theorem}\label{t:hopf-colax-conv}
Let $C$ be a comonoid in $\Cs$ and $D=\Fc(C)$. Assume that:
\begin{itemize}
\item The category $\Ms$ has equalizers of all coreflexive pairs.
\item These equalizers are preserved by the functors $\hC$ and $\Hc$.
\item The functor $\Hc$ is conservative.
\end{itemize}
Then the conclusion of Theorem~\ref{t:hopf-colax} holds.
\end{theorem}
\begin{proof}
This follows by noting that the given hypotheses are stronger than those of Theorem~\ref{t:hopf-colax}, or by appealing to Theorem~\ref{t:equivalence-coalgebra-conv}.
\end{proof}

We provide a statement regarding the necessity of the hypotheses on the functor $\Hc$. To this end, let $C=\unit$ be the trivial comonoid in $\Cs$.
Then $\Ms_{\hC}=\Ms$ and $D=\Fc(\unit)$.

\begin{proposition}\label{p:hopf-colax-conv}
In the above situation, consider the following hypotheses.
\begin{enumerate}[(i)]
\item The functor $\resH: \Ms \to\Ns_{\Dc}$ is an equivalence.
\item The functor $\hD$ preserves all existing equalizers in $\Ns$.
\item The transformation $\Hb_{\unit}$ is invertible.
\end{enumerate}
We then have the following statements.
\begin{itemize}
\item If \textup{(i)} holds, then the functor $\Hc$ satisfies hypothesis~\eqref{e:Hconserve}.
\item Assume \textup{(i)} and \textup{(ii)} hold.
Let $(f,g)$ be a parallel pair in $\Ms$ such that $(f,g)$ has an equalizer in $\Ms$ and
$\bigl(\Hc(f),\Hc(g)\bigr)$ has an equalizer in $\Ns$.
Then the functor $\Hc$ preserves the equalizer of  $(f,g)$.
In particular, $\Hc$ satisfies hypothesis~\eqref{e:Hpreserve}.
\item Assume \textup{(i)} and \textup{(iii)} hold. Then $\Hc$ and $\Ms$ satisfy hypothesis~\eqref{e:coreflexive}.
\end{itemize}
\end{proposition}
\begin{proof}
This follows from Proposition~\ref{p:equivalence-coalgebra-conv}.
\end{proof}

\subsection{The Fundamental Theorem}\label{ss:fundamental}

We move to the situation of Sections~\ref{ss:comod-monad} and~\ref{ss:hopf-mod}.
We are given a bimonad $\Tc$ on a monoidal category $\Cs$ and
a $\Tc$-comodule-monad $\Sc$ on a $\Cs$-category $\Ms$. In addition,
let $C$ be a comonoid in $\Cs$ and let $Z=\Tc(C)$ be the free
$\Tc$-algebra-comonoid on it, as in~\eqref{e:freecomTalg}. 

Consider the category $\HTS$ of Hopf $(\Tc;\Sc,Z)$-modules, as in Section~\ref{ss:hopf-mod}. The Fundamental Theorem establishes an explicit equivalence
\begin{equation}\label{e:hopf}
\xymatrix@C+20pt{
 \Ms_{\hC}  \ar@<0.5ex>[r]^-{\Ac} & \HTS \ar@<0.5ex>[l]^-{\Bc} 
}
\end{equation}
between $C$-comodules in $\Ms$ and Hopf $(\Tc;\Sc,Z)$-modules.
We proceed to set it up. 

The functor $\Ac$ sends a $C$-comodule $(M,c)$ to
\[
\Ac(M,c) = (\Sc(M),\mu_M,d)
\]
where $\mu$ is the multiplication of the monad $\Sc$ and $d$ is the composite
\begin{equation}\label{e:hopfP}
\Sc(M) \map{\Sc(c)} \Sc(C\bdot M) \map{\chi_{C,M}} \Tc(C)\bdot\Sc(M) = Z\bdot \Sc(M).
\end{equation}
Below we show that $\Ac(M,c)$ is indeed a Hopf $(\Tc;\Sc,Z)$-module.

The functor $\Bc$ sends a Hopf $(\Tc;\Sc,Z)$-module $(M,\sigma,\zeta)$
to the equalizer of the pair
\begin{equation}\label{e:hopfQ}
 C\bdot M \map{\Delta\bdot\id} C\bdot C\bdot M \map{\id\bdot\iota_C\bdot\id} 
 C\bdot \Tc(C)\bdot M = C\bdot Z\bdot M
\qand
C\bdot M \map{\id\bdot\zeta} C\bdot Z\bdot M.
\end{equation}
Here $\iota$ is the unit of the monad $\Tc$.
The functor $\Bc$ is defined when these equalizers exist in $\Ms_{\hC}$.
Below we give conditions which ensure this is the case. 

\begin{theorem}[The Fundamental Theorem of generalized Hopf modules]
\label{t:hopf}
In the above situation, assume that:
\begin{itemize}
\item The category $\Ms$ has equalizers of all coreflexive pairs
which are $\Sc$-split in $\Ms^\Sc$.
\item The functors $\hC$, $\hC^2$, and $\Sc$ preserve such equalizers.
\item The functor $\Sc$ is conservative.
\end{itemize}
Then the functors $\Ac$ and $\Bc$ from~\eqref{e:hopf} form an adjoint equivalence
\[
\Ms_{\hC} \cong \HTS
\]
if and only if the Galois map 
\[
\Gb_{C,M}: \Sc\bigl(C\bdot\Sc(M)\bigr) \to \Tc(C)\bdot\Sc(M) = Z\bdot \Sc(M)
\]
is invertible for all objects $M$ in $\Ms$. 
\end{theorem}
\begin{proof}
Consider the adjunctions
\[
\xymatrix@C+10pt{
\Cs \ar@/^/[r]^{\Fc^\Tc} & \Cs^\Tc  \ar@/^/[l]^{\Uc^\Tc}
}
\qqand
\xymatrix@C+10pt{
\Ms \ar@/^/[r]^{\Fc^\Sc} & \Ms^\Sc  \ar@/^/[l]^{\Uc^\Sc}
}
\]
associated to the bimonad $\Tc$ and the monad $\Sc$,
as in Section~\ref{ss:bimonad-adj}. By Proposition~\ref{p:bimonad-adj},
the latter is a comodule over the former. Thus, we are in the situation of
Section~\ref{ss:hopf-colax}.
We proceed to verify the hypotheses of Theorem~\ref{t:hopf-colax}.

First of all, by Corollary~\ref{c:Galois-Hopf}, the 
invertibility of the transformation $\Gb_{C}$ is equivalent to that of the 
transformation $\Hb_{C}$.

Since $\Sc=\Uc^{\Sc}\Fc^{\Sc}$ is conservative, so is $\Fc^{\Sc}$. Thus 
$\Fc^{\Sc}$ satisfies~\eqref{e:Hconserve}.
The hypothesis on existence of equalizers 
simply states that $\Ms$ and $\Fc^{\Sc}$ satisfy~\eqref{e:coreflexive}.
These equalizers are preserved by $\Sc$ by hypothesis, and 
created by $\Uc^{\Sc}$ by the dual of Lemma~\ref{l:creation}.
Therefore, they are preserved by $\Fc^{\Sc}$.
Thus, $\Fc^{\Sc}$ satisfies~\eqref{e:Hpreserve}.

Finally, the hypotheses on $\hC$ guarantee that $\hC$ satisfies hypothesis~\eqref{e:preserve}.

Therefore, we may apply Theorem~\ref{t:hopf-colax} to the adjunctions
$(\Fc^\Tc,\Uc^\Tc)$ and $(\Fc^\Sc,\Uc^\Sc)$, to obtain an adjoint equivalence
\[
\xymatrix@C+20pt{
\Ms_{\hC} \ar@<0.5ex>[r]^{(\Fc^{\Sc})_{\chi_C}} & (\Ms^{\Sc})_{\hZ} \ar@<0.5ex>[l]^{(\Uc^{\Sc})^{\rho_C}} 
}
\]
where $Z=\Fc^\Tc(C)$. In other words, $Z=\Tc(C)$ viewed
as a comonoid in the category $\Cs^{\Tc}$.

It simply remains to compose with the isomorphism
 \[
( \Ms^{\Sc})_{\hZ}=\HTS
 \]
 of Proposition~\ref{p:hopf-mod} to obtain the desired equivalence. We check that
 the resulting functors $\Ac$ and $\Bc$ are as stated.
 
 According to~\eqref{e:Dcoalg},
 \[
 (\Fc^{\Sc})_{\chi_C}(M,c) = \bigl(\Fc^{\Sc}(M),d\bigr) = \bigl(\Sc(M),\mu_M,d\bigr)
 \]
 where $d$ is the composite
 \[
 d: \Fc^{\Sc}(M) \map{\Fc^{\Sc}(c)} \Fc^{\Sc}(C\bdot M) \map{\chi_{C,M}} Z\bdot \Fc^{\Sc}(M).
\]
The map $\chi_{C,M}$ is described in the proof of Proposition~\ref{p:bimonad-adj};
it follows that $d$ coincides with~\eqref{e:hopfP}. Thus $\Ac$ is as stated. (This also shows that $\Ac(M,c)$ is a Hopf $(\Tc;\Sc,Z)$-module, as announced above.)
 
On a Hopf $(\Tc;\Sc,Z)$-module $(M,\sigma,\zeta)$, the functor
$(\Uc^{\Sc})^{\rho_C}$ is the equalizer of the pair~\eqref{e:eq-coalg}, which in the
present situation consists of the maps
\[
\hC\Uc^{\Sc}(M,\sigma)\map{\delta_{\Uc^{\Sc}(M,\sigma)}} \hC\hC\Uc^{\Sc}(M,\sigma) \map{\hC((\rho_C)_{(M,\sigma)})} \hC\Uc^{\Sc}\hZ(M,\sigma)
\]
and
\[
\hC\Uc^{\Sc}(M,\sigma)\map{\hC\Uc^{\Sc}(d)} \hC\Uc^{\Sc}\hZ(M,\sigma).
\] 
We have $\hC\Uc^{\Sc}(M,\sigma)=C\bdot M$ and $\delta_{\Uc^{\Sc}(M,\sigma)}=\Delta\bdot\id$.
The transformation $\rho_C$ is identified in Lemma~\ref{l:comodule-colax2}.
Note that in the present situation $\eta_C=\iota_C$ (Section~\ref{ss:bimonad-adj})
and $\chi=\id$ (proof of Proposition~\ref{p:bimonad-adj}).
It follows that the above pair coincides with~\eqref{e:hopfQ}
and therefore $\Bc$ is as stated.
\end{proof}
 
\begin{remark}\label{r:weaker-hopf}
Theorem~\ref{t:hopf} holds under weaker assumptions.
Namely, it suffices to require existence  in $\Ms$ and preservation (under
$\hC$, $\hC^2$, and $\Sc$)
of the
equalizers of the parallel pair~\eqref{e:hopfQ}, for each
Hopf $(\Tc;\Sc,Z)$-module $(M,\sigma,\zeta)$.
As explained in the above proof, this pair is~\eqref{e:eq-coalg} for 
the comonads $\Cc=\hC$
and $\Dc=\hD$ and the adjunction $(\Hc,\Kc)=(\Fc^{\Sc},\Uc^{\Sc})$.
Starting with the Conjugate Comonadicity Theorem (Theorem~\ref{t:descent}), all results leading to Theorem~\ref{t:hopf} can be generalized by relaxing the corresponding hypotheses on equalizers and using, when relevant, the proof of Beck's theorem rather than the theorem itself.
\end{remark}

As with Theorems~\ref{t:hopf-colax}--\ref{t:hopf-colax-conv},
stronger hypotheses yield a somewhat simpler statement.

\begin{theorem}[Fundamental Theorem of generalized Hopf modules, second version]
\label{t:hopf-conv}
In the above situation, assume that:
\begin{itemize}
\item The category $\Ms$ has equalizers of all coreflexive pairs.
\item The functors $\hC$ and $\Sc$ preserve such equalizers.
\item The functor $\Sc$ is conservative.
\end{itemize}
Then the conclusion of Theorem~\ref{t:hopf} holds. 
\end{theorem}
%

We provide a statement regarding the necessity of the hypotheses on the functor $\Sc$
in the Fundamental Theorem. To this end, let $C=\unit$ be the trivial comonoid in $\Cs$.
Then $\Ms_{\hC}=\Ms$ and $Z=\Tc(\unit)$.

\begin{proposition}\label{p:hopf-conv}
In the above situation, consider the following hypotheses.
\begin{enumerate}[(i)]
\item The functor $\Ac:  \Ms  \to \HTS$ is an equivalence.
\item The functor $\hZ$ preserves all existing equalizers in $\Ms^{\Sc}$.
\item The transformation $\Gb_{\unit}$ is invertible.
\end{enumerate}
We then have the following statements.
\begin{itemize}
\item If \textup{(i)} holds, then the functor $\Sc$ is conservative..
\item Assume \textup{(i)} and \textup{(ii)} hold.
Let $(f,g)$ be a parallel pair in $\Ms$ such that $(f,g)$ has an equalizer in $\Ms$ and
$\bigl(\Sc(f),\Sc(g)\bigr)$ has an equalizer in $\Ms$.
Then $\Sc$ preserves the equalizer of  $(f,g)$.
\item Assume \textup{(i)} and \textup{(iii)} hold. Then the category $\Ms$ has equalizers of all coreflexive pairs
which are $\Sc$-split in $\Ms^\Sc$.
\end{itemize}
\end{proposition}
\begin{proof}
The proof of Theorem~\ref{t:hopf} shows that if $\Ac$ is an equivalence,
then so is the functor $(\Fc^{\Sc})_{\chi_C}$. 
Also, by Corollary~\ref{c:Galois-Hopf}, the 
invertibility of the transformation $\Gb_{\unit}$ implies that of 
$\Hb_{\unit}$.
We may then apply
Proposition~\ref{p:hopf-colax-conv} to conclude a number of
properties satisfied by $\Fc^\Sc$. The desired
properties of $\Sc=\Uc^{\Sc}\Fc^\Sc$ follow from these
in view of the dual of Lemma~\ref{l:creation} and the fact that $\Uc^{\Sc}$ is conservative.
\end{proof}

\subsection{Application: idempotent monads}\label{s:idempotent}

Recall that a monad $(\Sc,\mu,\iota)$ is said to be \emph{idempotent} if
$\mu:\Sc^2\to\Sc$ is an invertible transformation.

Consider the data of Section~\ref{ss:fundamental} in the case when $\Cs$ is the 
trivial monoidal category (the one-arrow category).
The bimonad $\Tc$ and the comonoids $C$ and $Z=\Tc(C)$ are necessarily trivial,
while $\Sc$ is an arbitrary monad on a category $\Ms$. Diagram~\eqref{e:hopf-mod}
commutes automatically, and
the category $\HTS$ of Hopf modules is thus the category of $\Sc$-algebras.

The functors $\Ac$ and $\Bc$ identify respectively
with $\Fc^\Sc$ and $\Uc^\Sc$. (Note that both maps in~\eqref{e:hopfP} are identities.)

Both maps in the pair~\eqref{e:hopfQ} are identities, so existence and
preservation of the equalizers of such pairs is trivially satisfied.
This much implies that $\Fc^\Sc\Uc^\Sc\cong\Ic$ (for reasons that
can be traced back to Beck's Theorem). This recovers~\cite[Proposition~4.2.3]{Bor:1994ii}.

From~\eqref{e:Galois}
we see that the Galois map reduces to the multiplication of the monad $\Sc$. Therefore, the Galois map is invertible if and only
if the monad $\Sc$ is idempotent.

An idempotent monad $\Sc$ need not be conservative. 
(An example is the monad on the category of groups which sends a group
to its abelianization.) 
Assume that $\Sc$ is conservative. Then all hypotheses in Theorem~\ref{t:hopf}
are satisfied. The theorem then states that the
adjunction $(\Fc^\Sc,\Uc^\Sc)$ is an equivalence, and hence
$\Sc=\Uc^\Sc\Fc^\Sc\cong\Ic$. We have thus shown that
a conservative idempotent monad must be isomorphic to the identity.
Here is an elementary proof of this well-known fact:
from the unit axiom one deduces
$\mu^{-1} = \Sc(\iota)$; therefore $\iota:\Ic\to\Sc$ is an isomorphism.

\subsection{The Fundamental Theorem of
Brugui\`eres, Lack and Virelizier} \label{ss:BLV}

Let $\Tc$ be a bimonad on a monoidal category $\Cs$.
We specialize the situation of Section~\ref{ss:fundamental} by setting 
$\Ms=\Cs$, $\Sc=\Tc$, and $C=\unit$, the trivial comonoid.
The results of Section~\ref{ss:fundamental} then reduce to the Fundamental Theorem of
Brugui\`eres, Lack and Virelizier~\cite[Theorem~6.11]{BLV:2011}. 

First of all, as mentioned in Remark~\ref{r:hopf-mod}, a Hopf $(\Tc;\Sc,Z)$-module is
in this case a left Hopf $\Tc$-module in the sense of~\cite[Section~4.2]{BruVir:2007}
and~\cite[Section~6.5]{BLV:2011}.
Given such a Hopf module $M$, the Galois map  $\Gb_{\unit,M}$ is
\begin{equation}\label{e:BLV-fusion}
\Tc^2(M)=\Tc\bigl(\unit\bdot\Tc(M)\bigr) \map{\psi_{\unit,\Tc(M)}} \Tc(\unit)\bdot \Tc^2(M) \map{\id\bdot\mu_M} \Tc(\unit)\bdot \Tc(M).
\end{equation}
This is the left \emph{fusion operator} $H_{\unit,M}$ of~\cite[Section~2.6]{BLV:2011}; 
thus, the transformation $\Gb_{\unit}$ is invertible if and only if $\Tc$ is a left \emph{pre-Hopf monad} in the sense of~\cite[Section~2.7]{BLV:2011}.
Moreover, the pair~\eqref{e:hopfQ} is
\begin{equation}\label{e:BLV-Q}
M = \unit\bdot M \map{\iota_\unit \bdot\id} 
  Z \bdot M 
\qand
 M \map{\zeta} Z\bdot M
\end{equation}
where $Z=\Tc(\unit)$.
The equalizer of this pair is the \emph{coinvariant part} of the Hopf module
$(M,\sigma,\zeta)$ in the sense of~\cite[Section~6.4]{BLV:2011}. As mentioned in Remark~\ref{r:weaker-hopf}, the existence and preservation of these equalizers (plus the
fact that $\Sc=\Tc$ is conservative)
suffice
to derive the equivalence of categories from the invertibility of the Galois map,
as in Theorem~\ref{t:hopf}.
This is one of the implications in~\cite[Theorem~6.11]{BLV:2011}.
The converse implication is closely related to Proposition~\ref{p:hopf-conv}.
Brugui\`eres and Virelizier proved an earlier version of this result in
~\cite[Theorem~4.6]{BruVir:2007}.

\subsection{Sweedler's Fundamental Theorem}\label{ss:sweedler}

Assume that the monoidal category $\Cs$ is braided and $H$ is a bimonoid therein.
Let $\Tc$ be the functor $\hH$ as in Section~\ref{ss:comod-comonoid}.
Then $\Tc$ is a bimonad~\cite[Example~2.2]{BruVir:2007}. Let us go over
the ingredients and hypotheses of the Fundamental Theorem of Brugui\`eres, Lack and Virelizier. The discussion here is a simplified version of that in~\cite[Example~6.13]{BLV:2011}.

First of all, $Z=\Tc(\unit)=H$, and the Hopf $\Tc$-modules are the usual Hopf $H$-modules as in~\cite[Section~1.9]{Mon:1993}.

Since $\unit$ is a retract of $H$, the identity of $\Cs$ is a retract of the functor $\Tc$.
If a functor $\Fc$ admits a conservative retract $\Rc$
\[
\xymatrix@C+5pt{
\Fc \ar@<0.5ex>[r]^-{r}  & \Rc \ar@<0.5ex>[l]^-{s}
},
\qquad
rs=\id,
\]
then $\Fc$ is conservative. (Given $f:A\to B$ with $\Fc(f)$ invertible,
we have $\Rc(f)^{-1}=r_A\Fc(f)^{-1} s_B$.) 
Therefore, $\Tc$ is conservative.

The pair~\eqref{e:BLV-Q} is
\[
M = \unit\bdot M \map{\iota \bdot\id} 
  H \bdot M 
\qand
 M \map{\zeta} H\bdot M
\]
Its equalizer, when existing, is the coinvariant part $\MH$ of $M$.

The map $\Gb_{\unit,M}$ in~\eqref{e:BLV-fusion} is
\[
H\bdot H\bdot M \map{\Delta\bdot\id\bdot\id} H\bdot H\bdot H\bdot M  \map{\id\bdot \mu\bdot\id} H\bdot H\bdot M.
\]
It is known that the following statements are equivalent:
\begin{itemize}
\item $\Gb_{\unit,M}$ is invertible for every $M$.
\item $\Gb_{\unit,\unit}$ is invertible.
\item The bimonoid $H$ is a Hopf monoid.
\end{itemize}
(See~\cite[Example~2.1.2]{Sch:2004} and~\cite[Proposition~5.4]{BLV:2011} for a more general result.)
Assume these statements hold, and let $\apode$ be the antipode of $H$.
It then turns out that the composite map
\begin{equation}\label{e:idempotent}
M \map{\zeta} H\bdot M \map{\apode\bdot\id} H\bdot M\map{\sigma} M
\end{equation}
is idempotent. Brugui\`eres, Lack and Virelizier point out that the coinvariant part 
of $M$ exists if and only if this idempotent splits in the category $\Cs$. In this case,
the pair~\eqref{e:BLV-Q} is
part of a split cofork, with splitting data
\[
 \xymatrix@C+15pt{
\MH & M \ar[l]_-{\sigma(\apode\bdot\id)\zeta} & H\bdot M \ar[l]_-{\sigma(\apode\bdot\id)}.
}
\]
The equalizer of the pair~\eqref{e:BLV-Q} is then preserved by all functors
and the hypotheses of the Fundamental Theorem of Brugui\`eres, Lack and Virelizier
are satisfied.

In conclusion, if $H$ is a Hopf monoid and for all Hopf $H$-modules $M$
the idempotent~\eqref{e:idempotent} splits in $\Cs$, there is 
an equivalence between $\Cs$ and the category of 
Hopf $H$-modules. This is Sweedler's Fundamental Theorem,
in the general context of braided monoidal categories.
In particular, the theorem holds if all equalizers exist in $\Cs$.
In this generality, the Fundamental Theorem has been stated in~\cite[Theorem~3.4]{Tak:1999}, and (with extra assumptions) in~\cite[Theorem~1.1]{Lyu:1995}.
Note that the hypotheses are satisfied if $\Cs$ is the category of modules over 
a commutative ring. In this situation, the Fundamental Theorem appears in~\cite[Theorem~15.5]{BW:2003}.
Sweedler originally proved his theorem in the context of vector spaces~\cite[Theorem~4.1.1]{Swe:69}.

Sweedler's Fundamental Theorem has been extended in other directions
in recent papers. These include
the work of L\'opez-Franco~\cite{Lop:2009} and of Mesablishvili and Wisbauer~\cite{MesWis:2011}.
The contexts considered by these authors are different from the one considered here.

\section{Application: Schneider's theorem}\label{s:schneider}

Let $\field$ be a commutative ring and $H$ a bialgebra over $\field$. 
Let $A$ be a left $H$-comodule-algebra
(a monoid in the monoidal category of left $H$-comodules)~\cite[Definition~4.1.2]{Mon:1993}. We use $(\mu,\iota,\nu)$
to denote its structure. We employ Sweedler's notation
\[
\nu(a)= \sum a_1\otimes a_2
\]
for the comodule structure map $\nu:A\to H\otimes A$ (a morphism of $\field$-algebras).
The $\field$-subalgebra of \emph{$H$-coinvariants} is
\[
\AH := \{a\in A \mid \nu(a) = 1\otimes a\}.
\]
Let $B$ be any $\field$-subalgebra of $\AH$.

Out of this data we build the ingredients for an application of 
the Fundamental Theorem of generalized Hopf modules
and the other results of Section~\ref{ss:fundamental}.

\subsection{The bimonad and the comodule-monad}\label{ss:schneider-bimonad}

We let $\Cs$ be the category of $\field$-modules and define $\Tc:\Cs\to\Cs$ by
\[
\Tc(X) = H\otimes X
\]
(tensor product over $\field$). Since $(\Cs,\otimes)$ is braided monoidal and $H$ is a bimonoid therein, $\Tc$ is a bimonad~\cite[Example~2.2]{BruVir:2007}.

Next, we let $\Ms$ be the category of left $B$-modules. Then $\Cs$ acts on $\Ms$ by
\[
(X,M) \mapsto X\otimes M.
\]
The action of $B$ on $X\otimes M$ is on the right factor:
\[
b\cdot(x\otimes m) = x\otimes bm.
\]

\begin{lemma}\label{l:schneider1}
The map $\nu:A\to H\otimes A$ is a morphism of $B$-bimodules.
\end{lemma}
\begin{proof}
Take $a\in A$ and $b\in B$. Write $\nu(a)=\sum a_1\otimes a_2$.
Since $B\subseteq A^{H}$, we have $\nu(b)=1\otimes b$. Then, since $\nu$
is a morphism of algebras,
\[
\nu(ba)=\nu(b)\nu(a)=\sum a_1\otimes ba_2 = b\cdot \nu(a).
\]
Similarly, $\nu(a b)=\sum a_1\otimes a_2b$.
\end{proof}

Define $\Sc:\Ms\to \Ms$ by
\[
\Sc(M) = A\otimes_{B} M.
\]
The action of $B$ on $A\otimes_{B} M$ is on the left factor:
\[
b\cdot(a\otimes_{B} m) = ba\otimes_{B} m.
\]
Since $B$ is a subalgebra of $A$, we may view $A$ as a $B$-bimodule.
The category of $B$-bimodules is monoidal under $\otimes_{B}$ and
it acts on $\Ms$, again by means of $\otimes_{B}$. Moreover,
$A$ is a monoid in this category, and hence $\Sc$ is a monad.

Let $X$ and $Y$ be $\field$-modules, $M$ a left $B$-module, and $N$ a $B$-bimodule.
We will make use of the map
\begin{align*}
(Y\otimes N)\otimes_{B}(X\otimes M) &\to (Y\otimes X)\otimes(N\otimes_{B} M)\\
\quad
(y\otimes n)\otimes_{B}(x\otimes m) &\mapsto (y\otimes x)\otimes(n\otimes_{B} m).
\end{align*}
This is a well-defined isomorphism of left $B$-modules. We employ this map and the $H$-comodule
structure of $A$ to build the composite
\begin{equation}\label{e:schneider-chi}
A\otimes_{B}(X\otimes M) \map{\nu\otimes_{B}\id} (H\otimes A)\otimes_{B}(X\otimes M)
\map{\cong} (H\otimes X)\otimes(A\otimes_{B} M).
\end{equation}
By Lemma~\ref{l:schneider1}, this is a well-defined morphism of left $B$-modules.
It defines a transformation
\[
\chi_{X,M}: \Sc(X\otimes M) \to \Tc(X)\otimes \Sc(M).
\]

\begin{lemma}\label{l:schneider2}
In this manner, $\Sc$ is a comodule-monad over the bimonad $\Tc$.
\end{lemma}
\begin{proof}
Axioms~\eqref{e:comodule1} and~\eqref{e:comodule2}
follow from coassociativity and counitality of $\nu$.
Axioms~\eqref{e:comod-monad1} and~\eqref{e:comod-monad2}
follow from the fact that $\nu$ is a morphism of algebras.
\end{proof}

\subsection{The Hopf modules}\label{ss:schneider-hopf}

Let $(Z,\action,\Delta,\epsilon)$ be a $\Tc$-algebra-comonoid in $\Cs$ (Definition~\ref{d:comTalg}).
First of all, as a $\Tc$-algebra, $Z$
carries a structure of left $H$-module
\[
\action:H\otimes Z\to Z.
\]
Then, as a comonoid, it carries a structure $(\Delta,\epsilon)$ of coalgebra over 
$\field$. Finally, $\action$ should be a morphism of coalgebras, or equivalently
$\Delta$ and $\epsilon$ should be morphisms of $H$-modules. 
The totality is called an \emph{$H$-module-coalgebra} structure on $Z$;
it is the same as a comonoid in the monoidal category of left $H$-modules.

Now consider Hopf $(\Tc;\Sc,Z)$-modules $(M,\sigma,\zeta)$
in $\Ms$ (Definition~\ref{d:hopf-mod}).
Then $M$ is first of all a left $B$-module, but the $\Sc$-algebra structure
\[
\sigma: A\otimes_{B} M \to M
\]
turns it in fact into a left $A$-module, from which the $B$-module structure is inherited.
In addition, 
\[
\zeta:M\to Z\otimes M
\] 
is a left $Z$-comodule structure, and diagram~\eqref{e:hopf-mod} becomes
\[
\xymatrix@C+10pt@R-5pt{
A\otimes_{B} M \ar[r]^-{\sigma} \ar[d]_{\id\otimes_{B}\zeta} & M \ar[dd]^{\zeta} \\
A\otimes_{B}(Z\otimes M) \ar[d]_{\chi_{Z,M}} & \\
(H\otimes Z)\otimes(A\otimes_{B} M)  \ar[r]_-{\action \otimes \sigma} & Z\otimes M 
}
\]
where $\chi_{Z,M}$ is as in~\eqref{e:schneider-chi}.
The commutativity of this diagram is equivalent to that of the following.
\begin{equation}\label{e:schneider-hopfmod}
\begin{gathered}
\xymatrix@C+10pt@R-5pt{
A\otimes M \ar[r]^-{\sigma} \ar[d]_{\nu\otimes \zeta} & M \ar[dd]^{\zeta} \\
H\otimes A \otimes Z\otimes M \ar[d]_{\cong} & \\
H\otimes Z \otimes A\otimes M  \ar[r]_-{\action \otimes \sigma} & Z\otimes M 
}
\end{gathered}
\end{equation}
This diagram expresses the fact that $\zeta$ is a morphism of $A$-modules, or
equivalently that $\sigma$ is a morphism of $H$-comodules.

In summary, the starting data consists of:
\[
\text{a bialgebra $H$,
an $H$-comodule-algebra $A$,
and an $H$-module-coalgebra $Z$.}
\]
Then a Hopf $(\Tc;\Sc,Z)$-module $(M,\sigma,\zeta)$ consists of
\[
\text{a left $A$-module $(M,\sigma)$ and a left $Z$-comodule $(M,\zeta)$,
linked by~\eqref{e:schneider-hopfmod}.}
\]
This is precisely a left $(A,Z)$-Hopf
module in the sense of Doi~\cite[Section~1]{Doi:1992} and
Koppinen~\cite[Section~3]{Kop:1995}. Such objects are sometimes called
\emph{Doi-Koppinen Hopf modules} in recent literature.
In the special case when $Z=H$, they go back at least to
Doi~\cite{Doi:1983} and
they are precisely the left $(A,H)$-Hopf modules of Schneider~\cite[page~177]{Sch:1990}. The notion evolved from the original Hopf
modules of Sweedler~\cite[Section~4.1]{Swe:69} (the case $A=H$)
and the relative Hopf modules of Takeuchi~\cite[Section~1]{Tak:1979} 
(the case in which $A$ is a coideal subalgebra of $H$).

We let $\HHA$
denote the category of Doi-Koppinen Hopf modules. Equivalently, $\HHA=\HTS$ is the
category of Hopf $(\Tc;\Sc,Z)$-modules.

\subsection{The Galois map}\label{ss:schneider-galois}

We make the Galois map $\Gb_{X,M}$
associated to the bimonad $\Tc$ and the comodule-monad $\Sc$ explicit.
According to~\eqref{e:Galois} and the definitions in Section~\ref{ss:schneider-bimonad}, this map is the following composite.
\[
{
\def\objectstyle{\scriptstyle}
\def\labelstyle{\scriptstyle}
\xymatrix@R-5pt{
\Sc\bigl(X\otimes\Sc(M)\bigr)\ar[rr]^-{\chi_{X,\Sc(M)}}  \ar@{=}[d] &
& \Tc(X)\otimes \Sc^2(M) \ar[r]^-{\id\otimes\mu_M}  \ar@{=}[d] &
\Tc(X)\otimes \Sc(M) \ar@{=}[d] \\
A\otimes_{B}\bigl(X\otimes(A\otimes_{B} M)\bigr)
\ar[r]_-{\nu\otimes_{B}\id} &
(H\otimes A)\otimes_{B}\bigl(X\otimes(A\otimes_{B}M)\bigr)
\ar[r]_-{\cong} &
(H\otimes X)\otimes(A\otimes_{B}A\otimes_{B} M)
\ar[r]_-{\id\otimes(\mu\otimes_{B}\id)} &
(H\otimes X)\otimes(A\otimes_{B} M)
}
}
\]
Explicitly,
\begin{equation}\label{e:schneider-gal}
\Gb_{X,M}\Bigl(a\otimes_B\bigl(x\otimes(a'\otimes_B m)\bigr)\Bigr) =
(a_1\otimes x)\otimes(a_2a'\otimes_B m),
\end{equation}
where $a,a'\in A$, $x\in X$, $m\in M$, and $\nu(a)= \sum a_1\otimes a_2$.
In the special case when $X=\field$ and $M=B$, the map is
\begin{equation}\label{e:schneider-can}
A\otimes_{B}A \to H\otimes A,
\quad
a\otimes_{B} a' \mapsto \sum a_1\otimes a_2a'.
\end{equation}
Thus, the map $\Gb_{\field,B}$ coincides with Schneider's \emph{canonical map}~\cite[page~180]{Sch:1990} (the left version of it). It is sometimes called the \emph{$H$-Galois map}
of the extension $B\subseteq A$~\cite[Chapter~8]{Mon:1993}. 
This map and its prominent role
in \emph{Hopf Galois theory} can be traced back to~\cite{CHR:1965},~\cite{CS:1969}
and~\cite{KT:1981}. For a recent survey on Hopf Galois theory, see~\cite{Mon:2009}.

\begin{lemma}\label{l:schneider-galois}
Let $X$ be a $\field$-module. Consider the following statements.
\begin{enumerate}[(i)]
\item $\Gb_{\field,B}$ is invertible.
\item $\Gb_{X,B}$ is invertible.
\item $\Gb_{X,M}$ is invertible for all left $B$-modules $M$.
\end{enumerate}
We have \textup{(i) $\Rightarrow$ (ii) $\iff$ (iii).}
\end{lemma}
\begin{proof}
Note from~\eqref{e:schneider-gal}--\eqref{e:schneider-can} that $\Gb_{X,B}$ is a morphism of right $B$-modules, and moreover
\[
\Gb_{X,M}\cong X\otimes \Gb_{\field,B}\otimes_{B} M \cong \Gb_{X,B}\otimes_B M.
\]
The result follows.
\end{proof}

\subsection{The comodules and the functors}\label{ss:schneider-comodule}

Let $C$ be a coalgebra over $\field$ and let $Z$ be the free $\Tc$-algebra-comonoid on $C$~\eqref{e:freecomTalg}. Then 
\[
Z=\Tc(C)= H\otimes C
\]
and the $H$-module-coalgebra structure is as follows: $H$ acts on the left
component (so $Z$ is the free $H$-module on $C$), and the coalgebra
structure is the tensor product of the coalgebra structures of $H$ and $C$.

The next ingredient to describe is the category $\Ms_{\hC}$. An object of
this category is a left $B$-module $M$ with a left $C$-comodule structure
\[
\cgamma: M\to C\otimes M
\]
that is a morphism of $B$-modules, where $B$ acts on $M$ only. In other words,
the following diagram should commute
\[
\xymatrix@C+5pt@R-5pt{
B\otimes M \ar[r]^-{\ctau} \ar[d]_{\id\otimes\cgamma} & M \ar[dd]^{\cgamma} \\
B\otimes C\otimes M \ar[d]_{\cong} & \\
C\otimes B\otimes M \ar[r]_-{\id\otimes\ctau} & C\otimes M
}
\]
where $\ctau$ denotes the $B$-module structure. Contrast with~\eqref{e:schneider-hopfmod}. We refer to such objects as \emph{$(B,C)$-bimodules}.

We let $\BBC$ denote the category of $(B,C)$-bimodules; that is, 
$\BBC=\Ms_{\hC}$ is the category of $C$-comodules in $\Ms$.

The last ingredients to describe are the functors
\[
\xymatrix@C+10pt{
 \BBC \ar@<0.5ex>[r]^-{\Ac} &  \HHA \ar@<0.5ex>[l]^-{\Bc} 
}
\]
of Section~\ref{ss:fundamental}. They take the following form.

\smallskip

Given a $(B,C)$-bimodule $(M,\cgamma)$, 
\[
\Ac(M,\cgamma)= (A\otimes_B M, \mu\otimes_B\id,\cdelta)
\]
where $\cdelta$ is the composite
\[
A\otimes_B M \map{\id\otimes_B\cgamma} A\otimes_B(C\otimes M) 
\map{\chi_{C,M}} (H\otimes C)\otimes(A\otimes_B M) = Z\otimes(A\otimes_B M)
\]
and $\chi_{C,M}$ is as in~\eqref{e:schneider-chi}.

Given a Doi-Koppinen Hopf module $(M,\sigma,\zeta)$, $\Bc(M,\sigma,\zeta)$
is the equalizer of the pair 
\begin{equation}\label{e:schneiderQ}
\begin{gathered}
 C\otimes M \map{\Delta\otimes\id} C\otimes C\otimes M=C\otimes \field\otimes C\otimes M
 \map{\id\otimes\iota\otimes\id\otimes\id} 
 C\otimes H\otimes C\otimes M \\
C\otimes M \map{\id\otimes\zeta} C\otimes Z\otimes M= C\otimes H\otimes C\otimes M.
\end{gathered}
\end{equation}
in the category $\BBC$. Here $\iota$ denotes the unit map of the algebra $H$.
In particular, when $C=\field$ is the trivial $\field$-coalgebra,
\[
\Bc(M, \sigma, \zeta)  =  \{m\in M \mid  \zeta(m) = 1 \otimes m\}.
\]

\subsection{A generalization of Schneider's theorem}\label{ss:schneider-theorem}

We now make use of all the data in Sections~\ref{ss:schneider-bimonad}--\ref{ss:schneider-comodule}. In particular, we are given:
\begin{itemize}
\item A commutative ring $\field$ and a $\field$-bialgebra $H$.
\item A left $H$-comodule-algebra $A$ and 
a subalgebra $B$ of $\AH$.
\item A $\field$-coalgebra $C$ and $Z=H\otimes C$, an $H$-module-coalgebra as in Section~\ref{ss:schneider-comodule}.
\end{itemize}

\begin{theorem}\label{t:schneider}
Assume that:
\begin{itemize}
\item $A$ is faithfully flat as right $B$-module.
\item $C$ is flat as $\field$-module.

\end{itemize}
Then the following statements are equivalent.
\begin{enumerate}[(i)]
\item The functors $\Ac$ and $\Bc$ from Section~\ref{ss:schneider-comodule}
form an adjoint equivalence 
\[
 \BBC \cong  \HHA
\]
between the category of $(B,C)$-bimodules and the category of
Doi-Koppinen $(A,Z)$-Hopf modules.
\item The Galois map
\[
\Gb_{C,B}: A\otimes_B(C\otimes A) \to Z\otimes A
\]
is invertible.
\end{enumerate}
In particular, if the canonical map~\eqref{e:schneider-can} is invertible, then
the category equivalence holds.
\end{theorem}
\begin{proof}
The category $\Ms$ of left $B$-modules, being abelian
has all equalizers, and the functor $\hC=C\otimes(-)$ preserves them because $C$ is flat over $\field$. Faithful flatness of $A$ over $B$ guarantees that the functor $\Sc$
is exact and faithful, hence conservative by Corollary~\ref{c:con-fai}.
It then follows from Theorem~\ref{t:hopf-conv} that the stated adjoint
equivalence holds if and only if the Galois map $\Gb_{C,M}$ of
Section~\ref{ss:schneider-galois} is invertible for all left $B$-modules $M$.
By Lemma~\ref{l:schneider-galois}, this holds if and only if $\Gb_{C,B}$ is
invertible, and in particular if $\Gb_{\field,B}$ is invertible.
\end{proof}

\begin{remark}\label{r:weaker-schneider}
According to Remark~\ref{r:weaker-hopf}, Theorem~\ref{t:schneider} holds under weaker assumptions.
Namely, it suffices to require the preservation (under
$C\bdot(-)$, $C\bdot C\bdot(-)$, and $A\otimes_B(-)$)
of the
equalizer of the parallel pair~\eqref{e:schneiderQ} for each Doi-Koppinen
Hopf module $(M,\sigma,\zeta)$.
\end{remark}

We provide a converse statement. To this end, let $C=\field$ be the
trivial coalgebra. Then $ \BBC=\Ms$ is the category of left $B$-modules
and $Z=H$.

\begin{theorem}\label{t:schneider-conv}
Let $H$ be a $\field$-bialgebra.
Consider the following statements:
\begin{enumerate}[(i)]
\item $A$ is faithfully flat as right $B$-module
and the canonical map~\eqref{e:schneider-can} is invertible. 
\item The functor
$
\Ac: \Ms \to \HHAH
$
is an equivalence.
\item The functor $\Sc=A\otimes_B(-):\Ms\to\Ms$ is faithful.
\end{enumerate}
Then \textup{(i) $\Rightarrow$ (ii) $\Rightarrow$ (iii).}
If in addition $H$ is flat as a $\field$-module, then 
\textup{(i)} and \textup{(ii)} are equivalent.
\end{theorem}
\begin{proof}
The implication (i) $\Rightarrow$ (ii) is a special case of Theorem~\ref{t:schneider} (the case $C=\field$).

To prove that (ii) $\Rightarrow$ (iii), note that by Proposition~\ref{p:hopf-conv}, the functor $\Sc$ is conservative.  Since it is also right exact, it is faithful, by Corollary\ref{c:con-fai}.

Now assume that $H$ is a flat $\field$-module and that (ii) holds; then, as just noted, (iii) also holds.  Moreover, since $H$ is flat over $\field$, the functor $\hZ$ preserves all
equalizers in $\Ms^{\Sc}$. Note also that $\Ms$ and $\Ms^{\Sc}$ possess
all equalizers, since they are the categories of left modules over $B$ and $A$,
respectively. Hence Proposition~\ref{p:hopf-conv} applies again to yield that $\Sc$ preserves
all equalizers in $\Ms$. This means that $A$ is a flat right $B$-module, and hence faithfully flat, by (iii). 
Finally, our discussion shows that the hypotheses of Theorem~\ref{t:hopf-conv} are satisfied. It then follows that the map $\Gb_{\field,M}$ is invertible
for all left $B$-modules $M$, and in particular, the Galois map  $\Gb_{\field,B}$
is likewise invertible.
\end{proof}

%
%

\subsection{Further remarks on Schneider's theorem}\label{ss:schneider-further}

First of all, we point out that while in principle Theorem~\ref{t:schneider} 
only requires that $B$ be a subalgebra of $\AH$, the remaining hypotheses
imply that $B=\AH$.
More precisely, we have the following result.

\begin{proposition}\label{p:schneider-further}
Consider the following data:
\begin{itemize}
\item A commutative ring $\field$ and a $\field$-bialgebra $H$.
\item A left $H$-comodule-algebra $A$ and 
a subalgebra $B$ of $\AH$.
\end{itemize}
Assume that:
\begin{itemize}
\item $A$ is faithfully flat as right $B$-module.
\item The canonical map~\eqref{e:schneider-can} is invertible. 
\end{itemize}
Then $B=\AH$.
\end{proposition}
\begin{proof}
Since $B\subseteq \AH$, the canonical map factors
through $A\otimes_{\AH} A$:
\[
\xymatrix@-5pt{
A\otimes_{B} A \ar[r] \ar@{->>}[d]& H\otimes A\\
A\otimes_{\AH} A \ar@{-->}[ru]
}
\]
Invertibility of the canonical map then implies that the vertical map is injective.
Hence, if $x\in\AH$, then
\[
x\otimes_B 1 - 1\otimes_B x =0.
\]
It is known that faithful flatness of $A$ over $B$ then implies that $x\in B$; see for instance~\cite[Lemma~3.8]{CR:1965}. We provide an argument below for completeness.

Consider the maps
\[
\xymatrix@C+5pt{
 B \ar[r]^{e} & A \ar@<0.5ex>[r]^-{f} \ar@<-0.5ex>[r]_-{g} & A\otimes_B A
}
\]
given by
\[
e(b) = b,\quad f(a) = a\otimes_B 1,\qand g(a) = 1\otimes_B a.
\]
By the preceding $x$ belongs to equalizer of the pair $(f,g)$. We show below that
the above diagram is an equalizer, whence $x$ must belong to $B$. This proves that
$\AH=B$ as needed.

The diagram is clearly a cofork: $fe=ge$. Applying the functor $\Sc=A\otimes_B(-)$
to it we obtain the cofork
\[
\xymatrix@C+5pt{
 A \ar[r]^-{\Tilde{e}} & A\otimes_B A \ar@<0.5ex>[r]^-{\Tilde{f}} \ar@<-0.5ex>[r]_-{\Tilde{g}} & A\otimes_B A\otimes_B A,
}
\]
where the maps are given by
\[
\Tilde{e}(a) = a\otimes_B 1,\quad \Tilde{f}(a\otimes_B a') = a\otimes_B a'\otimes_B 1,\quad \Tilde{g}(a\otimes_B a') = a\otimes_B 1\otimes_B a'.
\]
This cofork is split by the maps
\[
 \xymatrix@C+5pt{
A & A\otimes_B A \ar[l]_-{p} & A\otimes_B A\otimes_B A \ar[l]_-{q},
}
\]
where 
\[
p(a\otimes_B a') =aa' \qand q(a\otimes_B a'\otimes_B a'')=aa'\otimes_B a''.
\]
Indeed, we have $p\Tilde{e}=\id$, $q\Tilde{g}=\id$, and $q\Tilde{f}=\Tilde{e}p$.
(One may also note that this is a special case of the dual of the first split fork
in~\cite[Section V.2.2, Example~1]{MS:2004}.)
Thus this new cofork is an equalizer.

Finally, since $A$ is faithfully flat as right $B$-module, $\Sc$ reflects equalizers,
and hence the original cofork is an equalizer.
\end{proof}

Theorem~\ref{t:schneider-conv} is
due to Schneider; it is the bi-implication $(2) \Leftrightarrow (5)$ 
in~\cite[Theorem~I]{Sch:1990}. It is sometimes known as Schneider's \emph{Structure
Theorem} for Hopf modules.  

Schneider assumes that the bialgebra $H$ possess an antipode in~\cite[Theorem ~3.7]{Sch:1990} and a bijective antipode in~\cite[Theorem~I]{Sch:1990}.
For the statements we have provided, these hypotheses are not necessary.
(Schneider's results involve several other statements, including other conditions
on $A$ and $\Gb_{\field,B}$ under which the equivalence of categories in
Theorem~\ref{t:schneider-conv} holds.)
Schneider also assumes that $B=\AH$, but as pointed
out in Proposition~\ref{p:schneider-further}, this results in no loss in generality.
In a different direction, Theorem~\ref{t:schneider} generalizes 
the implication $(5) \Rightarrow (2)$ in Schneider's theorem by
introducing the coalgebra $C$, which in his context must be trivial
but in ours may be arbitrary. 

Schneider's theorem has been generalized in a number of directions different from our own.
For more on this one may consult~\cite{Sch:2004},~\cite{SS:2005}, and~\cite{Mon:2009}, apart from Schneider's original paper. One such generalization
is~\cite[Theorem~3.7]{Sch:1990}. This last result is likely to admit a
generalization to the context of bimonads along the lines of Theorem~\ref{t:hopf},
but this appears to be beyond the scope
of this paper. We plan to pursue this question in the future.

If in addition to $C=\field$ we further specialize $A=H$, we recover
Sweedler's Fundamental Theorem.
See Section~\ref{ss:sweedler} for more details.

\subsection{Related cases and historical remarks}\label{ss:schneider-historical}

On~\cite[page~169]{Sch:1990}, Schneider lists several earlier special cases 
of his theorem in the literature. We provide a longer account here.

Let $H$ be a Hopf $\field$-algebra. We may use the antipode of $H$ to turn
left $H$-comodules into right ones and vice versa.
If $H$ is a finitely generated projective $\field$-module,
then an $H$-comodule is an $H^*$-module (with $H^*=\Hom_{\field}(H,\field)$
the dual Hopf algebra), and a Doi-Koppinen $(A,H)$-Hopf module is simply
a module over the \emph{smash product} $A\# H^*$. (For the definition of
smash product, see~\cite[Definition~4.1.3]{Mon:1993}.) Such a module $M$
can be described as a left $A$-module $M$ which is also a left $H^*$-module
with a \emph{semilinear} action of $H^*$; that is,
\begin{equation}\label{e:semilinear}
g\cdot am = \sum (g_1\cdot a) (g_2\cdot m)
\end{equation}
for all $g\in H^*$, $a\in A$, and $m\in M$. Now, since $A$ is an $H$-comodule-algebra, it is itself such a Hopf module, and the structure as a left $A\# H^*$-module
then yields a $\field$-algebra homomorphism
\begin{equation}\label{e:smash-end}
A\# H^* \to \End_B(A)
\end{equation}
with $B=\AH$ and $A$ viewed as a right $B$-module. This map is dual to the
Galois map~\eqref{e:schneider-can}, in the sense that each can be obtained
from the other by an application of $\Hom_A(-,A)$, so the invertibility of one map
is equivalent to that of the other.

In the earliest treatments of Hopf-Galois theory, $A$ and $H$ are commutative
and $A$ is a faithfully flat $\field$-module, from which it follows that $A$ is a faithfully
projective $\field$-module and $B=\field$. In that special case, the category equivalence of
Theorem~\ref{t:schneider-conv} is immediate from the isomorphism~\eqref{e:smash-end} and the Morita theorems~\cite[Theorems~9.3 and~9.6, Corollary~9.7]{CS:1969}.
In a sense, all the various forms of a Fundamental Theorem for relative Hopf modules
can be viewed as the progeny of this simple observation.

See~\cite[Theorem~8.3.3]{Mon:1993} for a generalization of the above argument to not 
necessarily commutative $H$ and $A$ (in which case possibly $B\neq\field$).
A somewhat similar approach to the Fundamental Theorem appears in the
theory of \emph{corings}~\cite{BW:2003,Cae:2004}, where it is observed that the
Galois isomorphism~\eqref{e:schneider-can} is one of $A$-corings, and the
category of comodules over the coring $H\otimes A$ is equivalent to the
category $\HHAH$. In this formulation, the role of the Morita theorems is
played by coring descent theory, which, for $A$ a faithfully flat right $B$-module,
identifies the category of left comodules over the coring $A\otimes_B A$ as equivalent
to the category of left $B$-modules. 
For a very readable discussion of these
matters, see~\cite{Cae:2004}, especially~\cite[Example~1.5 and Proposition~3.8]{Cae:2004}. If an $A$-coring is viewed as an internal category as in~\cite[Section~2.4]{Agu:1997}
(where corings are called \emph{coalgebroids}), then the isomorphism~\eqref{e:schneider-can} of corings becomes a close analogue of the familiar bijection
\[
G\times X \to X\times X, \qquad (g,x) \mapsto (g\cdot x, x)
\]
for $X$ a principal homogeneous space over a group $G$, which can be interpreted
as an isomorphism of small categories with object set $X$.

\subsection{Noncommutative Hilbert's Theorem 90}\label{ss:hilbert}

We turn to the ``classical'' case in which 
$A$ is a commutative $H$-comodule $\field$-algebra
 and $H=\field^G$ is the Hopf algebra of $\field$-valued functions on a finite group $G$. Then $H^*=\field G$ is the group $\field$-algebra of $G$ and the semilinearity condition~\eqref{e:semilinear} for a Doi-Koppinen $(A,H)$-Hopf module $M$ reduces to
\begin{equation}\label{e:semilinear-diag}
g\cdot am = (g\cdot a)(g\cdot m)
\end{equation}
for $g\in G$, $a\in A$, and $m\in M$. 
In particular, $A$ is simply a commutative algebra on which $G$ acts by automorphisms. The functor $\Bc$ of Theorem~\ref{ss:schneider-comodule} then maps the Hopf module $M$ to the $\field$-submodule $M^G$ of elements of $M$ left fixed by $G$.
In this context, we briefly review the
connection between such Hopf modules and the noncommutative \emph{Hilbert's Theorem 90}~\cite[Proposition~3]{Ser:1979}. For similar computations in a slightly
different context, see~\cite[Proposition~4]{Ser:1979}.

Let $N$ be an $(A,H)$-Hopf module, fixed throughout the discussion, and
$\Aut_A(N)$ denote the group of $A$-module automorphisms of $N$.
The group $G$ acts on $\Aut_A(N)$ according to the formula
\[
(g\cdot\alpha)(x) = g\cdot \alpha(g^{-1}\cdot x).
\]
for $g\in G$, $\alpha\in\Aut_A(N)$, and $x\in N$. Then, for all such $g$, $\alpha$, and $x$,
\[
g\cdot \alpha(x) = (g\cdot\alpha)(g\cdot x).
\]
The action is by automorphisms: 
\[
g\cdot \alpha\beta =(g\cdot\alpha)(g\cdot\beta)
\]
for $g\in G$, $\alpha,\beta\in\Aut_A(N)$.

A function $\varphi:G\to\Aut_A(N)$ is called a $1$-cocycle if
\begin{equation}\label{e:cocycle}
\varphi(fg) = \varphi(f) \bigl(f\cdot \varphi(g)\bigr)
\end{equation}
for all $f,g\in G$. We construct a category 
\[
\Zs^1\bigl(G,\Aut_A(N)\bigr)
\]
of which the objects are the $1$-cocycles just defined; a morphism $\alpha:\varphi\to\psi$ is an element $\alpha$ of $\Aut_A(N)$ such that
\[
\psi(g)(g\cdot\alpha) = \alpha \varphi(g)
\]
for all $g\in G$. If $\alpha:\varphi\to\psi$ and $\beta:\psi\to\rho$ are morphisms in 
$\Zs^1\bigl(G,\Aut_A(N)\bigr)$, then the composite $\beta\alpha$ is a morphism $\varphi\to\rho$;
this defines the composition in the category. The identity of $\varphi$ is the
identity automorphism of $N$. The category $\Zs^1\bigl(G,\Aut_A(N)\bigr)$ is a groupoid:
the inverse of $\alpha:\varphi\to\psi$ is $\alpha^{-1}:\psi\to\varphi$, where 
$\alpha^{-1}$ is the inverse automorphism of $N$.

We define 
\[
\Hs^1\bigl(G,\Aut_A(N)\bigr),
\]
the \emph{first cohomology set of $G$ with
coefficients in $\Aut_A(N)$}, to be the set of isomorphism classes of the
groupoid $\Zs^1\bigl(G,\Aut_A(N)\bigr)$. (For this definition in noncategorical
language, see~\cite[page~123]{Ser:1979}.) Two $1$-cocycles $\varphi$ and $\psi$
are called \emph{cohomologous} if they are isomorphic in  $\Zs^1\bigl(G,\Aut_A(N)\bigr)$.

We can now obtain, from a $1$-cocycle $\varphi$ as above, a new 
$(A,H)$-Hopf module $N_\varphi$ by keeping the $A$-module $N$ but twisting the action of $G$ on $N$; namely,
we define
\[
g\cdot_{\varphi} x := \varphi(g)(g\cdot x)
\]
for $g\in G$ and $x\in N$. Associativity for the twisted action follows from
the cocycle condition~\eqref{e:cocycle}, and since
\[
g\cdot_{\varphi} ax = \varphi(g)(g\cdot ax) = \varphi(g)\bigl((g\cdot a)(g\cdot x)\bigr)
= (g\cdot a) \varphi(g)(g\cdot x) =  (g\cdot a) (g\cdot_{\varphi} x),
\]
the semilinearity condition~\eqref{e:semilinear-diag} holds.

Let $\psi$ be another $1$-cocycle, and $\alpha:N_\varphi\to N_\psi$ be
an $A$-module isomorphism. Since $N_\varphi=N= N_\psi$ as $A$-modules,
$\alpha$ may be viewed as an element of $\Aut_A(N)$. A simple computation
shows that:
\begin{multline*}
\text{$\alpha:N_\varphi\to N_\psi$ is a morphism of Hopf modules}\\ 
\text{$\iff$
$\alpha:\varphi\to\psi$ is a morphism in $\Zs^1\bigl(G,\Aut_A(N)\bigr)$.}
\end{multline*}

Now let $M$ be any $(A,H)$-Hopf module which is isomorphic to $N$
as an $A$-module. Let $\theta:M\to N$ be such an isomorphism; then, for each $g\in G$, the mapping
\[
N\ni x \mapsto \theta\bigl( g\cdot \theta^{-1}(g^{-1}\cdot x) \bigr) \in N
\]
is an $A$-module automorphism of $N$ that we denote by $\Tilde{\theta}(g)$.
Thus we obtain a map $\Tilde{\theta}:G\to\Aut_A(N)$ such that
\[
\theta(g\cdot m) = \Tilde{\theta}(g)\bigl(g\cdot \theta(m)\bigr)
\]
for all $g\in G$ and $m\in M$. Another routine computation shows that $\Tilde{\theta}$
is a $1$-cocycle. The previous equation then implies that $\theta:M\to N_{\Tilde{\theta}}$ is an isomorphism of Hopf modules.

Consider the groupoid $\HAN$ whose objects are pairs $(M,\theta)$ as
in the preceding paragraph; a morphism $(M,\theta)\to(M',\theta')$ is an
isomorphism $\kappa:M\to M'$ of $(A,H)$-Hopf modules such that
\[
\xymatrix{
M \ar[rr]^-{\kappa} \ar[rd]_{\theta} & & M' \ar[ld]^{\theta'} \\
& N
}
\]
commutes. ($\HAN$ is a particular comma category.)

The preceding discussion contains the essential components of the
following result. 

\begin{proposition}\label{p:hilbert}
The functors
\[
\Zs^1\bigl(G,\Aut_A(N)\bigr) \to \HAN, \qquad
\varphi \mapsto (N_\varphi,\id_N)
\]
and
\[
\HAN \to \Zs^1\bigl(G,\Aut_A(N)\bigr), \qquad
(M,\theta) \mapsto \Tilde{\theta}
\]
form an equivalence of groupoids.
\end{proposition}

\begin{corollary}\label{c:hilbert}
The first cohomology set $\Hs^1\bigl(G,\Aut_A(N)\bigr)$ is in bijection
with the set of isomorphism classes of $(A,H)$-Hopf module structures
on the $A$-module $N$.
\end{corollary}

By the latter we mean an isomorphism class in the groupoid $\HAN$.

\smallskip

Now, with $H$ as above, assume that $A$ is a faithfully flat $\field$-module
and the Galois map~\eqref{e:schneider-can} is an isomorphism; then,
as noted in Section~\ref{ss:schneider-historical}, $A$ is a faithfully projective
$\field$-module and $B=\field$. Since $H=\field^G$, $A$ is then a Galois
extension of $\field$ with Galois group $G$, as in~\cite{CHR:1965}.
We then obtain from Theorem~\ref{t:schneider-conv} and Corollary~\ref{c:hilbert}
a bijection between $\Hs^1\bigl(G,\Aut_A(N)\bigr)$ and the set of isomorphism
classes of $\field$-modules $M$ with the property that $A\otimes M \cong N$ as $A$-modules. In particular, we may take $N=A\otimes M_0$ with $M_0$ a fixed $\field$-module; $N$ is then an $(A,H)$-Hopf module according to the formula
\[
g\cdot (a\otimes m) = (g\cdot a) \otimes m
\]
for $g\in G$, $a\in A$, and $m\in M$. If $M_0$ is free of rank $n$, then $\Aut_A(N)$
may be identified with the general linear group $\GL(n,A)$. In that case, we obtain
a bijection between $\Hs^1\bigl(G,\GL(n,A)\bigr)$ and the set of isomorphism
classes of projective $\field$-modules $P$ such that $A\otimes P$ is a free $A$-module of rank $n$. If $n=1$, it turns out that this bijection is actually an isomorphism of abelian groups, thus yielding an exact sequence
\[
1 \to \Hs^1(G,A^\times) \to \Pic(\field) \to \Pic(A)
\]
with $\Hs^1(G,A^\times)$ the usual first cohomology group of $G$ with
coefficients in the abelian group $A^\times$ of invertible elements in $A$, 
$\Pic(\field)$ the Picard group of isomorphism classes of invertible $\field$-modules,
and the right-most arrow in the sequence being induced by the functor $A\otimes(-)$.
This exact sequence, slightly modified, is the first part of the $7$-term exact sequence
of~\cite[Corollary~5.5]{CHR:1965}.

Of course, if $\field$ is a field, our discussion shows that $\Hs^1\bigl(G,\GL(n,A)\bigr)$
is trivial; this is the noncommutative Hilbert's Theorem 90.


\bibliographystyle{plain}  
\bibliography{fundhopf}

\begin{thebibliography}{10}

\bibitem{Agu:1997}
Marcelo Aguiar.
\newblock {\em Internal categories and quantum groups}.
\newblock ProQuest LLC, Ann Arbor, MI, 1997.
\newblock Thesis (Ph.D.)--Cornell University.

\bibitem{AguMah:2010}
Marcelo Aguiar and Swapneel Mahajan.
\newblock {\em Monoidal functors, species and {H}opf algebras}, volume~29 of
  {\em CRM Monograph Series}.
\newblock American Mathematical Society, Providence, RI, 2010.

\bibitem{BW:2005}
Michael Barr and Charles Wells.
\newblock Toposes, triples and theories.
\newblock {\em Repr. Theory Appl. Categ.}, (12):x+288, 2005.
\newblock Corrected reprint of the 1985 original.

\bibitem{Bor:1994i}
Francis Borceux.
\newblock {\em Handbook of categorical algebra. \textup1}, volume~50 of {\em
  Encyclopedia Math. Appl.}
\newblock Cambridge Univ. Press, Cambridge, 1994.

\bibitem{Bor:1994ii}
Francis Borceux.
\newblock {\em Handbook of categorical algebra. \textup2}, volume~51 of {\em
  Encyclopedia Math. Appl.}
\newblock Cambridge Univ. Press, Cambridge, 1994.

\bibitem{BLV:2011}
Alain Brugui{\`e}res, Steve Lack, and Alexis Virelizier.
\newblock Hopf monads on monoidal categories.
\newblock {\em Adv. Math.}, 227(2):745--800, 2011.

\bibitem{BruVir:2007}
Alain Brugui{\`e}res and Alexis Virelizier.
\newblock Hopf monads.
\newblock {\em Adv. Math.}, 215(2):679--733, 2007.

\bibitem{BW:2003}
Tomasz Brzezinski and Robert Wisbauer.
\newblock {\em Corings and comodules}, volume 309 of {\em London Mathematical
  Society Lecture Note Series}.
\newblock Cambridge University Press, Cambridge, 2003.

\bibitem{Bur:1973}
{\'E}lisabeth Burroni.
\newblock Lois distributives mixtes.
\newblock {\em C. R. Acad. Sci. Paris S\'er. A-B}, 276:A897--A900, 1973.

\bibitem{Cae:2004}
Stefaan Caenepeel.
\newblock Galois corings from the descent theory point of view.
\newblock In {\em Galois theory, {H}opf algebras, and semiabelian categories},
  volume~43 of {\em Fields Inst. Commun.}, pages 163--186. Amer. Math. Soc.,
  Providence, RI, 2004.

\bibitem{CHR:1965}
Stephen~U. Chase, David~K. Harrison, and Alex Rosenberg.
\newblock Galois theory and {G}alois cohomology of commutative rings.
\newblock {\em Mem. Amer. Math. Soc. No.}, 52:15--33, 1965.

\bibitem{CR:1965}
Stephen~U. Chase and Alex Rosenberg.
\newblock Amitsur cohomology and the {B}rauer group.
\newblock {\em Mem. Amer. Math. Soc. No.}, 52:34--79, 1965.

\bibitem{CS:1969}
Stephen~U. Chase and Moss~E. Sweedler.
\newblock {\em Hopf algebras and {G}alois theory}.
\newblock Lecture Notes in Mathematics, Vol. 97. Springer-Verlag, Berlin, 1969.

\bibitem{Doi:1983}
Yukio Doi.
\newblock On the structure of relative {H}opf modules.
\newblock {\em Comm. Algebra}, 11(3):243--255, 1983.

\bibitem{Doi:1992}
Yukio Doi.
\newblock Unifying {H}opf modules.
\newblock {\em J. Algebra}, 153(2):373--385, 1992.

\bibitem{FreMac:1971}
Armin Frei and John~L. MacDonald.
\newblock Algebras, coalgebras and cotripleability.
\newblock {\em Arch. Math. (Basel)}, 22:1--6, 1971.

\bibitem{JK:2001}
George Janelidze and Gregory~M. Kelly.
\newblock A note on actions of a monoidal category.
\newblock {\em Theory Appl. Categ.}, 9:61--91, 2001/02.
\newblock CT2000 Conference (Como).

\bibitem{Joh:1975}
Peter~T. Johnstone.
\newblock Adjoint lifting theorems for categories of algebras.
\newblock {\em Bull. London Math. Soc.}, 7(3):294--297, 1975.

\bibitem{Joh:2002}
Peter~T. Johnstone.
\newblock {\em Sketches of an elephant: a topos theory compendium. {V}ol. 1},
  volume~43 of {\em Oxford Logic Guides}.
\newblock The Clarendon Press Oxford University Press, New York, 2002.

\bibitem{Kel:1974}
Gregory~M. Kelly.
\newblock Doctrinal adjunction.
\newblock In {\em Category Seminar (Sydney, 1972/1973)}, volume 420 of {\em
  Lecture Notes in Math.}, pages 257--280. Springer, Berlin, 1974.

\bibitem{KelStr:1974}
Gregory~M. Kelly and Ross Street.
\newblock Review of the elements of {$2$}-categories.
\newblock In {\em Category Seminar (Sydney, 1972/1973)}, volume 420 of {\em
  Lecture Notes in Math.}, pages 75--103. Springer, Berlin, 1974.

\bibitem{Kop:1995}
Markku Koppinen.
\newblock Variations on the smash product with applications to group-graded
  rings.
\newblock {\em J. Pure Appl. Algebra}, 104(1):61--80, 1995.

\bibitem{KT:1981}
Herbert~F. Kreimer and Mitsuhiro Takeuchi.
\newblock Hopf algebras and {G}alois extensions of an algebra.
\newblock {\em Indiana Univ. Math. J.}, 30(5):675--692, 1981.

\bibitem{Lei:2004}
Tom Leinster.
\newblock {\em Higher operads\textup, higher categories}, volume 298 of {\em
  London Math. Soc. Lecture Note Ser.}
\newblock Cambridge Univ. Press, Cambridge, 2004.

\bibitem{Lin:1969}
Fred E.~J. Linton.
\newblock Coequalizers in categories of algebras.
\newblock In {\em Sem. on {T}riples and {C}ategorical {H}omology {T}heory
  ({ETH}, {Z}\"urich, 1966/67)}, pages 75--90. Springer, Berlin, 1969.

\bibitem{Lop:2009}
Ignacio~L. L{\'o}pez~Franco.
\newblock Formal {H}opf algebra theory. {I}. {H}opf modules for pseudomonoids.
\newblock {\em J. Pure Appl. Algebra}, 213(6):1046--1063, 2009.

\bibitem{Lyu:1995}
Volodymyr Lyubashenko.
\newblock Modular transformations for tensor categories.
\newblock {\em J. Pure Appl. Algebra}, 98(3):279--327, 1995.

\bibitem{Mac:1998}
Saunders Mac~Lane.
\newblock {\em Categories for the working mathematician}, volume~5 of {\em
  Grad. Texts in Math.}
\newblock Springer, New York, 2nd edition, 1998.

\bibitem{MS:2004}
John MacDonald and Manuela Sobral.
\newblock Aspects of monads.
\newblock In {\em Categorical foundations}, volume~97 of {\em Encyclopedia
  Math. Appl.}, pages 213--268. Cambridge Univ. Press, Cambridge, 2004.

\bibitem{McC:2002}
Paddy McCrudden.
\newblock Opmonoidal monads.
\newblock {\em Theory Appl. Categ.}, 10:No. 19, 469--485, 2002.

\bibitem{MesWis:2011}
Bachuki Mesablishvili and Robert Wisbauer.
\newblock Bimonads and {H}opf monads on categories.
\newblock {\em J. K-Theory}, 7(2):349--388, 2011.

\bibitem{Moe:2002}
Ieke Moerdijk.
\newblock Monads on tensor categories.
\newblock {\em J. Pure Appl. Algebra}, 168(2-3):189--208, 2002.
\newblock Category theory 1999 (Coimbra).

\bibitem{Mon:1993}
Susan Montgomery.
\newblock {\em Hopf algebras and their actions on rings}, volume~82 of {\em
  CBMS Regional Conference Series in Mathematics}.
\newblock Published for the Conference Board of the Mathematical Sciences,
  Washington, DC, 1993.

\bibitem{Mon:2009}
Susan Montgomery.
\newblock Hopf {G}alois theory: a survey.
\newblock In {\em New topological contexts for {G}alois theory and algebraic
  geometry ({BIRS} 2008)}, volume~16 of {\em Geom. Topol. Monogr.}, pages
  367--400. Geom. Topol. Publ., Coventry, 2009.

\bibitem{Sch:2004}
Peter Schauenburg.
\newblock Hopf-{G}alois and bi-{G}alois extensions.
\newblock In {\em Galois theory, {H}opf algebras, and semiabelian categories},
  volume~43 of {\em Fields Inst. Commun.}, pages 469--515. Amer. Math. Soc.,
  Providence, RI, 2004.

\bibitem{SS:2005}
Peter Schauenburg and Hans-J{\"u}rgen Schneider.
\newblock On generalized {H}opf {G}alois extensions.
\newblock {\em J. Pure Appl. Algebra}, 202(1-3):168--194, 2005.

\bibitem{Sch:1990}
Hans-J{\"u}rgen Schneider.
\newblock Principal homogeneous spaces for arbitrary {H}opf algebras.
\newblock {\em Israel J. Math.}, 72(1-2):167--195, 1990.

\bibitem{Ser:1979}
Jean-Pierre Serre.
\newblock {\em Local fields}, volume~67 of {\em Graduate Texts in Mathematics}.
\newblock Springer-Verlag, New York, 1979.
\newblock Translated from the French by Marvin Jay Greenberg.

\bibitem{Str:1972}
Ross Street.
\newblock The formal theory of monads.
\newblock {\em J. Pure Appl. Algebra}, 2(2):149--168, 1972.

\bibitem{Swe:69}
Moss~E. Sweedler.
\newblock {\em Hopf algebras}.
\newblock Math. Lecture Note Ser. W. A. Benjamin, Inc., New York, 1969.

\bibitem{Szl:2003}
Korn{\'e}l Szlach{\'a}nyi.
\newblock The monoidal {E}ilenberg-{M}oore construction and bialgebroids.
\newblock {\em J. Pure Appl. Algebra}, 182(2-3):287--315, 2003.

\bibitem{Tak:1979}
Mitsuhiro Takeuchi.
\newblock Relative {H}opf modules---equivalences and freeness criteria.
\newblock {\em J. Algebra}, 60(2):452--471, 1979.

\bibitem{Tak:1999}
Mitsuhiro Takeuchi.
\newblock Finite {H}opf algebras in braided tensor categories.
\newblock {\em J. Pure Appl. Algebra}, 138(1):59--82, 1999.

\bibitem{Wol:1973}
Harvey Wolff.
\newblock {$V$}-localizations and {$V$}-monads.
\newblock {\em J. Algebra}, 24:405--438, 1973.

\end{thebibliography}

\end{document}